\documentclass[11pt,reqno]{amsart}

\usepackage[T1]{fontenc}
\usepackage{lmodern}
\usepackage{amsmath,amssymb,mathtools,mathrsfs}
\usepackage{microtype}
\usepackage[colorlinks=true,linkcolor=blue,citecolor=blue,urlcolor=blue]{hyperref}
\usepackage{tikz-cd}

\numberwithin{equation}{section}
\allowdisplaybreaks

\newtheorem{theorem}{Theorem}[section]
\newtheorem{proposition}[theorem]{Proposition}
\newtheorem{lemma}[theorem]{Lemma}
\newtheorem{corollary}[theorem]{Corollary}
\theoremstyle{definition}

\theoremstyle{remark}
\newtheorem{remark}[theorem]{Remark}
\newtheorem{assumption}[theorem]{Assumption}

\DeclareMathOperator{\Ext}{Ext}
\DeclareMathOperator{\Area}{Area}

\hypersetup{
  pdftitle={Quasi-convexity of energy functions along Teichm\"uller geodesics},
  pdfauthor={Inkang Kim, Xueyuan Wan, Genkai Zhang},
  pdfsubject={Teichmuller theory, harmonic maps, energy functions},
  pdfkeywords={Teichmuller geodesics, harmonic maps, harmonic-map energy, hyperbolic length, quasi-convexity, extremal length, Weil-Petersson geometry}
}

\title[Quasi-convexity of energy functions]
{Quasi-convexity of energy functions along Teichm\"uller geodesics}

\author{Inkang Kim}
\address{School of Mathematics, Korea Institute for Advanced Study (KIAS),
85 Hoegiro, Dongdaemun-gu, Seoul 02455, Republic of Korea}
\email{inkang@kias.re.kr}

\author{Xueyuan Wan}
\address{Mathematical Science Research Center, Chongqing University of Technology,
Chongqing 400054, China}
\email{xwan@cqut.edu.cn}

\author{Genkai Zhang}
\address{Mathematical Sciences, Chalmers University of Technology, University of Gothenburg, 41296
Gothenburg, Sweden}
\email{genkai@chalmers.se}

\date{}

\makeatletter
\@namedef{subjclassname@2020}{\textup{2020} Mathematics Subject Classification}
\makeatother

\begin{document}


\maketitle

\begin{center}
\small\emph{In memory of Misha Kapovich}
\end{center}

\begin{abstract}
Hyperbolic length functions are among the most fundamental ones on
Teichm\"uller space,  and they are quasi-convex along Teichm\"uller geodesics.
In this  paper we investigate the same question for 
energy functions of harmonic maps
in two natural settings, which may be viewed as nonlinear
and higher-dimensional analogs of the length functions.
For a fixed domain and a
varying hyperbolic target, we prove the energy is quasi-convex along Teichm\"uller
geodesics under a filling hypothesis. Furthermore
we generalize Masur's result
on asymptotic growth of the length function along
the Teichm\"u{}ller geodesic determined by
a Jenkins-Strebel differential
to the energy functions.
 We also prove the quasi-convexity
for
covering maps between closed hyperbolic surfaces with fixed target
and varying domains.
We derive first and second
variation formulas of energy functions along Teichm\"uller geodesics and explain why the natural
global statement is quasi-convexity rather than genuine convexity.
\end{abstract}

\section{Introduction}
\label{sec:introduction}

The  geometry of Teichm\"u{}ller space
of a surface $S$ is
mostly encoded by properties
of some naturally defined quantities of $S$.  Among the most classic examples are hyperbolic
length functions
\(
  X\mapsto \ell_X(\alpha),
\)
where \(\alpha\) is an essential closed curve on the underlying surface $S$.
These functions are indispensable in both the analytic and coarse geometry of
Teichm\"uller space.  Along Weil--Petersson geodesics there is a rich
convexity theory, while along Teichm\"uller geodesics genuine convexity is too
rigid to be expected in general. One replacement for convexity, discovered by
Lenzhen and Rafi, is quasi-convexity: for every simple closed curve
\(\alpha\), both the hyperbolic length \(\ell_X(\alpha)\) and the extremal
length \(\operatorname{Ext}_X(\alpha)\) are quasi-convex along Teichm\"uller
geodesics \cite{LenzhenRafi2011}.  Their theorem also implies that
Teichm\"uller balls are uniformly quasi-convex.

The purpose of this paper is to investigate how far this phenomenon extends
from length functions to harmonic-map energy functions.  This is a natural
question, because length squares are
one-dimensional examples of energy functions. Indeed, if \(S_L^1\) is a circle of length \(L\) and
\(u_X\colon S_L^1\to X\) is the constant-speed geodesic representative of the
free homotopy class \(\alpha\), then, with the normalization used throughout
this paper,
\(
  E_X(u_X)
  =
  (1/2L)\ell_X(\alpha)^2.
\)
We would like to
study the quasi-convexity along
a Teichm\"uller geodesic for energy functions
for general harmonic maps.

The answer obtained here is affirmative under natural topological hypotheses.
However, the proof of the result is not a direct second-variation argument.
For Weil--Petersson geodesics, energy convexity is closely related
to the positivity of the second variation of the target hyperbolic metric; see
Yamada's theorem and the later refinements in
\cite{Yamada1999,KimWanZhang2022}.  Along Teichm\"uller geodesics, however,
the second variation contains additional terms and a Jacobi
relaxation term coming from the fact that the harmonic representative itself
moves.  These terms have no uniform sign.
We prove  that the pointwise convexity
does not hold and that quasi-convexity is a coarse global phenomenon: 
it is proved here by
comparison with length functions and extremal lengths.

Throughout the paper, a positive function \(F\) on Teichm\"uller space is
called \emph{multiplicatively \(K\)-quasi-convex along Teichm\"uller
geodesics} if
\begin{equation}
\label{eq:def-multiplicative-quasi-convexity}
  F(\gamma(t))\le K\max\{F(\gamma(s)),F(\gamma(r))\}
  \qquad (s<t<r) 
\end{equation}
for every Teichm\"uller geodesic \(\gamma\). Equivalently, \(\log F\) is
additively quasi-convex, with additive constant \(\log K\).

We also study the asymptotic growth of the energy function along the Teichm\"uller geodesic. For the precise description of the limit, we need to use \(\mathbb R\)-tree action, which is the limit of scaled metrics on the hyperbolic plane.
This action is already manifested in the study of limit representations in Kleinian group actions \cite{Kapovich, Otal, Kim-Lecuire-Ohshika} and harmonic maps \cite{DaskalopoulosWentworth2007,KorevaarSchoen1997, Wolf1996}, etc.

We consider two complementary families of energy functions.

\medskip
\noindent\textbf{Energy functions with varying target.}
Let \((M,g)\) be a closed, connected Riemannian manifold,  \(S\)  a closed oriented
surface of genus at least \(2\), and 
\(
  u_0\colon M\to S
\)
 a continuous map.
For each \(X\in\mathcal T(S)\), let
\(\sigma_X\) denote the hyperbolic metric in the conformal class \(X\), and let
\(
  u_X\colon (M,g)\to (S,\sigma_X)
\)
be the harmonic map homotopic to \(u_0\).  The associated energy function is
\begin{equation}
  \mathcal E_{u_0}(X)
  :=
  \frac12\int_M |du_X|^2
  \,d\mu_g.
  \label{eq:intro-target-energy}
\end{equation}
Here and in the text below
$|du_X|^2=|du_X|_{\mathrm{HS}}^2$
denotes the Hilbert-Schmidt norm
of $du_X$ as a section of $\text{End}(TM, TS)$.
The existence and uniqueness theory for such harmonic representatives is
classical under the non-elementary hypotheses used below
\cite{EellsSampson1964,Hartman1967,Sampson1978}.
Since the resulting minimum-energy function depends only
on the homotopy class, we shall fix this smooth representative $u_X$.
Energy functions of this
kind with varying targets have played a central role in the analytic approach to Teichm\"uller
theory, notably in the work of Tromba, Wolf, Minsky, Yamada, and others
\cite{Tromba1992,Wolf1989,Minsky1992,Yamada1999,DaskalopoulosWentworth2007}.

We shall say that \(u_0\) satisfies the \emph{filling hypothesis} if
there exists a finite collection of essential simple closed curves
\(
  \Gamma=\{\alpha_1,\ldots,\alpha_N\}
\)
which can be realized in minimal position with connected union and disk
complementary components, and, for each \(j\), there is a fixed piecewise
smooth closed curve \(\beta_j\) in \(M\) such that \(u_0\circ\beta_j\) is
freely homotopic to \(\alpha_j\). We regard
\(
  \mathfrak F=(\Gamma,\{\beta_j\}_{j=1}^N)
\)
as part of the filling data. This formulation is independent of basepoints
and is the one used in the proof below. The filling hypothesis automatically
forces \((u_0)_*\pi_1(M)\) to be nontrivial and noncyclic. Indeed, an
essential simple closed curve on \(S\) is not a proper power. If the image
were cyclic, all simple conjugacy classes represented by the loops
\(u_0\circ\beta_j\) would therefore coincide with the same primitive
generator, up to inversion and conjugacy, which is incompatible with the fact
that \(\Gamma\) fills \(S\). Thus the uniqueness and nondegeneracy hypotheses
used below are automatic under the filling hypothesis. A useful sufficient
condition is that
\(u_{0*}\colon\pi_1(M)\twoheadrightarrow\pi_1(S)\) be surjective.

\begin{theorem}
\label{thm:intro-varying-target-quasiconvexity}
Assume that \(u_0\) satisfies the filling hypothesis.  Then there exists
\(K_E\ge1\), depending only on \(M,g,u_0,\mathfrak F\), and \(S\), such
that, for every Teichm\"uller geodesic
\(\gamma\colon I\to\mathcal T(S)\) and every
\(s<t<r\) in \(I\),
\begin{equation}
  \mathcal E_{u_0}(\gamma(t))
  \le
  K_E
  \max\bigl\{
    \mathcal E_{u_0}(\gamma(s)),
    \mathcal E_{u_0}(\gamma(r))
  \bigr\}.
  \label{eq:intro-target-energy-quasiconvexity}
\end{equation}
Equivalently, \(\log\mathcal E_{u_0}\) is additively quasi-convex along
Teichm\"uller geodesics.
\end{theorem}

The proof of Theorem~\ref{thm:intro-varying-target-quasiconvexity} reduces the
energy to the length functions studied by Lenzhen and Rafi.  More precisely,
under the filling hypothesis we prove that there are constants
\(0<c_0\le C_0<\infty\), independent of \(X\), such that
\begin{equation}
  c_0\sum_{j=1}^N\ell_X(\alpha_j)^2
  \le
  \mathcal E_{u_0}(X)
  \le
  C_0\sum_{j=1}^N\ell_X(\alpha_j)^2.
  \label{eq:intro-energy-length-comparison}
\end{equation}
The lower bound follows by applying a Bochner--mean-value estimate to fixed
loops in \(M\) representing the filling curves under \(u_0\), while the upper
bound is obtained by constructing explicit competitor maps from a filling
geodesic graph.  Applying the Lenzhen--Rafi
quasi-convexity theorem to each \(\ell_X(\alpha_j)\) then gives
\eqref{eq:intro-target-energy-quasiconvexity}.

A classical result of Masur gives the precise asymptotic growth of
hyperbolic length functions along Jenkins--Strebel rays. Let
\(q\in\mathcal Q^1(X_0)\) be a Jenkins--Strebel quadratic differential of unit area,
let \(X_t=\gamma_q(t)\) be the associated unit-speed Teichm\"uller ray,
normalized by \(d_{\mathrm T}(X_0,X_t)=t\), and let
\(\delta_1,\ldots,\delta_r\) be the core curves of the maximal cylinders
in the contracting foliation. Set
\(\Delta_q:=\delta_1+\cdots+\delta_r\). Then, for every essential simple
closed curve \(\alpha\),
\begin{equation}
\label{eq:introduction-masur-length-asymptotic}
  \lim_{t\to\infty}
  \frac{\ell_{X_t}(\alpha)}{4t}
  =
  i(\Delta_q,\alpha)
  =
  \sum_{j=1}^{r}i(\delta_j,\alpha).
\end{equation}
See \cite[Theorem~1.1]{Masur1982} and the final estimates in the proof.
Thus the leading linear growth of \(\ell_{X_t}(\alpha)\) is determined
entirely by the intersection of \(\alpha\) with the cylinder multicurve
\(\Delta_q\).

Let \(T_{\Delta_q}\) be the dual \(\mathbb R\)-tree, normalized by
\(\ell_{T_{\Delta_q}}(\gamma)=i(\Delta_q,\gamma)\) for
\(\gamma\in\pi_1(S)\). Given a continuous map \(u_0:M\to S\), let
\(\pi_1(M)\) act on \(T_{\Delta_q}\) through
\((u_0)_*:\pi_1(M)\to\pi_1(S)\), and define
\[
  \mathscr E_{\Delta_q}(u_0)
  :=
  \inf_U
  \frac12\int_D |dU|^2\,d\mu_g,
\]
where \(D\subset\widetilde M\) is a fundamental domain and the infimum is
taken over all equivariant
\(W_{\mathrm{loc}}^{1,2}\)-maps
\(U:\widetilde M\to T_{\Delta_q}\).

\begin{theorem}
\label{Intro-thm:energy-masur-asymptotic}
Let \((M,g)\) be a closed Riemannian manifold, let \(u_0:M\to S\) satisfy
the filling hypothesis, and let
\(q\in\mathcal Q^1(X_0)\) be Jenkins--Strebel. With the notation above,
\begin{equation}
\label{eq:introduction-energy-masur-asymptotic}
  \lim_{t\to\infty}
  \frac{\mathcal E_{u_0}(X_t)}{t^2}
  =
  16\,\mathscr E_{\Delta_q}(u_0).
\end{equation}
In particular, the limit exists and is strictly positive.
\end{theorem}

Theorem~\ref{Intro-thm:energy-masur-asymptotic} may be viewed as a nonlinear
and higher-dimensional extension of Masur's formula \eqref{eq:introduction-masur-length-asymptotic}. Hyperbolic
length is replaced by harmonic-map energy, while the
intersection number with \(\Delta_q\) is replaced by the equivariant
energy into the dual tree. The compatibility is exact
when the domain is one-dimensional; see Remark \ref{rem:circle-energy-masur-compatibility}.

In the case of a covering map between closed surfaces, the
coefficient in
\eqref{eq:introduction-energy-masur-asymptotic} admits a more explicit
description.

\begin{corollary}
\label{cor:covering-energy-masur-asymptotic}
Let \(u_0:\Sigma\to S\) be an orientation-preserving covering map
between closed surfaces, and assume that \(u_0\) satisfies the filling
hypothesis. Let \(X_g\in\mathcal T(\Sigma)\) be the conformal structure
determined by the fixed domain metric \(g\), and regard \(u_0^*\Delta_q\) as
the pullback measured foliation on \(\Sigma\). Then
\begin{equation}
\label{eq:introduction-covering-tree-energy}
  \mathscr E_{\Delta_q}(u_0)
  =
  \frac12\,
  \operatorname{Ext}_{X_g}\!\left(u_0^*\Delta_q\right),
\end{equation}
and consequently
\begin{equation}
\label{eq:introduction-covering-energy-asymptotic}
  \lim_{t\to\infty}
  \frac{\mathcal E_{u_0}(X_t)}{t^2}
  =
  8\,\operatorname{Ext}_{X_g}\!\left(u_0^*\Delta_q\right).
\end{equation}

If moreover, \(u_0\) has degree \(d\) and \(X_g=u_0^*X_0\), then
\begin{equation}
\label{eq:introduction-covering-energy-degree-asymptotic}
  \lim_{t\to\infty}
  \frac{\mathcal E_{u_0}(X_t)}{t^2}
  =
  8d\,\operatorname{Ext}_{X_0}(\Delta_q).
\end{equation}
Equivalently, let
\(\Phi_{\Delta_q,X_0}\) be the Hubbard--Masur differential whose vertical
measured foliation is \(\Delta_q\), and let \(M_i\) be the conformal modulus
of its characteristic cylinder with core curve \(\delta_i\). Then
\begin{equation}
\label{eq:introduction-covering-energy-moduli-asymptotic}
  \lim_{t\to\infty}
  \frac{\mathcal E_{u_0}(X_t)}{t^2}
  =
{8d}  \sum_{i=1}^{r}\frac{1}
  {M_i}.
\end{equation}

\end{corollary}

\medskip
\noindent\textbf{Energy functions with varying domain.}
We also consider a dual situation in which the target is fixed and the domain
Riemann surface varies.  Let
\(
  u_0\colon \Sigma\to (S,h)
\)
be an orientation-preserving covering map between closed hyperbolic
surfaces.  For \(X\in\mathcal T(\Sigma)\), let
\(
  u_X\colon X\to (S,h)
\)
be the harmonic map homotopic to \(u_0\), and define
\begin{equation}
  E_{u_0}(X)
  :=
  \frac12\int_X |du_X|^2\,dA_X.
  \label{eq:intro-domain-energy}
\end{equation}
Set
\(
  Y=(\Sigma,u_0^*h)\in\mathcal T(\Sigma),
  \) \(
  A_Y=\operatorname{Area}(Y).
\)
Then \(u_0\colon Y\to(S,h)\) is a local isometry.  The covering structure
implies that \(E_{u_0}\) factors through the energy of the harmonic
diffeomorphism \(X\to Y\).  The unique critical point and the positive
Weil--Petersson Hessian at that point were established in the covering-map
setting by Kim--Wan--Zhang \cite{KimWanZhang2024}.  There is also a broader
complex-analytic theory for varying-domain harmonic-map energy with general
nonpositively curved targets: Toledo proved plurisubharmonicity of the energy
for targets of non-positive Hermitian sectional curvature, and
Kim--Wan--Zhang proved plurisuperharmonicity of the reciprocal energy together
with plurisubharmonicity of \(\log E\) and \(E\)
\cite{Toledo2012,KimWanZhang2020}.  Our coarse quasi-convexity theorem below,
however, uses the special geometry of a covering of hyperbolic
surfaces.  Here we prove the following global coarse strengthening along
Teichm\"uller geodesics.

\begin{theorem}
\label{thm:intro-varying-domain-quasiconvexity}
Let \(u_0\colon\Sigma\to(S,h)\) be an orientation-preserving 
covering map and let
\(Y=(\Sigma,u_0^*h)\).  Then there are constants depending only on \(Y\) such
that
\begin{equation}
  \frac{\eta_Y}{2}e^{2d_{\mathrm T}(X,Y)}
  \le
  E_{u_0}(X)
  \le
  A_Ye^{2d_{\mathrm T}(X,Y)}
  \label{eq:intro-domain-energy-distance-comparison}
\end{equation}
for every \(X\in\mathcal T(\Sigma)\), where
\(
  \eta_Y
  :=
  \min_{[\lambda]\in\mathbb P\mathcal{ML}(\Sigma)}
  \frac{\ell_Y(\lambda)^2}{\operatorname{Ext}_Y(\lambda)}
  >0.
\)

Consequently, there exists \(K_E\ge1\) such that, whenever
\(X_a,X_b,X_c\) occur in this order on a Teichm\"uller geodesic in
\(\mathcal T(\Sigma)\),
\begin{equation}
  E_{u_0}(X_b)
  \le
  K_E\max\{E_{u_0}(X_a),E_{u_0}(X_c)\}.
  \label{eq:intro-domain-energy-quasiconvexity}
\end{equation}
In particular,
\begin{equation}
  \frac12\log E_{u_0}(X)
  =
  d_{\mathrm T}(X,Y)+O_Y(1).
  \label{eq:intro-half-log-energy-distance}
\end{equation}
\end{theorem}

The lower bound in \eqref{eq:intro-domain-energy-distance-comparison} is based
on extremal lengths of measured laminations.  For every measured lamination \(\lambda\), the pullback
pseudometric associated to the harmonic map \(X\to Y\) gives
\(
  \ell_Y(\lambda)^2
  \le
  2E_{u_0}(X)\operatorname{Ext}_X(\lambda),
\)
and Kerckhoff's formula then converts this estimate into the exponential lower
bound in \eqref{eq:intro-domain-energy-distance-comparison}
\cite{Kerckhoff1980}.  The upper bound follows by using the minimizing map at
one point of a Teichm\"uller geodesic as a competitor at another point.
Finally, the quasi-convexity of
Teichm\"uller balls, again due to Lenzhen and Rafi \cite{LenzhenRafi2011},
turns the two-sided comparison with \(e^{2d_{\mathrm T}(X,Y)}\) into
\eqref{eq:intro-domain-energy-quasiconvexity}. 

The coarse comparison with Teichm\"uller distance determines the exponential
growth rate of the covering-map energy, but it does not identify its leading
coefficient. The fixed-domain expression for the energy yields a sharper
asymptotic statement.

Let
\(
  u_0:\Sigma\to(S,h)
\)
be an orientation-preserving  covering map, and set
\(
  Y=(\Sigma,u_0^*h)\in\mathcal T(\Sigma).
\)
Let
\(
  X_t=\gamma_q(t)
\)
be the unit-speed Teichm\"uller ray determined by
\(q\in\mathcal Q^1(X_0)\). In a \(q\)-natural coordinate \(z=x+iy\), define the
vertical energy of \(v\in\mathscr H_{\mathrm{id}}\) by
\[
  V_q(v)
  :=
  \frac12
  \int_{X_0\setminus Z(q)}
  |v_y|_{u_0^*h}^2\,dx\,dy,
\]
and set
\[
  \mathscr V_q(Y)
  :=
  \inf_{v\in\mathscr H_{\mathrm{id}}}V_q(v).
\]
Here \(\mathscr H_{\mathrm{id}}\) denotes the marking-preserving homotopy class
of Lipschitz maps \(v:X_0\to Y\). 

\begin{proposition}
\label{prop:intro-covering-energy-leading-coefficient}
With the notation above, the normalized energy is nonincreasing and satisfies
\begin{equation}
  e^{-2t}E_{u_0}(X_t)
  \searrow
  \mathscr V_q(Y)
  \qquad
  \text{as }t\to\infty.
  \label{eq:intro-normalized-covering-energy-limit}
\end{equation}
In particular,
\begin{equation}
  \lim_{t\to\infty}
  \frac{E_{u_0}(X_t)}{e^{2t}}
  =
  \mathscr V_q(Y),
  \qquad
  E_{u_0}(X_t)
  =
  \mathscr V_q(Y)e^{2t}
  +
  o(e^{2t}).
  \label{eq:intro-covering-energy-asymptotic}
\end{equation}
Moreover,
\(
  0<\mathscr V_q(Y)<\infty.
\)

If \(q\) is Jenkins--Strebel and its vertical foliation is the contracting
foliation of the ray, let
\(
  C_1,\ldots,C_r
\)
be its maximal flat cylinders. Write
\(
  C_i
  \cong
  (0,h_i)\times
  \bigl(\mathbb R/c_i\mathbb Z\bigr),
\)
where the closed vertical trajectories have \(q\)-length \(c_i\), the
transverse height is \(h_i\), and \(\delta_i\) is the core curve of \(C_i\).
Then
\begin{equation}
  \mathscr V_q(Y)
  =
  \frac12
  \sum_{i=1}^r
  \operatorname{Mod}(C_i)\,
  \ell_Y(\delta_i)^2
  =
  \frac12
  \sum_{i=1}^r
  \frac{h_i}{c_i}\,
  \ell_Y(\delta_i)^2.
  \label{eq:intro-strebel-covering-energy-coefficient}
\end{equation}
Consequently,
\begin{equation}
  \lim_{t\to\infty}
  \frac{E_{u_0}(X_t)}{e^{2t}}
  =
  \frac12
  \sum_{i=1}^r
  \operatorname{Mod}(C_i)\,
  \ell_Y(\delta_i)^2.
  \label{eq:intro-strebel-covering-energy-limit}
\end{equation}
\end{proposition}

\medskip
\noindent\textbf{Structure of the paper.}
Section~\ref{sec:teichmuller-space-and-geodesics} recalls the necessary
background on Teichm\"uller space, Teichm\"uller geodesics, energy functions, and the local expansions of hyperbolic metrics along Teichm\"uller geodesics.  Section~\ref{subsec:quasi-convexity-energy-functions} treats
energy functions with fixed domain and varying target and proves
Theorem~\ref{thm:intro-varying-target-quasiconvexity}.
Section~\ref{subsec:quasi-convexity-energy-varying-domain} treats the covering-map case with
varying domain and proves Theorem~\ref{thm:intro-varying-domain-quasiconvexity}.
In Section~\ref{Energy asymptotics} we study the asymptotic behavior of energy functions along a Teichm\"uller geodesic associated with a Jenkins-Strebel differential, and prove Theorem~\ref{Intro-thm:energy-masur-asymptotic}, Corollary~\ref{cor:covering-energy-masur-asymptotic}, and Proposition \ref{prop:intro-covering-energy-leading-coefficient}.
In Section \ref{Variations of energy functions} we calculate the variations of energy functions in the two settings varying domains or targets.

\medskip
\noindent\textbf{Acknowledgments.}
We would like to thank the Banff International Research Station for Mathematical
Innovation and Discovery (BIRS) for hosting the Research in Teams program
``Higher Teichm\"uller theory and harmonic maps'' from July 26 to August 9,
2026, and for providing an excellent environment for
collaboration. Research by Inkang Kim was partially supported by
RS-2026-25468457 and KIAS Individual Grant (MG031408), Xueyuan Wan was
supported by the National Key R\&D Program of China (Grant No.2024YFA1013200) and the National Natural Science Foundation of China (Grant No. 12671100), 
and Genkai Zhang was  supported partially by the Swedish Research
Council VR 11253580.
\vspace{3mm}

\section{Teichm\"uller space and Teichm\"uller geodesics}
\label{sec:teichmuller-space-and-geodesics}

Throughout this paper $S$ will be a closed, connected, oriented smooth
surface of genus $g\ge 2$. All complex structures on $S$ are assumed to be
compatible with the orientation. We use the normalization
\( d_{\mathrm T}=\frac12\log K \)
for the Teichm\"uller metric \cite{Masur2009}.
We shall recall some well-known
facts on Teichm\"uller space;
see e.g. 
\cite{Ahlfors2006,GardinerLakic2000,Hubbard2006}
for general results and
\cite{Strebel1984,Masur2009}
for quadratic
differentials and the Teichm\"uller geodesic flow.
An account emphasizing the relation with
harmonic maps can be found in \cite{DaskalopoulosWentworth2007}.

\subsection{Teichm\"uller space}
\label{subsec:teichmuller-space}

A \emph{marked Riemann surface} of topological type $S$ is a pair $(X,f)$,
where $X$ is a Riemann surface and
\( f\colon S\to X \)
is an orientation-preserving diffeomorphism, called a \emph{marking}.
Two marked Riemann surfaces $(X,f)$ and $(Y,g)$ are said to be equivalent
if there exists a biholomorphism $h\colon X\to Y$ such that $h\circ f$ is
isotopic to $g$. The Teichm\"uller space of $S$ is defined by
\[
  \mathcal T(S)
  :=
  \left\{(X,f):
  \begin{array}{l}
  X \text{ is a Riemann surface,}\\
  f\colon S\to X \text{ is a marking}
  \end{array}
  \right\}\big/\!\sim .
\]
We denote the equivalence class of $(X,f)$ by $[X,f]$ and usually suppress
the marking when no confusion can arise.

By the uniformization theorem, each conformal structure on $S$ contains a
unique hyperbolic metric of curvature $-1$. Thus, $\mathcal T(S)$ may
equivalently be viewed as the space of marked hyperbolic metrics on $S$,
where two such metrics are identified if they are related by an isometry
isotopic to the identity. The space $\mathcal T(S)$ is a contractible
complex manifold of complex dimension $3g-3$ and real dimension $6g-6$.

Let $h\colon X\to Y$ be an orientation-preserving quasiconformal map. In
local conformal coordinates, its Beltrami coefficient is
\[ \mu_h
  :=
  \frac{h_{\bar z}}{h_z}, \quad
  \|\mu_h\|_{L^\infty(X)}<1, \]
and its maximal dilatation is
\[ K(h)
  :=
  \operatorname*{ess\,sup}_{z\in X}
  \frac{|h_z|+|h_{\bar z}|}{|h_z|-|h_{\bar z}|}
  =
  \frac{1+\|\mu_h\|_{L^\infty(X)}}
       {1-\|\mu_h\|_{L^\infty(X)}}. \]
For $x=[X,f]$ and $y=[Y,g]$ in $\mathcal T(S)$, the Teichm\"uller distance
is defined by
\begin{equation}
  d_{\mathrm T}(x,y)
  :=
  \frac12\inf_h\log K(h),
  \label{eq:teichmuller-distance}
\end{equation}
where the infimum is taken over all quasiconformal maps
$h\colon X\to Y$ satisfying
\( h\circ f\simeq_{\text{iso}} g. \)
Here $\simeq_{\text{iso}}$ denotes isotopy. The function
$d_{\mathrm T}$ defines a complete Finsler metric on $\mathcal T(S)$.

For a Riemann surface $X$, let
\( \mathcal Q(X)
  :=
  H^0(X,K_X^{\otimes 2}) \)
denote the vector space of holomorphic quadratic differentials on $X$,
where $K_X$ is the canonical bundle of $X$. If
\( q=\phi(z)\,dz^2 \)
in a local conformal coordinate $z=x+iy$, then $q$ determines the singular
flat metric
\( ds_q^2
  :=
  |q|
  =
  |\phi(z)|\,|dz|^2 \)
and the associated area form
\( dA_q
  :=
  |\phi(z)|\,dx\,dy. \)
We write
\begin{equation}
  \|q\|_1
  :=
  \int_X dA_q,
  \quad
  \mathcal Q^1(X)
  :=
  \bigl\{q\in\mathcal Q(X):\|q\|_1=1\bigr\}.
  \label{eq:L1-norm-quadratic-differential}
\end{equation}
By the Riemann--Roch theorem,
\( \dim_{\mathbb C}\mathcal Q(X)=3g-3. \)

There is a canonical identification
\( T_X^*\mathcal T(S)\cong\mathcal Q(X). \)
More precisely, a tangent vector $v\in T_X\mathcal T(S)$ may be represented
by an essentially bounded Beltrami differential $\mu$, and its natural
pairing with $q\in\mathcal Q(X)$ is
\( \langle v,q\rangle
  :=
  \int_X\mu q. \)
Under this pairing, the infinitesimal Teichm\"uller norm is dual to the
$L^1$-norm on holomorphic quadratic differentials:
\begin{equation}
  \|v\|_{\mathrm T}
  =
  \inf\bigl\{
     \|\nu\|_{L^\infty(X)}:
     \nu \text{ represents }v
  \bigr\}
  =
  \sup_{q\in\mathcal Q^1(X)}
  \left|\int_X\mu q\right|.
  \label{eq:infinitesimal-teichmuller-norm}
\end{equation}
See
\cite{Ahlfors2006,GardinerLakic2000,DaskalopoulosWentworth2007}
for these identifications.

\subsection{Teichm\"uller geodesics}
\label{subsec:teichmuller-geodesics}

Let $q\in\mathcal Q(X)\setminus\{0\}$. Away from the zeros of $q$, a local
branch of $\sqrt q$ determines a \emph{natural coordinate}
\( z
  =
  \int\sqrt q
  =
  x+iy, \)
in which
\( q=dz^2. \)
The transition maps between natural coordinates are of the form
\( z\longmapsto\pm z+c. \)
Consequently, these coordinates identify the singular flat metric with the
Euclidean metric,
\( ds_q^2=dx^2+dy^2, \)
and determine a pair of transverse measured foliations. In a natural
coordinate, the horizontal and vertical measured foliations are given,
respectively, by
\[ \mathcal F_{\mathrm h}(q):
  \quad
  y=\mathrm{constant},
  \quad
  d\nu_{\mathrm h}=|dy|, \]
and
\[ \mathcal F_{\mathrm v}(q):
  \quad
  x=\mathrm{constant},
  \quad
  d\nu_{\mathrm v}=|dx|. \]
At a zero of order $m$, these foliations have an $(m+2)$-pronged
singularity, and the singular flat metric has cone angle $(m+2)\pi$.
See \cite{Strebel1984,Masur2009}.

The following classical theorem provides the fundamental link between
holomorphic quadratic differentials and the Teichm\"uller metric
\cite{Teichmuller1940,Bers1960,GardinerLakic2000,Hubbard2006}.

\begin{theorem}[Teichm\"uller's Theorem]
\label{thm:teichmuller-existence-uniqueness}
Let $x=[X,f]$ and $y=[Y,g]$ be points of $\mathcal T(S)$. There exists a
unique quasiconformal map, called the Teichm\"uller map, 
\( h\colon X\to Y \)
satisfying $h\circ f\simeq_{\text{iso}} g$.

If $x\ne y$, then there exist unique unit-area quadratic differentials
\( q\in\mathcal Q^1(X), q'\in\mathcal Q^1(Y), \)
and a number $K>1$ such that
\begin{equation}
  \mu_h
  =
  k\frac{\overline q}{|q|},
  \quad
  k
  =
  \frac{K-1}{K+1},
  \label{eq:teichmuller-map-beltrami}
\end{equation}
almost everywhere on $X$. After choosing compatible natural coordinates
\( z=x+iy\) for $q$, $w=u+iv$ for $q'$,
the map $h$ is affine and has the form
\begin{equation}
  u=K^{1/2}x,
  \quad
  v=K^{-1/2}y.
  \label{eq:teichmuller-affine-map}
\end{equation}
Moreover,
\( d_{\mathrm T}(x,y)=\frac12\log K. \)
If $x=y$, the extremal map is conformal.
\end{theorem}

We now describe the associated Teichm\"uller geodesics. Fix
\( x=[X,f]\in\mathcal T(S)\) and \(q\in\mathcal Q^1(X). \)
For $t\in\mathbb R$, consider the real-linear map
\[ A_t\colon\mathbb C\longrightarrow\mathbb C, \quad
  A_t(x+iy)
  :=
  e^t x+i e^{-t}y. \]
Applying $A_t$ to the natural coordinate charts of $q$, $q(z)=dz^2$, produces a new Riemann
surface $X_t$. Indeed, since the transition maps of the original natural
coordinates are of the form $z\mapsto\pm z+c$, the transformed transition
maps are again holomorphic affine maps.

Let
\( F_t\colon X\to X_t \)
be the resulting quasiconformal map, and let $q_t$ be the holomorphic
quadratic differential on $X_t$ given in the transformed coordinates
\( z_t=e^t x+i e^{-t}y \)
by
\( q_t=dz_t^2. \)
Since $\det A_t=1$, the flat area is preserved, and hence
\( \|q_t\|_1=\|q\|_1=1. \)
Furthermore,
\begin{equation}
  z_t
  =
  \cosh(t)z+\sinh(t)\overline z,
  \quad
  \mu_{F_t}
  =
  \tanh(t)\frac{\overline q}{|q|},
  \quad
  K(F_t)
  =
  e^{2|t|}.
  \label{eq:teichmuller-deformation-beltrami}
\end{equation}

Define
\begin{equation}
  \gamma_q\colon\mathbb R\longrightarrow\mathcal T(S),
  \quad
  \gamma_q(t)
  :=
  [X_t,F_t\circ f].
  \label{eq:teichmuller-geodesic}
\end{equation}
For $s,t\in\mathbb R$, the transition map
\( F_{s,t}
  :=
  F_t\circ F_s^{-1}
  \colon
  X_s\to X_t \)
is given in natural coordinates by
\( z_t
  =
  e^{t-s}\operatorname{Re}z_s
  +
  i e^{-(t-s)}\operatorname{Im}z_s. \)
Consequently,
\( K(F_{s,t})=e^{2|t-s|}. \)
By Theorem~\ref{thm:teichmuller-existence-uniqueness}, the map $F_{s,t}$
is extremal, and therefore
\begin{equation}
  d_{\mathrm T}\bigl(\gamma_q(s),\gamma_q(t)\bigr)
  =
  |t-s|.
  \label{eq:teichmuller-geodesic-unit-speed}
\end{equation}
Thus, $\gamma_q$ is a unit-speed Teichm\"uller geodesic.

The forward tangent vector to $\gamma_q$ at time $t$ is represented by the
Beltrami differential
\begin{equation}
  \dot\gamma_q(t)
  =
  \left[
    \frac{\overline{q_t}}{|q_t|}
  \right]
  \in T_{X_t}\mathcal T(S),
  \quad
  \|\dot\gamma_q(t)\|_{\mathrm T}=1.
  \label{eq:tangent-teichmuller-geodesic}
\end{equation}
Here, the quotient $\overline{q_t}/|q_t|$ is understood almost everywhere,
with an arbitrary value assigned at the zeros of $q_t$.

The map \(F_t\) preserves
the horizontal and
vertical directions in the $q$-coordinate,
\begin{equation}
  F_t^*(ds_{q_t}^2)
  =
  e^{2t}\,dx^2
  +
  e^{-2t}\,dy^2.
  \label{eq:flat-metric-along-teichmuller-geodesic}
\end{equation}
and its pullbacks
on the
associated
measured foliations 
are
\begin{equation}
  F_t^*\mathcal F_{\mathrm h}(q_t)
  =
  e^{-t}\mathcal F_{\mathrm h}(q),
  \quad
  F_t^*\mathcal F_{\mathrm v}(q_t)
  =
  e^t\mathcal F_{\mathrm v}(q).
  \label{eq:foliations-along-teichmuller-geodesic}
\end{equation}

It is often convenient to package this construction as a flow on the
unit-area quadratic-differential bundle
\( \pi\colon
  \mathcal Q^1\mathcal T(S)
  \to
  \mathcal T(S),
 \) and \(
  \pi^{-1}([X,f])=\mathcal Q^1(X). \)
The \emph{Teichm\"uller geodesic flow} is defined by
\[ g_t([X,f],q)
  :=
  ([X_t,F_t\circ f],q_t). \]
The identity
\( A_{s+t}=A_s\circ A_t \)
implies
\( g_{s+t}=g_s\circ g_t, \)
and the projection of the flow orbit is the corresponding Teichm\"uller
geodesic:
\( \pi\bigl(g_t([X,f],q)\bigr)=\gamma_q(t). \)

Every unit-speed Teichm\"uller geodesic arises in this way. If the unit-area
normalization is dropped, all positive multiples of a nonzero quadratic
differential determine the same parametrized geodesic. Replacing $q$ by
$-q$ reverses its orientation:
\( \gamma_{-q}(t)=\gamma_q(-t). \)
In particular, any two points of $\mathcal T(S)$ are joined by a unique
Teichm\"uller geodesic segment, and this segment extends uniquely to a
complete bi-infinite geodesic \cite{Masur2009}.

Although the Teichm\"uller geodesic is constructed using the singular flat
metrics $ds_{q_t}^2$, each Riemann surface $X_t$ also carries its unique
uniformizing hyperbolic metric, which we shall denote by $\sigma_t$. We
will keep these two metrics distinct throughout:
$ds_{q_t}^2=|q_t|$ and  
  $\sigma_t$ is the hyperbolic metric in the conformal class of $X_t$. 

\subsection{Harmonic maps and energy functions}

We begin by recalling the basic definitions. Let
\((M,g_M)\) and \((N,g_N)\) be Riemannian manifolds, with \(M\) closed,
and let \(v:M\to N\) be a smooth map. Its Dirichlet energy is
\[
  E_{g_M,g_N}(v)
  :=
  \frac12\int_P
  |dv|^2\,d\mu_{g_P}
  =
  \frac12\int_M
  \operatorname{tr}_{g_M}(v^*g_N)\,d\mu_{g_M}.
\]
Let \(\nabla dv\) denote the Levi-Civita covariant derivative of \(dv\)
as sections of $\text{End}(TM,  TN)$
equipped with the corresponding
metric obtained from \(g_M\) and \(g_N\). The tension field of
\(v\) is
\[
  \tau_{g_M,g_N}(v)
  :=
  \operatorname{tr}_{g_M}(\nabla dv).
\]
Equivalently, if \(\{e_1,\ldots,e_m
\}\), \(m={\dim M}\),
is a local
\(g_M\)-orthonormal frame, then
\[
  \tau_{g_M,g_N}(v)
  =
  \sum_i
  \left\{
    \nabla^{N}_{dv(e_i)}dv(e_i)
    -
    dv\bigl(\nabla^{M}_{e_i}e_i\bigr)
  \right\}.
\]
The map \(v\) is called \emph{harmonic} if
\(\tau_{g_M,g_N}(v)=0\), also called
the Bochner equation.

We apply these definitions in two complementary situations: first, the
domain is fixed and the hyperbolic target varies, and then the target is
fixed while the conformal structure on the domain varies.

\subsubsection{Varying the target}
\label{subsec:definition-energy-functions}

Let \((M,g)\) be a closed, connected Riemannian manifold, let \(S\) be a
closed oriented surface of genus at least \(2\), and fix a smooth map
\(
  u_0:M\to S.
\)

A smooth map
\(
  u_X:(M,g)\to(X,\sigma_X)
\)
in the same homotopy
class of $u_0$ is harmonic if
\(
  \tau_{g,\sigma_X}(u_X)=0.
\)
The Eells--Sampson theorem gives the existence of a harmonic representative
in the homotopy class of \(f\circ u_0\)
\cite{EellsSampson1964}. Under the standing assumption that
\((u_0)_*\pi_1(M)\) is nontrivial and noncyclic, this harmonic representative
is unique by Hartman's uniqueness theorem \cite{Hartman1967}. 

For a smooth map \(v:M\to X\), define
\begin{equation}
  E_X(v)
  :=
  \frac12\int_M |dv|^2\,d\mu_g
  =
  \frac12\int_M
  \operatorname{tr}_g\bigl(v^*\sigma_X\bigr)\,d\mu_g.
  \label{eq:map-energy-varying-target}
\end{equation}
The \emph{energy function} associated with the fixed data
\((M,g,u_0)\) is
\begin{equation}
  \mathcal E_{u_0}:
  \mathcal T(S)\longrightarrow\mathbb R_{>0},
  \qquad
  \mathcal E_{u_0}(X)
  :=
  E_X(u_X), \quad X\in \mathcal{T}(S).
  \label{eq:def-energy-function}
\end{equation}
Since the target has nonpositive sectional curvature, the harmonic map
is an absolute energy minimizer in its homotopy class
\cite{EellsSampson1964,Jost2017}. Consequently,
\begin{equation}
  \mathcal E_{u_0}(X)
  =
  \inf_{\substack{v:M\to X\\ v\simeq f\circ u_0}}
  E_X(v), \, X\in T(S),
  \label{eq:energy-function-as-minimum}
\end{equation}
where \(\simeq\) denotes homotopy of maps.

This definition is independent of the chosen representative of the marked
Riemann surface and \(\mathcal E_{u_0}(X)\) depends only on the point
\(X\in\mathcal T(S)\), which by the uniqueness of harmonic maps.

The resulting function is smooth on \(\mathcal T(S)\); see
\cite{Yamada1999}. For foundational relations among harmonic maps,
hyperbolic length, and energy on Teichm\"uller space, see also
\cite{Minsky1992}. We shall write \(\mathcal E(X)\) when the fixed data
\((M,g,u_0)\) are understood, reserving \(E_X(v)\) for the energy of an
individual map \(v:M\to(X,\sigma_X)\).

If the standing noncyclicity assumption is dropped, this construction still
contains hyperbolic length functions as a one-dimensional special case.
Let \(M=S_L^1\) be a circle of length \(L\), and suppose that \(u_0\)
represents a nontrivial free homotopy class \(\alpha\) on \(S\). The harmonic
representative maps \(S_L^1\) at constant speed onto the
\(\sigma_X\)-geodesic representative of \(\alpha\). Its speed is
\(\ell_X(\alpha)/L\), and therefore
\begin{equation}
  \mathcal E_{\alpha}(X)
  =
  \frac12
  \int_{S_L^1}
  \left(\frac{\ell_X(\alpha)}{L}\right)^2\,ds
  =
  \frac{\ell_X(\alpha)^2}{2L}.
  \label{eq:circle-energy-length-square}
\end{equation}
Thus energy functions with varying hyperbolic targets may be regarded as
higher-dimensional analogs of squared hyperbolic length functions.

\subsubsection{Varying the domain}
\label{subsec:def-energy-varying-domain}

Let \(\Sigma\) and \(S\) be closed, connected, oriented surfaces of genus at
least \(2\), and fix a hyperbolic metric \(h\) on \(S\). Let
\(
  u_0:\Sigma\to S
\)
be a smooth map of nonzero degree. For a marked Riemann surface
\(
  X=[X,m_X]\in\mathcal T(\Sigma),
\)
let \(\sigma_X\) denote the hyperbolic metric in the conformal class of
\(X\). The marking \(m_X:\Sigma\to X\) determines the homotopy class of
maps
\(
  u_0\circ m_X^{-1}:X\to S.
\)
For a smooth map \(v:X\to S\), define
\begin{equation}
  E_{X,h}(v)
  :=
  \frac12\int_X
  |dv|^2\,dA_{\sigma_X}
  =
  \frac12\int_X
  \operatorname{tr}_{\sigma_X}(v^*h)\,dA_{\sigma_X}.
  \label{eq:map-energy-varying-domain}
\end{equation}
Such a map is harmonic if
\(
  \tau_{\sigma_X,h}(v)
  =
  \operatorname{tr}_{\sigma_X}(\nabla dv)
  =
  0,
\)
or equivalently, if it is a critical point of
\(v\mapsto E_{X,h}(v)\).

There exists a unique harmonic map
\(
  u_X:(X,\sigma_X)\to(S,h)
\)
in the homotopy class of \(u_0\circ m_X^{-1}\); see
\cite{EellsSampson1964,Hartman1967,Sampson1978}. Here nonzero degree
precludes the image of the induced fundamental-group homomorphism from
being trivial or cyclic, and hence gives the uniqueness required below.

The associated \emph{energy function} is
\begin{equation}
  E_{u_0}:
  \mathcal T(\Sigma)\to\mathbb R_{>0},
  \quad
  E_{u_0}(X)
  :=
  E_{X,h}(u_X)
  =
  \frac12\int_X
  |du_X|^2\,dA_{\sigma_X}.
  \label{eq:def-energy-varying-domain}
\end{equation}
Since the energy of a map from a surface is conformally invariant, $E_{u_0}(X)$ depends only on the point $X\in\mathcal T(\Sigma)$ and not on the chosen metric representative of its conformal class. After choosing a local smooth family of marked metric representatives, the maps $u_X$, and hence the energy, depend smoothly on $X$. Indeed, nonzero degree precludes a cyclic fundamental-group image, so the Jacobi operator is nondegenerate and the same \(C^{2,\alpha}\)-implicit-function argument used in Lemma~\ref{lem:nondegeneracy-target-jacobi} applies. We will primarily be concerned with the following special case. \begin{assumption} \label{assump:-covering} The map $u_0\colon\Sigma\to S$ is an orientation-preserving covering map. \end{assumption}
Pulling back the fixed target metric gives a hyperbolic metric on $\Sigma$. We denote the resulting point of Teichm\"uller space by
\begin{equation} Y := (\Sigma,u_0^*h) \in\mathcal T(\Sigma), \qquad A_Y := \operatorname{Area}(Y) =2\pi|\chi(\Sigma)|.
  \label{eq:def-covering-minimizer-Y}
\end{equation}
Thus \( u_0\colon(Y,u_0^*h)\to(S,h) \) is a local isometry. \begin{lemma} \label{lem:covering-factorization} For every $X\in\mathcal T(\Sigma)$, let \( f_X\colon(X,\sigma_X)\to(Y,u_0^*h) \) be the harmonic diffeomorphism in the homotopy class of the identity \cite{Sampson1978,SchoenYau1978}. Then \begin{equation} u_X=u_0\circ f_X \qquad\text{and}\qquad E_{u_0}(X)=E(f_X). \label{eq:covering-factorization} \end{equation} In particular, \begin{equation} E_{u_0}(X)\ge A_Y, \qquad E_{u_0}(Y)=A_Y, \label{eq:covering-energy-minimum} \end{equation} and equality holds if and only if $X=Y$. \end{lemma} \begin{proof} Since $u_0$ is a local isometry, it is totally geodesic. Hence the composition $u_0\circ f_X$ is harmonic and belongs to the homotopy class of $u_0$. Uniqueness of the harmonic representative gives $u_X=u_0\circ f_X$. The identity $E_{u_0}(X)=E(f_X)$ follows because $du_0$ is an isometry on every tangent space. Let $s_1$ and $s_2$ be the singular values of $df_X$. Pointwise, \( \frac12\bigl(s_1^2+s_2^2\bigr)\ge s_1s_2=J(f_X). \) After integration, and using that $f_X$ has degree one, we obtain \[ E(f_X) \ge \int_XJ(f_X)\,dA_{\sigma_X} = \operatorname{Area}(Y) = A_Y. \] Equality holds precisely when $s_1=s_2$ everywhere, that is, when $f_X$ is conformal. This is equivalent to $X=Y$ in Teichm\"uller space. \end{proof} The factorization in Lemma~\ref{lem:covering-factorization} reduces the covering-map energy to the energy of the harmonic marking map from the varying surface $X$ to the fixed surface $Y$. The point $Y$ will serve as the center in the coarse estimates below.

  \subsection{The pull-back metric \(\widehat\sigma_t:=f_t^*\sigma_t\)
 by the harmonic map $f_t$}
\label{subsec:harmonic-map-gauge}

For each $t$, let
\(
  f_t\colon (X_0,\sigma_0)\to (X_t,\sigma_t)
\)
be the harmonic map in the homotopy class determined by the marking. By the
classical existence, uniqueness, and univalence theorems for harmonic maps
between negatively curved surfaces, $f_t$ is a uniquely determined harmonic
diffeomorphism
\cite{EellsSampson1964,Hartman1967,Sampson1978,SchoenYau1978}. We choose the
normalization $f_0=\operatorname{id}_{X_0}$ and set \(
  \widehat\sigma_t:=f_t^*\sigma_t.  \). 
Write
\(
  \sigma_0=\rho(z)|dz|^2
\)
in a local conformal coordinate $z=x+iy$ on $X_0$, and write
$\sigma_t=\rho_t(w)|dw|^2$ in a local conformal coordinate on $X_t$. Define
\[
  H_t
  :=
  \frac{\rho_t(f_t(z))}{\rho(z)}\bigl|(f_t)_z\bigr|^2,
  \quad
  L_t
  :=
  \frac{\rho_t(f_t(z))}{\rho(z)}\bigl|(f_t)_{\bar z}\bigr|^2,
\]
and let
\[
  \Phi(t)=\phi(t)\,dz^2
  :=
  \rho_t(f_t(z))(f_t)_z
  \overline{(f_t)_{\bar z}}\,dz^2
\]
be the Hopf differential of $f_t$. Since $f_t$ is harmonic, $\Phi(t)$ is
holomorphic. Moreover,
\(
  H_tL_t=\frac{|\phi(t)|^2}{\rho^2},
\)
and the pulled-back metric
$ \widehat\sigma_t$ is
\begin{equation}
  \widehat\sigma_t
  =
  \phi(t)\,dz^2
  +
  \rho\left(
    H_t+\frac{|\phi(t)|^2}{\rho^2H_t}
  \right)dz\,d\bar z
  +
  \overline{\phi(t)}\,d\bar z^2.
  \label{eq:harmonic-gauge-exact-decomposition}
\end{equation}

Let \(
  \Delta_0
  :=
  \frac{4}{\rho}\partial_z\partial_{\bar z}
  =
  \frac{1}{\rho}
  \left(\partial_x^2+\partial_y^2\right)
  \) be the Laplace-Beltrami operator
  on the space \((X_0, \sigma_0)\). The Bochner equation for a harmonic map between hyperbolic surfaces becomes
\begin{equation}
  \Delta_0\log H_t
  =
  2H_t
  -
  2\frac{|\phi(t)|^2}{\rho^2H_t}
  -2.
  \label{eq:bochner-equation-harmonic-gauge}
\end{equation}
Equations~\eqref{eq:harmonic-gauge-exact-decomposition} and
\eqref{eq:bochner-equation-harmonic-gauge} are the basic identities underlying
Wolf's harmonic-map parametrization of Teichm\"uller space
\cite{Wolf1989,DaskalopoulosWentworth2007}. In fact, the map
\[
  \mathcal W_{\sigma_0}\colon
  \mathcal T(S)\longrightarrow\mathcal Q(X_0),
  \qquad
  [X]\longmapsto\operatorname{Hopf}(f_X),
\]
is a diffeomorphism.

We now apply these identities to the Teichm\"uller geodesic $\gamma_q$.
Since $\Phi(0)=0$, write
\begin{equation}
  \Phi(t)
  =
  t\Phi_1+\frac{t^2}{2}\Phi_2+O(t^3),
  \qquad
  \Phi_j=\phi_j\,dz^2\in\mathcal Q(X_0).
  \label{eq:hopf-expansion-teichmuller-geodesic}
\end{equation}
The first coefficient is determined by the initial Teichm\"uller tangent
vector. More precisely, let
\(
  \nu_q:=\frac{\overline q}{|q|}
\)
be the extremal Beltrami differential of the Teichm\"uller deformation,
understood almost everywhere. The Beltrami differential of $f_t$ satisfies
\(
  \Phi(t)=\sigma_0 H_t\,\overline{\mu(f_t)},
\)
and hence
\[
  \left.\frac{d}{dt}\right|_{t=0}\mu(f_t)
  =
  \frac{\overline{\Phi_1}}{\sigma_0}
  =
  \frac{\overline{\phi_1}}{\rho}\frac{d\bar z}{dz}.
\]
It follows that
\begin{equation}
  \frac{\overline{\Phi_1}}{\sigma_0}
  =
  P_{\mathrm{harm}}(\nu_q),
  \label{eq:harmonic-projection-teichmuller-direction}
\end{equation}
where $P_{\mathrm{harm}}$ denotes passage to the unique harmonic Beltrami
representative of an infinitesimal Teichm\"uller class. Equivalently,
\begin{equation}
  \int_{X_0}
  (
    \nu_q-\frac{\overline{\Phi_1}}{\sigma_0}
  )\Psi
  =0
  \qquad
  \text{for every }\Psi\in\mathcal Q(X_0),
  \label{eq:harmonic-projection-pairing}
\end{equation}
where the integral denotes the canonical pairing between Beltrami
differentials and holomorphic quadratic differentials. There is no canonical
equality $\Phi_1=q$, and such an equality should not be assumed: $q$
determines the extremal representative $\nu_q$, whereas $\Phi_1$ determines
the harmonic representative of the same tangent class
\cite{DaskalopoulosWentworth2007}.

The second coefficient
\(
  \Phi_2
  =
  \left.\frac{d^2}{dt^2}\right|_{t=0}
  \operatorname{Hopf}(f_t)
\)
is the acceleration of the Teichm\"uller geodesic in Wolf coordinates.

\begin{proposition}
\label{prop:harmonic-gauge-expansion}
Set
\(
  U_1
  :=
  |\Phi_1|_{\sigma_0}^2
  =
  \frac{|\phi_1|^2}{\rho^2}.
\)
Then
\begin{equation}
  \dot H_0=0,
  \qquad
  \ddot H_0
  =
  -4(\Delta_0-2)^{-1}U_1.
  \label{eq:first-second-variation-H}
\end{equation}
Consequently, for every $k\ge 0$,
\begin{equation}
  \begin{aligned}
  \widehat\sigma_t
  ={}&
  \sigma_0
  +2t\operatorname{Re}\Phi_1
  \\
  &+
  \frac{t^2}{2}
  \left[
    2\operatorname{Re}\Phi_2
    +
    \left(
      2U_1-4(\Delta_0-2)^{-1}U_1
    \right)\sigma_0
  \right]
  +O_{C^k}(t^3).
  \end{aligned}
  \label{eq:harmonic-gauge-second-order-expansion}
\end{equation}
Here $O_{C^k}(t^3)$ means that, after subtracting the displayed
terms up to order two, the remainder $R_t$ satisfies
\(
  \|R_t\|_{C^k(X_0,\sigma_0)}\le C_k |t|^3
\)
for $|t|$ sufficiently small. Equivalently, all covariant derivatives
of $R_t$ up to order $k$, computed with respect to $\sigma_0$, are
uniformly $O(t^3)$.
Equivalently, in the coordinate $z$,
\begin{equation}
  \begin{aligned}
  \widehat\sigma_t
  ={}&
  \rho\,dz\,d\bar z
  +
  t\left(
    \phi_1\,dz^2+\overline{\phi_1}\,d\bar z^2
  \right)
  +
  \frac{t^2}{2}
  \Bigg[
    \phi_2\,dz^2+\overline{\phi_2}\,d\bar z^2
  \\
  &+
    \left\{
      \frac{2|\phi_1|^2}{\rho^2}
      -
      4(\Delta_0-2)^{-1}
      \left(
        \frac{|\phi_1|^2}{\rho^2}
      \right)
    \right\}
    \rho\,dz\,d\bar z
  \Bigg]
  +O_{C^k}(t^3).
  \end{aligned}
  \label{eq:harmonic-gauge-second-order-expansion-local}
\end{equation}
\end{proposition}

\begin{proof}
At $t=0$ one has $H_0=1$ and $\Phi(0)=0$. Differentiating
\eqref{eq:bochner-equation-harmonic-gauge} once gives
\(
  (\Delta_0-2)\dot H_0=0.
\)
Since $\Delta_0-2$ is invertible on the closed surface $X_0$, it follows that
$\dot H_0=0$. Differentiating a second time yields
\(
  (\Delta_0-2)\ddot H_0=-4U_1,
\)
which proves \eqref{eq:first-second-variation-H}. Furthermore,
\[
  \frac{|\phi(t)|^2}{\rho^2H_t}
  =
  t^2U_1+O(t^3).
\]
Substituting this identity,
\eqref{eq:hopf-expansion-teichmuller-geodesic}, and
\eqref{eq:first-second-variation-H} into
\eqref{eq:harmonic-gauge-exact-decomposition} gives
\eqref{eq:harmonic-gauge-second-order-expansion}.
\end{proof}

\begin{remark}
If one takes a straight line in Wolf coordinates,
\(
  \Phi(t)=t\Phi,
\)
then $H_t$ and $L_t$ are even functions of $t$. In that special case
$\Phi_2=0$, and the remainder in
\eqref{eq:harmonic-gauge-second-order-expansion} improves to
$O_{C^k}(t^4)$. Wolf also showed that the associated Beltrami section has the
same two-jet at the origin as a Weil--Petersson geodesic
\cite{Wolf1989}. This statement concerns straight lines in harmonic-map
coordinates. A genuine Teichm\"uller geodesic need not be a straight line in
those coordinates, and therefore the term $\Phi_2$ must be retained in
\eqref{eq:harmonic-gauge-second-order-expansion}.
\end{remark}

\subsection{The pull-back metric \(\widetilde\sigma_t:=F_t^*\sigma_t\)
 by the  Teichm\"uller map $F_t$}
\label{subsec:teichmuller-map-gauge}

We next use the extremal maps defining the Teichm\"uller geodesic,
\(
  F_t\colon X_0\to X_t,
\)
and set
\(
  \widetilde\sigma_t:=F_t^*\sigma_t.
\)
This metric has an explicit form on the regular set $X_0\setminus Z(q)$. Choose a
natural coordinate
\(
  \zeta=x+iy,
\) \(
  q=d\zeta^2,
\)
and write
\(
  \sigma_0=\lambda(x,y)(dx^2+dy^2).
\)
In these coordinates,
\[
  F_t(x+iy)=e^t x+i e^{-t} y.
\]

If
\(
  \zeta_t=e^t x+i e^{-t}y
\)
is the corresponding conformal coordinate on $X_t$, write
\(
  \sigma_t=\lambda_t(\zeta_t)|d\zeta_t|^2
\)
and set
\(
  \Lambda_t(x,y):=\lambda_t(e^t x,e^{-t}y).
\)
Then
\[
  \widetilde\sigma_t
  =
  \Lambda_t(x,y)
  \left(
    e^{2t}dx^2+e^{-2t}dy^2
  \right).
\]
Define $u_t$ by
\(
  \Lambda_t=\lambda e^{2u_t}
\), \(
  u_0=0.
\)
On $X_0\setminus Z(q)$, introduce the symmetric tensor
\begin{equation}
  B_q
  :=
  \frac{\sigma_0}{|q|}\operatorname{Re}q.
  \label{eq:def-Bq}
\end{equation}
Here $\sigma_0/|q|$ denotes the ratio of the two conformal metrics. In a
$q$-natural coordinate,
\(
  B_q=\lambda(dx^2-dy^2),
\)
so that
\(
  \operatorname{tr}_{\sigma_0}B_q=0
\), \(
  (B_q^\sharp)^2=\operatorname{Id}.
\)
Here $B_q^\sharp$ denotes the endomorphism of $TX_0$ obtained by
raising one index of $B_q$ with respect to $\sigma_0$, namely
\(
  \sigma_0(B_q^\sharp X,Y)=B_q(X,Y)
\)
for all tangent vectors $X,Y$. 
The pulled-back hyperbolic metric therefore has the exact form
\begin{equation}
  \widetilde\sigma_t
  =
  e^{2u_t}
  \left(
    \cosh(2t)\,\sigma_0
    +
    \sinh(2t)\,B_q
  \right).
  \label{eq:teichmuller-gauge-exact-metric}
\end{equation}

Since a conformal metric
$\lambda_t(\zeta_t)|d\zeta_t|^2$ has curvature $-1$ precisely when
\(
  \left(
    \partial_{\operatorname{Re}\zeta_t}^2
    +
    \partial_{\operatorname{Im}\zeta_t}^2
  \right)
  \log\lambda_t
  =
  2\lambda_t,
\)
the function $u_t$ satisfies the exact uniformization equation
\begin{equation}
  \left(
    e^{-2t}\partial_x^2+e^{2t}\partial_y^2
  \right)
  \left(
    \log\lambda+2u_t
  \right)
  =
  2\lambda e^{2u_t}.
  \label{eq:teichmuller-gauge-uniformization}
\end{equation}
This equation is understood on $X_0\setminus Z(q)$.

Define the second-order operator
\begin{equation}
  \mathcal D_qv
  :=
  \frac{1}{\lambda}
  \left(
    v_{xx}-v_{yy}
  \right).
  \label{eq:def-Dq}
\end{equation}
Because two natural coordinates differ by
$\zeta\mapsto\pm\zeta+c$, the operator $\mathcal D_q$ is well defined on
$X_0\setminus Z(q)$. Equivalently,
\begin{equation}
  \mathcal D_qv
  =
  \operatorname{div}_{\sigma_0}
  \left(
    B_q^\sharp\nabla^{\sigma_0}v
  \right).
  \label{eq:Dq-invariant-form}
\end{equation}
For the function
\(
  \log\lambda=\log\left(\frac{\sigma_0}{|q|}\right),
\)
one also has
\begin{equation}
  \mathcal D_q\log\lambda
  =
  \nabla^i\nabla^j(B_q)_{ij}
  \label{eq:Dq-loglambda-divdiv-Bq}
\end{equation}
away from $Z(q)$, where indices are raised with $\sigma_0$. Here $\nabla^j:=(\sigma_0)^{jk}\nabla_k$.

\begin{proposition}
\label{prop:teichmuller-gauge-expansion}
Let $a$ and $b$ be the real-valued functions determined locally by
\begin{equation}
  u_t
  =
  ta+\frac{t^2}{2}b+O(t^3).
  \label{eq:ut-expansion}
\end{equation}
Then, on $X_0\setminus Z(q)$,
\begin{align}
  (\Delta_0-2)a
  &=
  \mathcal D_q\log\lambda,
  \label{eq:equation-for-a}
  \\
  (\Delta_0-2)b
  &=
  4\left(
    \mathcal D_q a+a^2-1
  \right).
  \label{eq:equation-for-b}
\end{align}
Consequently, for every compact set
$K\Subset X_0\setminus Z(q)$ and every $k\ge 0$,
\begin{equation}
  \begin{aligned}
  \widetilde\sigma_t
  ={}&
  \sigma_0
  +
  2t\left(
    B_q+a\sigma_0
  \right)
  \\
  &+
  t^2
  \left[
    \left(
      2+b+2a^2
    \right)\sigma_0
    +
    4aB_q
  \right]
  +
  O_{C^k(K)}(t^3).
  \end{aligned}
  \label{eq:teichmuller-gauge-second-order-expansion}
\end{equation}
Equivalently,
\begin{align}
  \dot{\widetilde\sigma}_0
  &=
  2B_q+2a\sigma_0,
  \label{eq:first-variation-teichmuller-gauge}
  \\
  \ddot{\widetilde\sigma}_0
  &=
  \left(
    4+2b+4a^2
  \right)\sigma_0
  +
  8aB_q.
  \label{eq:second-variation-teichmuller-gauge}
\end{align}
\end{proposition}

\begin{proof}
At $t=0$, equation~\eqref{eq:teichmuller-gauge-uniformization} reduces to
\(
  \left(
    \partial_x^2+\partial_y^2
  \right)\log\lambda
  =
  2\lambda.
\)
Differentiating
\eqref{eq:teichmuller-gauge-uniformization} once at $t=0$ gives
\[
  2\left(
    \partial_x^2+\partial_y^2
  \right)a
  +
  2\left(
    \partial_y^2-\partial_x^2
  \right)\log\lambda
  =
  4\lambda a,
\]
which is equivalent to \eqref{eq:equation-for-a}. Differentiating a second
time, and using
\(
  2u_t=2ta+t^2b+O(t^3)
\)
and the zeroth-order equation gives
\(
  8\lambda
  -8(\partial_x^2-\partial_y^2)a
  +2(\partial_x^2+\partial_y^2)b
  =
  4\lambda b+8\lambda a^2.
\)
After division by \(2\lambda\), this becomes
\(
  4-4\mathcal D_q a+\Delta_0b
  =
  2b+4a^2.
\)
Rearranging yields
\[
  (\Delta_0-2)b
  =
  4\left(
    \mathcal D_q a+a^2-1
  \right),
\]
proving \eqref{eq:equation-for-b}.
Finally,
\[
  e^{2u_t}
  =
  1+2ta+t^2(b+2a^2)+O(t^3),
\]
while
\(
  \cosh(2t)=1+2t^2+O(t^4)
 \) and \(
  \sinh(2t)=2t+O(t^3).
\)
Substitution into \eqref{eq:teichmuller-gauge-exact-metric} gives
\eqref{eq:teichmuller-gauge-second-order-expansion}, and
\eqref{eq:first-variation-teichmuller-gauge}--
\eqref{eq:second-variation-teichmuller-gauge} follow immediately.
\end{proof}

Equations~\eqref{eq:equation-for-a} and
\eqref{eq:equation-for-b} are local identities on
\(X_0\setminus Z(q)\). We deliberately do not rewrite them by applying a
global inverse of \(\Delta_0-2\). Indeed, the natural-coordinate coefficients
are singular at \(Z(q)\), and the functions \(a\) and \(b\) generally have
direction-dependent asymptotics there. Thus neither
\(\mathcal D_q\log\lambda\) nor \(\mathcal D_qa\) belongs, without further
interpretation, to the standard smooth domain of the global inverse of
\(\Delta_0-2\) on \(X_0\). A global weak or distributional formulation would
require specifying the admissible asymptotic behavior at every zero of
\(q\). We will use only the local equations above.

The two one-parameter
families of metrics describe the same path in Teichm\"uller space on the fixed surface $X_0$. In the
Teichm\"uller-map family, the first variation
\(
  2B_q+2a\sigma_0
\)
is generally neither trace-free nor divergence-free. After removing its
infinitesimal diffeomorphism component, its transverse-traceless part is
precisely
\(
  2\operatorname{Re}\Phi_1.
\)
At the second order, the  harmonic-map family has additional terms, whose
transverse-traceless component is encoded by $2\operatorname{Re}\Phi_2$ in
\eqref{eq:harmonic-gauge-second-order-expansion}; see
\cite{Wolf1989,DaskalopoulosWentworth2007} for the harmonic-map slice and the
identification of tangent tensors with holomorphic quadratic differentials.

Finally, the distinction in regularity is essential. The harmonic-map family
produces a global smooth expansion on $X_0$. By contrast, natural coordinates
degenerate at the zeros of $q$, so
\eqref{eq:teichmuller-gauge-exact-metric}--
\eqref{eq:second-variation-teichmuller-gauge} are $C^k$ expansions only on
compact subsets of $X_0\setminus Z(q)$. For global variational arguments
across $Z(q)$, we shall therefore use the harmonic-map family. Higher-order
coefficients in either family can be obtained recursively by differentiating
\eqref{eq:bochner-equation-harmonic-gauge} or
\eqref{eq:teichmuller-gauge-uniformization}, respectively.

\section{Quasi-convexity of energy functions with varying target}
\label{subsec:quasi-convexity-energy-functions}

Let $(M^n,g)$ be a closed, connected Riemannian manifold, and fix a smooth
map
\(
  u_0\colon M\to S.
\)
We shall study the convexity along $u_t: M\to X_t$ described above.
If the original representative of the homotopy class is only continuous, we
replace it once and for all by a smooth representative; the energy function
defined below depends only on the homotopy class.
let $\sigma_X$ be the hyperbolic
metric on $X$, and let
\(
  u_X\colon (M,g)\to (X,\sigma_X)
\)
be a harmonic map in the homotopy class of $f\circ u_0$. Its existence
follows from the Eells--Sampson theorem \cite{EellsSampson1964}. Throughout
this section, unless explicitly stated otherwise, we assume that the
subgroup
\(
  (u_0)_*\pi_1(M)\subset\pi_1(S)
\)
is nontrivial and noncyclic. In particular, the homotopy class is not
represented by a map whose image is a point or a closed geodesic. Under this
assumption, Hartman's uniqueness theorem gives a unique harmonic
representative; see \cite{Hartman1967}.

We impose the following filling hypothesis on the homotopy class of \(u_0\).
There exists a finite collection
\begin{equation}
  \Gamma=\{\alpha_1,\ldots,\alpha_N\}
  \label{eq:def-filling-family-Gamma}
\end{equation}
of isotopy classes of essential simple closed curves on \(S\) satisfying the
following conditions:
\begin{enumerate}
  \item The collection \(\Gamma\) fills \(S\). More precisely, the classes
  \(\alpha_1,\ldots,\alpha_N\) admit representatives
  \(a_1,\ldots,a_N\) in pairwise minimal position such that every connected
  component of
  \(
    S\setminus\bigcup_{j=1}^N a_j
  \)
  is a disk. Here pairwise minimal position means that, for every
  \(i\ne j\), the representatives \(a_i\) and \(a_j\) intersect transversely
  and satisfy
  \(
    \#(a_i\cap a_j)=i(\alpha_i,\alpha_j),
  \)
  where \(i(\alpha_i,\alpha_j)\) denotes the geometric intersection number
  of the two isotopy classes, see \cite[Pages 30-31]{FarbMargalit2012}. The filling condition is also equivalent to requiring that
  every essential simple closed curve \(\delta\) on \(S\) satisfy
  \(i(\delta,\alpha_j)>0\) for at least one \(j\); see
  \cite[Section~III, Definition]{Kerckhoff1983}.

  \item For every \(j\), there exists a fixed piecewise smooth closed curve
  \(\beta_j:S^1\to M\) such that \(u_0\circ\beta_j\) is freely homotopic in
  \(S\) to a representative of the class \(\alpha_j\).
\end{enumerate}

We fix these loops and write
\(
  \mathfrak F=(\Gamma,\mathcal B),
\) \(
  \mathcal B=\{\beta_1,\ldots,\beta_N\},
\)
calling \(\mathfrak F\) the filling data. Constants in the comparison below
are allowed to depend on this chosen data; in particular, the lower bound uses
the fixed \(g\)-lengths of the loops \(\beta_j\).
As mentioned in the introduction, the filling hypothesis forces
\((u_0)_*\pi_1(M)\) to be nontrivial and noncyclic. Hence the standing
existence, uniqueness, and nondegeneracy assumptions from the beginning of
this section hold throughout this subsection.
In particular, this hypothesis holds whenever
\(
  (u_0)_*\colon\pi_1(M)\twoheadrightarrow\pi_1(S)
\)
is surjective. 

For a marked hyperbolic surface
\((X,m_X)\), with hyperbolic metric \(\sigma_X\), and an isotopy class
\(\alpha\) of essential simple closed curves on \(S\), let
\(\alpha^X\subset X\) denote the unique \(\sigma_X\)-geodesic representative
of the free homotopy class \(m_X(\alpha)\). We use the standard notation
\begin{equation}
  \ell_X(\alpha)
  :=
  \operatorname{Length}_{\sigma_X}(\alpha^X)
  =
  \ell_{\sigma_X}\bigl(m_X(\alpha)\bigr).
\end{equation}
This convention is independent of the chosen representative of the marked
surface, since the biholomorphism relating two equivalent marked surfaces is
an isometry of their uniformizing hyperbolic metrics and carries the
corresponding free homotopy classes to one another.
Thus the marking is incorporated into the notation \(\ell_X(\alpha)\) and
will usually be suppressed.
Define
\begin{equation}
  \Lambda_\Gamma(X)
  :=
  \sum_{j=1}^N\ell_X(\alpha_j)^2,\quad
  L_\Gamma(X):=\sum_{j=1}^N\ell_X(\alpha_j).  \label{eq:def-filling-length-square}
\end{equation}
Since \(\Gamma\) fills \(S\), the length sum
\(
  L_\Gamma(X)
\)
is a proper function on \(\mathcal T(S)\); see
\cite[Lemma~3.1]{Kerckhoff1983}. In particular, \(L_\Gamma\) attains a
positive minimum on \(\mathcal T(S)\).
Consequently any estimate of the form \(C(1+L_\Gamma(X))\) may be replaced,
after changing the constant, by \(C'L_\Gamma(X)\).

We regard each \(\alpha_j\) as an isotopy class of essential simple closed
curves on the fixed oriented surface \(S\). Recall that a finite collection
\(\Gamma=\{\alpha_1,\ldots,\alpha_N\}\) \emph{fills} \(S\) if every
essential simple closed curve \(\delta\) satisfies
\(i(\delta,\alpha_j)>0\) for at least one \(j\).
See \cite[Proposition~1.7 and Corollary~1.9]{FarbMargalit2012}.
For this criterion and for the fact that distinct simple closed geodesics on a hyperbolic surface
are in minimal position.

A \emph{ribbon graph} is a finite graph, allowing loops and multiple edges, together
with a cyclic ordering of the half-edges incident to each vertex; a
ribbon-graph isomorphism is required to preserve these cyclic orderings. We
use the standard terminology and boundary-cycle construction from
\cite[Definitions~1.5--1.7]{MulasePenkava1998}. In the present setting, a
\emph{\(\Gamma\)-labeled ribbon graph} is a ribbon graph whose edges carry
labels in \(\{1,\ldots,N\}\), with the edges labeled \(j\) forming the
closed edge-cycle corresponding to the curve \(\alpha_j\). An isomorphism of
\(\Gamma\)-labeled ribbon graphs is required to preserve all labels. Its
\emph{\(\Gamma\)-labeled ribbon-graph type} is its isomorphism class.

A \emph{marked realization} of a \(\Gamma\)-labeled ribbon graph \(G\) is a
label-preserving cellular embedding \(i:G\hookrightarrow S\) such that the
closed cycle labeled \(j\) represents the prescribed isotopy class
\(\alpha_j\). Two marked realizations
\(i_0:G_0\hookrightarrow S\) and \(i_1:G_1\hookrightarrow S\) are called
\emph{marked ambient-isotopic} if there exist a label-preserving ribbon-graph
isomorphism \(\phi:G_0\to G_1\) and an ambient isotopy
\(H_t:S\to S\), \(0\le t\le1\), through orientation-preserving
homeomorphisms, such that \(H_0=\operatorname{id}_S\) and
\(H_1\circ i_0=i_1\circ\phi\). The resulting equivalence class is called a
\emph{marked \(\Gamma\)-configuration}. This is the labeled and marked
refinement of the notion of a configuration used by Hass--Scott; see
\cite[p.~201]{HassScott1999}.

 If
\(m_*:S\to X_*\) and \(m_X:S\to X\) are markings, a map
\(F_X:X_*\to X\) will be called \emph{marking-compatible} if
\(F_X\circ m_*\simeq m_X\). Here and below, \(\simeq\) denotes homotopy of maps.

\begin{lemma}
\label{lem:lipschitz-marking-filling-lengths}
Fix a reference marked hyperbolic surface \(X_*\in\mathcal T(S)\), with
reference marking \(m_*:S\to X_*\). There exists a constant
\(C_{\Gamma,*}>0\), depending only on \(\Gamma\) and the fixed reference
data, such that, for every marked hyperbolic surface
\(X\in\mathcal T(S)\), with marking \(m_X:S\to X\), there is a
marking-compatible Lipschitz map \(F_X:X_*\to X\) satisfying
\begin{equation}
  \operatorname{Lip}(F_X)
  \le C_{\Gamma,*}L_\Gamma(X)
  .
  \label{eq:lipschitz-marking-bound}
\end{equation}
\end{lemma}

\begin{proof}
We divide the construction into five steps.

\smallskip
\noindent
\emph{Step 1: existence of short cellular representatives.}
For \(X\in\mathcal T(S)\), with marking \(m_X:S\to X\), let
\(\alpha_j^X\subset X\) denote the unique
\(\sigma_X\)-geodesic representative of the free homotopy class
\(m_X(\alpha_j)\). By definition,
\(
  \operatorname{Length}_{\sigma_X}(\alpha_j^X)
  =
  \ell_X(\alpha_j).
\)
Distinct simple closed geodesics are pairwise in minimal position, so
\(\#(\alpha_i^X\cap\alpha_j^X)=i(\alpha_i,\alpha_j)\) for every
\(i\ne j\). Since \(\Gamma\) fills \(S\), every component of
\(X\setminus\bigcup_j\alpha_j^X\) is a disk. In particular, the union is
connected. Indeed, the disk-complement characterization of a filling system
is exactly the one recalled above from \cite[Section~III]{Kerckhoff1983}.

Several of the geodesics may pass through the same point. There are only
finitely many such multiple intersection points. Choose pairwise disjoint
coordinate disks around them, each containing no other intersection point.
Inside each disk, replace the incident geodesic arcs by a generic
\(C^1\)-small perturbation, fixed near the boundary of the disk, in such a
way that every pair of arcs which met at the original multiple point still
meets exactly once, while all resulting intersections are distinct and
transverse. Performing these perturbations simultaneously gives simple closed
curves \(\widehat\alpha_1^X,\ldots,\widehat\alpha_N^X\), each isotopic to
\(\alpha_j^X\), having only transverse double intersections and satisfying
\(
  \#(\widehat\alpha_i^X\cap\widehat\alpha_j^X)
  =
  i(\alpha_i,\alpha_j).
\)
Hence they remain pairwise in minimal position.
The perturbations can be made arbitrarily small in the \(C^1\)-topology.
Since the length of an immersed curve is continuous under \(C^1\)-convergence,
for every \(\varepsilon>0\) they may be chosen so that
\(
  \sum_{j=1}^N
  \operatorname{Length}_X(\widehat\alpha_j^X)
  \le
  \sum_{j=1}^N
  \operatorname{Length}_X(\alpha_j^X)
  +\varepsilon.
\)
Taking \(\varepsilon=L_\Gamma(X)>0\) gives
\begin{equation}
  \sum_{j=1}^N
  \operatorname{Length}_X(\widehat\alpha_j^X)
  \le 2L_\Gamma(X).
  \label{eq:perturbed-filling-graph-length}
\end{equation}
Because the perturbed curves represent the same filling classes, their union
\(
  \widehat G_X
  :=
  \bigcup_{j=1}^N\widehat\alpha_j^X
\)
is again connected and cellularly embedded: every component of
\(X\setminus\widehat G_X\) is a disk.

\smallskip
\noindent
\emph{Step 2: finitely many marked \(\Gamma\)-configurations.}
The orientation of \(X\) gives the four half-edges at every vertex of
\(\widehat G_X\) a cyclic order, and each graph edge inherits the label of
the curve \(\widehat\alpha_j^X\) containing it. Thus \(\widehat G_X\) is a
\(\Gamma\)-labeled ribbon graph.

Only finitely many \(\Gamma\)-labeled ribbon-graph types can occur. To see
this, temporarily orient the curves. For \(i<j\), set
\(
  n_{ij}:=i(\alpha_i,\alpha_j),
\)
and give the \(n_{ij}\) intersections of
\(\widehat\alpha_i^X\) with \(\widehat\alpha_j^X\) auxiliary names
\(
  v_{ij}^1,\ldots,v_{ij}^{n_{ij}}.
\)
For each oriented curve, record the cyclic order in which its finitely many
named vertices occur, and at each vertex record the local intersection sign
of the corresponding ordered pair of oriented curves. These data determine
the labeled ribbon graph: consecutive vertices in the cyclic word along the
\(i\)-th curve determine the edges labeled \(i\), while the local
intersection sign determines the cyclic order of the four incident
half-edges. There are only finitely many cyclic words and only finitely many
sign assignments. Forgetting the auxiliary names and the temporary
orientations can only identify possibilities. Hence the set of
\(\Gamma\)-labeled ribbon-graph types is finite.

We next prove that a fixed \(\Gamma\)-labeled ribbon-graph type has only
finitely many marked \(\Gamma\)-configurations. Fix an abstract
\(\Gamma\)-labeled ribbon graph \(G\) of that type and a marked realization
\(i:G\hookrightarrow S\). Let \(i':G\hookrightarrow S\) be another marked
realization, after identifying the two abstract graphs by a label-preserving
ribbon-graph isomorphism. Choose closed regular neighborhoods
\(N(i(G))\) and \(N(i'(G))\). Because the graph identification preserves the
cyclic order at every vertex, the ribbon structure gives an
orientation-preserving homeomorphism
\(
  h_N:N(i(G))\to N(i'(G))
\)
extending the graph identification; equivalently, it identifies the
corresponding boundary cycles. This is precisely the thickening construction
associated with a ribbon graph in
\cite[Definitions~1.5--1.7]{MulasePenkava1998}.

Both embeddings are cellular, so every component of
\(S\setminus\operatorname{int}N(i(G))\) and of
\(S\setminus\operatorname{int}N(i'(G))\) is a compact disk. The restriction
of \(h_N\) to the boundary of each such disk extends over the disk, for
example after identifying the two disks with the unit disk and applying the
radial Alexander extension. Extending over all complementary disks gives an
orientation-preserving homeomorphism \(h:S\to S\) satisfying
\(
  i'=h\circ i.
\)
Since the labels are preserved and the cycle labeled \(j\) represents
\(\alpha_j\) in both realizations, one has
\(h(\alpha_j)=\alpha_j\) for every \(j\). Thus \([h]\) belongs to the
pointwise stabilizer
\[
  \operatorname{Stab}^{\mathrm{pt}}(\Gamma)
  :=
  \bigl\{
    [\varphi]\in\operatorname{Mod}(S):
    \varphi(\alpha_j)=\alpha_j
    \text{ for every }j
  \bigr\}.
\]

This stabilizer is finite. Indeed, fix \(Y_0\in\mathcal T(S)\) and let
\(
  K_0
  :=
  \bigl\{
    Y\in\mathcal T(S):
    L_\Gamma(Y)\le L_\Gamma(Y_0)
  \bigr\}.
\)
Kerckhoff's properness theorem for the length sum of a filling system implies
that \(K_0\) is compact; see
\cite[Lemma~3.1]{Kerckhoff1983}. If
\(
  [\varphi]\in\operatorname{Stab}^{\mathrm{pt}}(\Gamma),
\)
then \(\varphi^{-1}(\alpha_j)=\alpha_j\) for every \(j\). With the action
convention
\(
  [\varphi]\cdot[Z,m]=[Z,m\circ\varphi^{-1}],
\)
the length functions satisfy
\(
  \ell_{\varphi Y}(\alpha)=\ell_Y(\varphi^{-1}\alpha).
\)
Therefore
\[
  L_\Gamma(\varphi Y_0)
  =
  \sum_{j=1}^N\ell_{Y_0}(\varphi^{-1}\alpha_j)
  =
  \sum_{j=1}^N\ell_{Y_0}(\alpha_j)
  =
  L_\Gamma(Y_0).
\]
 Set
\(K:=K_0\cup\{Y_0\}\). Then
\(\varphi K\cap K\ne\varnothing\). Proper discontinuity of the
\(\operatorname{Mod}(S)\)-action on Teichm\"uller space therefore implies
that only finitely many such \([\varphi]\) exist; see
\cite[Theorem~12.2]{FarbMargalit2012}. Consequently each labeled
ribbon-graph type gives only finitely many marked
\(\Gamma\)-configurations, and there are only finitely many marked
\(\Gamma\)-configurations in total. We keep this finite list rather than
replacing it by a single model, since distinct minimal configurations can
occur; see \cite[Example~5]{HassScott1999}.

Choose representatives
\(
  G^{(1)},\ldots,G^{(A)}\subset S
\)
of these marked \(\Gamma\)-configurations. By ambient isotopy, choose in
each graph a vertex and move it to the same point \(p\in S\). Since \(S\)
is obtained from each \(G^{(a)}\) by attaching its disk faces, the inclusion
induces a surjection
\(
  \pi_1(G^{(a)},p)
  \twoheadrightarrow
  \pi_1(S,p);
\)
this also follows directly from van Kampen's theorem
\cite[Theorem~1.20]{Hatcher2002}.

\smallskip
\noindent
\emph{Step 3: uniformly short representatives of the reference edge loops.}
Let \(p_*:=m_*(p)\). Choose a finite one-vertex 2-complex
triangulation \(T_*\) of \(X_*\) with unique vertex \(p_*\). Such a
triangulation is obtained, for example, from the standard \(4g\)-gon model
of \(S\) by drawing diagonals from one vertex and transporting the resulting
cell structure by \(m_*\). Every oriented edge \(e\) of \(T_*\) is a based
loop at \(p_*\), and hence determines
\(
  \gamma_e
  :=
  (m_*)_*^{-1}[e]
  \in
  \pi_1(S,p).
\)

Let \(E(T_*)\) denote the finite set of oriented edges of
\(T_*\). For every
\(
  (a,e)\in\{1,\ldots,A\}\times E(T_*),
\)
choose a based finite combinatorial edge loop
\(\beta_e^{(a)}\subset G^{(a)}\) whose image in \(\pi_1(S,p)\) is
\(\gamma_e\). Such a loop exists by the preceding surjectivity. Define
\(\lvert\beta_e^{(a)}\rvert_{\mathrm{comb}}\) to be the number of graph
edges traversed, counted with multiplicity. The choice need not be unique.
Since there are only finitely many configurations and finitely many oriented
edges of \(T_*\), the integer
\(
  B
:=
\max_{\substack{1\le a\le A,e\in E(T_*)}}
|\beta_e^{(a)}|_{\mathrm{comb}}
\)
is finite.

Now fix \(X\), and consider the marked realization
\(m_X^{-1}(\widehat G_X)\subset S\). It belongs to one of the marked
configurations represented by \(G^{(a)}\). Replacing \(m_X\) by an isotopic
representative, we may therefore arrange that
\(
  m_X(G^{(a)})=\widehat G_X\) and \(
  m_X(p)=p_X,
\)
where \(p_X\) is the corresponding vertex. Put
\(\beta_{e,X}:=m_X(\beta_e^{(a)})\). Relative to the marking of \(X\),
this is a based loop representing \(\gamma_e\). It traverses at most \(B\)
graph edges, each of which is a subarc of one of the curves
\(\widehat\alpha_j^X\). Hence
\begin{equation}
  \operatorname{Length}_X(\beta_{e,X})
  \le
  B\sum_{j=1}^N
  \operatorname{Length}_X(\widehat\alpha_j^X)
  \le
  2B L_\Gamma(X).
  \label{eq:uniform-short-generator-loop}
\end{equation}

\smallskip
\noindent
\emph{Step 4: the equivariant map on the lifted one-skeleton.}
Let
\(
  \pi_*:\widetilde X_*\to X_*\) and \(
  \pi_X:\widetilde X\to X
\)
be the universal covering maps, and let \(\widetilde T_*\) denote the lift
of the \(\Delta\)-complex structure \(T_*\) to \(\widetilde X_*\).
Using the markings and the basepoints, identify both deck groups with
\(\pi_1(S,p)\). Choose lifts
\(
  \widetilde p_*\in\pi_*^{-1}(p_*)\) and \(
  \widetilde p_X\in\pi_X^{-1}(p_X).
\)
For \(\gamma\in\pi_1(S,p)\), write \(\gamma_*\) and \(\gamma_X\) for the
corresponding deck transformations of \(\widetilde X_*\) and
\(\widetilde X\), respectively.

The vertices of \(\widetilde T_*\) are the points
\(\gamma_*\widetilde p_*\). Define
\(
  \widetilde F_X(\gamma_*\widetilde p_*)
  :=
  \gamma_X\widetilde p_X.
\)
This is well-defined and equivariant on the vertex set. A lifted oriented
edge covering \(e\) has endpoints
\(\gamma_*\widetilde p_*\) and
\((\gamma\gamma_e)_*\widetilde p_*\). The distance between their image
points equals
\(
  d_{\widetilde X}
  \bigl(
    \widetilde p_X,
    (\gamma_e)_X\widetilde p_X
  \bigr),
\)
which is at most \(\operatorname{Length}_X(\beta_{e,X})\), since the lift
of \(\beta_{e,X}\) starting at \(\widetilde p_X\) ends at
\((\gamma_e)_X\widetilde p_X\).

Map each lifted edge at constant speed onto the unique geodesic segment
joining its image endpoints. Let \(\ell_{*,\min}>0\) be the minimum of the
lengths of the finitely many edges of \(T_*\). With respect to the intrinsic
path metric on the lifted one-skeleton, the resulting equivariant map
satisfies
\begin{equation}
  \operatorname{Lip}
  \bigl(
    \widetilde F_X|_{\widetilde T_*^{(1)}}
  \bigr)
  \le
  \frac{2B}{\ell_{*,\min}}L_\Gamma(X).
  \label{eq:lifted-one-skeleton-lipschitz}
\end{equation}

\smallskip
\noindent
\emph{Step 5: extension over the two-simplices and descent.}
For each two-simplex \(\Delta_0\) of the
\(\Delta\)-complex \(T_*\), choose once and for all an abstract
nondegenerate Euclidean model triangle
\(
  \Delta=[v_0,v_1,v_2],
\)
and designate \(v_0\) as its distinguished vertex. Let
\(
  \chi_{\Delta_0}:\Delta\to X_*
\)
be the characteristic map of \(\Delta_0\), chosen so that its restriction
to the interior of \(\Delta\) is a homeomorphism onto the interior of
\(\Delta_0\), and its restriction to each side of \(\Delta\) is the
constant-speed parametrization of the corresponding edge of \(T_*\).
We choose the characteristic maps piecewise smoothly, so that their lifts
to the universal cover are bi-Lipschitz onto the corresponding lifted
two-simplices.
For each simplex \(\Delta_0\), choose one lift
\(
  \widetilde\Delta_0
  \subset
  \widetilde T_*
\)
and let
\(
  \widetilde\chi_{\Delta_0}:
  \Delta\to\widetilde\Delta_0
\)
be the lift of \(\chi_{\Delta_0}\). On every translated simplex
\(\gamma_*\widetilde\Delta_0\), where
\(\gamma\in\pi_1(S,p)\), we use the transported parametrization
\(
  \gamma_*\circ\widetilde\chi_{\Delta_0}.
\)
Thus all parametrizations of lifted simplices are chosen equivariantly with
respect to the deck transformations. 

Let
\(
  e:=[v_1,v_2]
\)
be the side of the model triangle opposite the distinguished vertex
\(v_0\). Since \(\widetilde F_X\) has already been defined on the lifted
one-skeleton, it induces a boundary map
\[
  f_{\Delta_0,X}
  :=
  \left.
  \widetilde F_X
  \right|_{\partial\widetilde\Delta_0}
  \circ
  \left.
  \widetilde\chi_{\Delta_0}
  \right|_{\partial\Delta}
  :
  \partial\Delta
  \longrightarrow
  \widetilde X.
\]
Equip \(\partial\Delta\) with its intrinsic path metric and set
\(
  L_{\Delta_0,X}
  :=
  \operatorname{Lip}
  \bigl(
    f_{\Delta_0,X}:
    \partial\Delta\to\widetilde X
  \bigr).
\)
By the one-skeleton estimate and the fixed choice of
\(\widetilde\chi_{\Delta_0}\), one has
\[
  L_{\Delta_0,X}
  \le
  \operatorname{Lip}
  \bigl(
    \widetilde F_X|_{\widetilde T_*^{(1)}}
  \bigr)
  \operatorname{Lip}
  \bigl(
    \widetilde\chi_{\Delta_0}|_{\partial\Delta}
  \bigr)
  \le
  C_{\Delta_0}L_\Gamma(X),
\]
where \(C_{\Delta_0}\) depends only on the fixed reference simplex
\(\Delta_0\) and its chosen parametrization. Since \(T_*\) has only
finitely many two-simplices, the constants \(C_{\Delta_0}\) are uniformly
bounded.

For \(0\le r\le1\) and \(x\in e\), set
\(
  y_{r,x}:=(1-r)v_0+rx\in\Delta.
\)
Every point of \(\Delta\) can be written in this form; when \(r=0\), the
point \(y_{0,x}=v_0\) is independent of \(x\).
For \(a,b\in\widetilde X\), let \([a,b]_r\) denote the point at parameter
\(r\) on the constant-speed geodesic segment from \(a\) to \(b\). Define
\(
  \widehat f:\Delta\to\widetilde X
\)
by
\begin{equation}\label{widehat_f}
  \widehat f(y_{r,x})
  :=
  \bigl[
    f_{\Delta_0,X}(v_0),
    f_{\Delta_0,X}(x)
  \bigr]_r.
\end{equation}
This is well defined at \(r=0\), since
\(
  [a,b]_0=a
\)
for every \(b\).
One can see that
\(
  \widehat f|_{\partial\Delta}
  =
  f_{\Delta_0,X}.
\)

We now estimate the Lipschitz constant of \(\widehat f\). Since
\(\widetilde X\simeq\mathbb H^2\) is CAT\((-1)\), and hence CAT\((0)\),
convexity of distance between geodesics \cite[Proposition~II.2.2]{BridsonHaefliger1999} gives, for fixed $r$,
\[
\begin{aligned}
  d_{\widetilde X}
  \bigl(
    \widehat f(y_{r,x}),
    \widehat f(y_{r,x'})
  \bigr)
  &\le
  r\,
  d_{\widetilde X}
  \bigl(
    f_{\Delta_0,X}(x),
    f_{\Delta_0,X}(x')
  \bigr)
  \\
  &\le
  rL_{\Delta_0,X}
  d_{\partial\Delta}(x,x')
=
  rL_{\Delta_0,X}|x-x'|.
\end{aligned}
\]

For fixed \(x\), both points lie on the same constant-speed geodesic, so
\[
\begin{aligned}
  d_{\widetilde X}
  \bigl(
    \widehat f(y_{r,x}),
    \widehat f(y_{r',x})
  \bigr)
  &=
  |r-r'|\,
  d_{\widetilde X}
  \bigl(
    f_{\Delta_0,X}(v_0),
    f_{\Delta_0,X}(x)
  \bigr)
  \\
  &\le
  |r-r'|L_{\Delta_0,X}
  d_{\partial\Delta}(v_0,x).
\end{aligned}
\]
Since \(\Delta\) is fixed, the constant
\(
  D_\Delta
  :=
  \sup_{x\in e}d_{\partial\Delta}(v_0,x)
\)
is finite. Hence
\[
  d_{\widetilde X}
  \bigl(
    \widehat f(y_{r,x}),
    \widehat f(y_{r',x})
  \bigr)
  \le
  D_\Delta L_{\Delta_0,X}|r-r'|.
\]

Suppose, for example, that \(r'\le r\). Moving first radially from
\(y_{r,x}\) to \(y_{r',x}\), and then along the level \(r'\) from
\(y_{r',x}\) to \(y_{r',x'}\), gives
\[
\begin{aligned}
  d_{\widetilde X}
  \bigl(
    \widehat f(y_{r,x}),
    \widehat f(y_{r',x'})
  \bigr)
  &\le
  D_\Delta L_{\Delta_0,X}|r-r'|
  +
  r'L_{\Delta_0,X}|x-x'|.
\end{aligned}
\]
The same argument with \(r\) and \(r'\) interchanged applies when
\(r\le r'\). Consequently, after setting
\(C_\Delta:=\max\{1,D_\Delta\}\), one obtains
\begin{equation}
\label{eq:cone-extension-intermediate-bound}
  d_{\widetilde X}
  \bigl(
    \widehat f(y_{r,x}),
    \widehat f(y_{r',x'})
  \bigr)
  \le
  C_\Delta L_{\Delta_0,X}
  \left(
    |r-r'|
    +
    \min\{r,r'\}|x-x'|
  \right).
\end{equation}

It remains to compare the expression on the right with the Euclidean
distance in \(\Delta\). Since \(\Delta\) is fixed and nondegenerate, there
exists a constant \(C_\Delta'>0\) such that
\(
  |r-r'|
  +
  \min\{r,r'\}|x-x'|
  \le
  C_\Delta'
  |y_{r,x}-y_{r',x'}|.
\)
Combining with
\eqref{eq:cone-extension-intermediate-bound}, we conclude that
\(
  \operatorname{Lip}(\widehat f)
  \le
  C_\Delta''L_{\Delta_0,X},
\)
where \(C_\Delta''\) depends only on the fixed Euclidean model triangle
\(\Delta\). Using the previously established estimate
\(
  L_{\Delta_0,X}\le C_{\Delta_0}L_\Gamma(X),
\)
we further obtain
\[
  \operatorname{Lip}(\widehat f)
  \le
  C_\Delta''C_{\Delta_0}L_\Gamma(X).
\]

For each  two-simplex \(\Delta_0\) of \(T_*\), let
\(\widetilde\Delta_0\subset\widetilde T_*\) be the chosen lift and let
\(
  \widetilde\chi_{\Delta_0}:
  \Delta\to\widetilde{\Delta}_0
\)
be the fixed bi-Lipschitz parametrization introduced above. The extension
over \(\widetilde\Delta_0\) is defined by
\(
  \widetilde F_X
  |_{\widetilde\Delta_0}
  :=
  \widehat f_{\Delta_0,X}
  \circ
  \widetilde\chi_{\Delta_0}^{-1}.
\) Here, $\widehat f_{\Delta_0,X}$ is the corresponding map defined by \eqref{widehat_f}.
Since
\(\widehat f_{\Delta_0,X}|_{\partial\Delta}
=f_{\Delta_0,X}\), this extension agrees on
\(\partial\widetilde\Delta_0\) with the map already defined on the lifted
one-skeleton.
For every deck transformation
\(\gamma\in\pi_1(S,p)\), define the map on the translated simplex
\(\gamma_*\widetilde\Delta_0\) by
\(
  \widetilde F_X(\gamma_*z)
  :=
  \gamma_X\widetilde F_X(z),
\) \(
  z\in\widetilde\Delta_0.
\)
This definition is independent of the chosen representative: the deck action
on the universal cover is free, and the chosen lifted simplices are precisely
the deck translates of the finitely many selected lifts. Moreover, if two
lifted two-simplices share an edge, the two extensions agree on that edge,
because both restrict there to the previously defined equivariant
one-skeleton map. Hence the simplexwise extensions glue to a continuous
\(\pi_1(S,p)\)-equivariant map
\(
  \widetilde F_X:
  \widetilde X_*
  \to
  \widetilde X.
\)
For the chosen lift of \(\Delta_0\), the estimates above give
\[
\begin{aligned}
  \operatorname{Lip}
  \bigl(
    \widetilde F_X|_{\widetilde\Delta_0}
  \bigr)
  &\le
  \operatorname{Lip}(\widehat f_{\Delta_0,X})
  \operatorname{Lip}
  \bigl(
    \widetilde\chi_{\Delta_0}^{-1}
  \bigr)
\le
  C_{\Delta_0}'L_\Gamma(X),
\end{aligned}
\]
where \(C_{\Delta_0}'\) depends only on the fixed model triangle and the
chosen parametrization of the reference simplex \(\Delta_0\). The same
estimate holds on every deck translate, since \(\gamma_*\) and \(\gamma_X\)
act by isometries. As the complex \(T_*\) has only finitely many
two-simplices, the constants \(C_{\Delta_0}'\) admit a common upper bound.
Together with the one-skeleton estimate, this yields a constant
\(C_{\Gamma,*}>0\), independent of \(X\), such that
\begin{equation}
  \operatorname{Lip}(\widetilde F_X)
  \le
  C_{\Gamma,*}L_\Gamma(X).
  \label{Lip-tildeF}
\end{equation}
Indeed, any rectifiable path in \(\widetilde X_*\) can be subdivided into
subpaths lying in individual lifted simplices, and the preceding uniform
simplexwise estimate bounds the length of its image by
\(C_{\Gamma,*}L_\Gamma(X)\) times the length of the original path.

The equivariant map descends to a map \(F_X:X_*\to X\). Its Lipschitz
constant does not increase under passage to the quotient. More explicitly,
for \(x,y\in X_*\), choose lifts
\(\widetilde x,\widetilde y\). Equivariance gives
\[
\begin{aligned}
  d_X\bigl(F_X(x),F_X(y)\bigr)
  &=
  \inf_{\gamma\in\pi_1(S,p)}
  d_{\widetilde X}
  \bigl(
    \widetilde F_X(\widetilde x),
    \gamma_X\widetilde F_X(\widetilde y)
  \bigr)\\
  &=
  \inf_{\gamma\in\pi_1(S,p)}
  d_{\widetilde X}
  \bigl(
    \widetilde F_X(\widetilde x),
    \widetilde F_X(\gamma_*\widetilde y)
  \bigr)\\
  &\le
  C_{\Gamma,*}L_\Gamma(X)
  \inf_{\gamma\in\pi_1(S,p)}
  d_{\widetilde X_*}
  \bigl(
    \widetilde x,
    \gamma_*\widetilde y
  \bigr)\\
  &=
  C_{\Gamma,*}L_\Gamma(X)
  d_{X_*}(x,y),
\end{aligned}
\]
where the inequality follows from \eqref{Lip-tildeF}.
This proves \eqref{eq:lipschitz-marking-bound}.

Finally, \(F_X(p_*)=p_X\), and the defining equivariance shows that
\(
  (F_X)_*\circ(m_*)_*
  =
  (m_X)_*
\)
on \(\pi_1(S,p)\). Since a closed hyperbolic surface is a \(K(\pi,1)\),
two based maps into \(X\) inducing the same homomorphism on fundamental
groups are based homotopic; see
\cite[Proposition~1B.9]{Hatcher2002}. Hence
\(
  F_X\circ m_*
  \simeq
  m_X,
\)
so \(F_X\) is marking-compatible.
\end{proof}

\begin{proposition}
\label{prop:energy-filling-length-comparison}
Under the filling hypothesis above, there exist constants
\(0<A_{\mathfrak F}\le B_{\mathfrak F}<\infty\), depending only on the
fixed domain \((M,g)\), the homotopy class of \(u_0\), the chosen filling
data \(\mathfrak F\), and the fixed reference marked surface
\((X_*,m_*)\), such that
\begin{equation}
  A_{\mathfrak F}\Lambda_\Gamma(X)
  \le
  \mathcal E_{u_0}(X)
  \le
  B_{\mathfrak F}\Lambda_\Gamma(X)
  \qquad
  \text{for every }X\in\mathcal T(S).
  \label{eq:energy-filling-length-comparison}
\end{equation}
\end{proposition}

\begin{proof}
We prove the two inequalities separately.

\smallskip
\noindent
\emph{Lower bound.}
Let \(u_X:(M,g)\to(X,\sigma_X)\) be the harmonic representative of the
prescribed marked homotopy class, and put
$e_X:=|du_X|^2=|du_X|^2_{\mathrm{HS}}$. Choose $\kappa\ge 0$ such that
\(\operatorname{Ric}_g\ge-\kappa g\). With the convention
\(\Delta_g=\operatorname{div}_g\nabla\), the Bochner formula for harmonic
maps gives
\begin{equation}
\begin{aligned}
  \frac12\Delta_g e_X
  =&
  |\nabla du_X|^2
  +
  \sum_{i=1}^n
  \bigl\langle
    du_X(\operatorname{Ric}_g e_i),du_X(e_i)
  \bigr\rangle_{\sigma_X}
  \\
  &-
  \sum_{i,j=1}^n
  \bigl\langle
    R^{\sigma_X}(du_X(e_i),du_X(e_j))du_X(e_j),
    du_X(e_i)
  \bigr\rangle_{\sigma_X},
  \label{eq:bochner-energy-density-prop44}
\end{aligned}
\end{equation}
where \(e_1,\ldots,e_n\) is a local \(g\)-orthonormal frame; compare
\cite[(9.2.13)]{Jost2017}. At a fixed point, choose the frame so
that it diagonalizes the Ricci endomorphism. Since all of its eigenvalues are
at least \(-\kappa\), the Ricci term in
\eqref{eq:bochner-energy-density-prop44} is bounded below by
\(-\kappa e_X\). Moreover, because \((X,\sigma_X)\) has curvature \(-1\),
for tangent vectors \(A,B\in TX\) one has
\(
  -\bigl\langle R^{\sigma_X}(A,B)B,A\bigr\rangle
  =
  |A|^2|B|^2-\langle A,B\rangle^2
  \ge0.
\)
Hence
\begin{equation}
  \Delta_g e_X\ge-2\kappa e_X.
  \label{eq:scalar-subsolution-energy-density}
\end{equation}

The standard local boundedness estimate for nonnegative subsolutions of
uniformly elliptic equations, applied to
\(
  (\Delta_g+2\kappa)e_X\ge0,
\)
gives
\[
  \|e_X\|_{L^\infty(M)}
  \le
  C(M,g,\kappa)\int_Me_X\,d\mu_g,
\]
see \cite[Theorem~8.17]{GilbargTrudinger2001}. Here we use a fixed finite
coordinate cover of \(M\), so the constant depends only on the fixed
operator \(\Delta_g+2\kappa\) and is independent of \(X\). Since
\(
  \int_Me_X\,d\mu_g=2\mathcal E_{u_0}(X),
\)
we obtain
\begin{equation}
  \|du_X\|_{L^\infty(M)}^2
  \le
  C_M\mathcal E_{u_0}(X).
  \label{eq:gradient-bound-energy}
\end{equation}
Here \( \|du_X\|_{L^\infty(M)}^2\)
is the $L^\infty$-norm of the
function $|du_X|^2 =|du_X|^2_{HS}$.

For every \(j\), let \(\beta_j:S^1\to M\) be the fixed piecewise smooth
loop belonging to the filling data. Since
\(u_X\simeq m_X\circ u_0\), the loop \(u_X\circ\beta_j\) represents the
marked free homotopy class \(m_X(\alpha_j)\). Therefore
\[
  \ell_X(\alpha_j)
  \le
  \operatorname{Length}_X(u_X\circ\beta_j).
\]
Using first the operator norm and then the Hilbert--Schmidt norm of the
differential, we obtain
\[
  \operatorname{Length}_X(u_X\circ\beta_j)
  \le
  \int_{S^1}
  |du_X(\dot\beta_j)|_{\sigma_X}\,ds
  \le
  \|du_X\|_{L^\infty(M)}
  \operatorname{Length}_g(\beta_j).
\]
After squaring, summing over \(j\), and applying
\eqref{eq:gradient-bound-energy}, we find
\[
  \Lambda_\Gamma(X)
  \le
  C_M
(
    \sum_{j=1}^N\operatorname{Length}_g(\beta_j)^2
 )
  \mathcal E_{u_0}(X).
\]
Thus the lower bound holds with
\(
  A_{\mathfrak F}^{0}
  :=
  [
    C_M
    \sum_{j=1}^N\operatorname{Length}_g(\beta_j)^2
  ]^{-1}>0.
\)

\smallskip
\noindent
\emph{Upper bound.}
Let \(v_*:=u_{X_*}:M\to X_*\) be the harmonic representative at the
reference surface. Then \(v_*\simeq m_*\circ u_0\). By
Lemma~\ref{lem:lipschitz-marking-filling-lengths}, for every marked surface
\((X,m_X)\) there is a marking-compatible Lipschitz map
\(F_X:X_*\to X\) such that
\begin{equation}
  \operatorname{Lip}(F_X)
  \le
  C_{\Gamma,*}L_\Gamma(X).
  \label{eq:lip-FX-bound-used}
\end{equation}
The relevant maps fit into the following homotopy-commutative diagram:
\[
\begin{tikzcd}[column sep=5.2em,row sep=3.2em]
  &
  S
  \arrow[d,"m_*" description]
  \arrow[dr,"m_X"]
  &
  \\
  M
  \arrow[ur,"u_0"]
  \arrow[r,swap,"v_*:=u_{X_*}"]
  \arrow[rr,bend right=18,swap,"w_X:=F_X\circ v_*"]
  &
  X_*
  \arrow[r,swap,"F_X"]
  &
  X.
\end{tikzcd}
\]Indeed,
\(
  F_X\circ v_*
  \simeq
  F_X\circ m_*\circ u_0
  \simeq
  m_X\circ u_0.
\)
Hence the Lipschitz map
\(
  w_X:=F_X\circ v_*:M\to X
\)
lies in the same classical homotopy class as \(u_X\).

Since \(w_X\) is Lipschitz, it has finite energy. By White's
least-energy theorem for maps into negatively curved targets
\cite[Theorem~4]{White1985}, see also \cite{White1988} for the complete proof, the harmonic representative
\(u_X\) minimizes the  energy among Lipschitz maps in its classical
homotopy class. Since \(w_X\simeq u_X\), it follows that
\begin{equation}
  \mathcal E_{u_0}(X)
  =
  E_X(u_X)
  \le
  E_X(w_X).
  \label{eq:harmonic-energy-below-lipschitz-competitor}
\end{equation}

We now make explicit how \(\operatorname{Lip}(F_X)\) enters the energy
estimate. By Rademacher's theorem in local coordinates
\cite[Section~3.1.2]{EvansGariepy2015}, the Lipschitz map \(w_X\) is
differentiable almost everywhere. At every such point \(x\), for every
\(\xi\in T_xM\), the defining Lipschitz inequality for \(F_X\), applied
along a smooth curve tangent to \(\xi\in T_xM\)
gives
\(
  |dw_X(\xi)|_{\sigma_X}
  \le
  \operatorname{Lip}(F_X)
  |dv_*(\xi)|_{\sigma_{X_*}},
  \)
  almost everywhere in $x$.
If \(e_1,\ldots,e_n\) is a \(g\)-orthonormal basis of \(T_xM\), then
\[
\begin{aligned}
  |dw_X|^2 = |dw_X|_{HS}^2
  &=
  \sum_{i=1}^n|dw_X(e_i)|_{\sigma_X}^2
  \\
  &\le
  \operatorname{Lip}(F_X)^2
  \sum_{i=1}^n|dv_*(e_i)|_{\sigma_{X_*}}^2
  \\
  &=
  \operatorname{Lip}(F_X)^2
  |dv_*|^2
\end{aligned}
\]
almost everywhere. Integrating this pointwise inequality yields the
intermediate estimate
\begin{equation}
\begin{aligned}
  E_X(w_X)
  &=
  \frac12
  \int_M
  |dw_X|^2\,d\mu_g
  \\
  &\le
  \operatorname{Lip}(F_X)^2
  \frac12
  \int_M
  |dv_*|^2\,d\mu_g
  \\
  &=
  \operatorname{Lip}(F_X)^2E_{X_*}(v_*).
  \label{eq:energy-lipschitz-postcomposition}
\end{aligned}
\end{equation}
Combining
\eqref{eq:harmonic-energy-below-lipschitz-competitor},
\eqref{eq:energy-lipschitz-postcomposition}, and
\eqref{eq:lip-FX-bound-used}, we obtain
\[
  \mathcal E_{u_0}(X)
  \le
  C_{\Gamma,*}^2
  E_{X_*}(v_*)
  L_\Gamma(X)^2.
\]
Finally, the Cauchy--Schwarz inequality gives
\[
  L_\Gamma(X)^2
  =
  (\sum_{j=1}^N\ell_X(\alpha_j))^2
  \le
  N\sum_{j=1}^N\ell_X(\alpha_j)^2
  =
  N\Lambda_\Gamma(X).
\]
Hence
\[
  \mathcal E_{u_0}(X)
  \le
  B_{\mathfrak F}^{0}\Lambda_\Gamma(X),
  \qquad
  B_{\mathfrak F}^{0}
  :=
  NC_{\Gamma,*}^2E_{X_*}(v_*)
  =
  NC_{\Gamma,*}^2\mathcal E_{u_0}(X_*).
\]

The desired statement follows by taking
\(
  A_{\mathfrak F}:=A_{\mathfrak F}^{0}\) and \(
  B_{\mathfrak F}
  :=
  \max\{
    A_{\mathfrak F}^{0},
    B_{\mathfrak F}^{0}
  \}.
\)
This guarantees
\(0<A_{\mathfrak F}\le B_{\mathfrak F}<\infty\) and completes the proof.
\end{proof}\begin{theorem}
\label{thm:energy-quasi-convexity}
Assume that the homotopy class of $u_0$ satisfies the filling hypothesis
stated above. Then there exists $K_E\ge1$, depending only on
$(M,g,u_0,\mathfrak F)$ and $S$, such that, for every Teichm\"uller geodesic
$\gamma\colon\mathbb R\to\mathcal T(S)$ and every $s<t<r$,
\begin{equation}
  \mathcal E_{u_0}(\gamma(t))
  \le
  K_E
  \max\left\{
    \mathcal E_{u_0}(\gamma(s)),
    \mathcal E_{u_0}(\gamma(r))
  \right\}.
  \label{eq:energy-quasi-convexity}
\end{equation}
In particular, the conclusion holds if $(u_0)_*$ is surjective.
\end{theorem}

\begin{proof}
By the theorem of Lenzhen--Rafi, there is a constant $K_S\ge1$, depending
only on the topological type of $S$, such that
\[
  \ell_{\gamma(t)}(\alpha_j)
  \le
  K_S\max\left\{
    \ell_{\gamma(s)}(\alpha_j),
    \ell_{\gamma(r)}(\alpha_j)
  \right\}
\]
for every $j$. It follows that
\begin{align*}
  \Lambda_\Gamma(\gamma(t))
  &\le
  K_S^2
  \sum_{j=1}^N
  \max\left\{
    \ell_{\gamma(s)}(\alpha_j)^2,
    \ell_{\gamma(r)}(\alpha_j)^2
  \right\}
  \\
  &\le
  K_S^2
  \left(
    \Lambda_\Gamma(\gamma(s))
    +
    \Lambda_\Gamma(\gamma(r))
  \right)
  \\
  &\le
  2K_S^2
  \max\left\{
    \Lambda_\Gamma(\gamma(s)),
    \Lambda_\Gamma(\gamma(r))
  \right\}
\end{align*}
for any $s<t<r$.
Applying Proposition~\ref{prop:energy-filling-length-comparison} at the
three points gives
\[
  \mathcal E_{u_0}(\gamma(t))
  \le
  2\frac{B_{\mathfrak F}}{A_{\mathfrak F}}K_S^2
  \max\left\{
    \mathcal E_{u_0}(\gamma(s)),
    \mathcal E_{u_0}(\gamma(r))
  \right\}.
\]
Thus one may take
\(
  K_E
  =
  2\frac{B_{\mathfrak F}}{A_{\mathfrak F}}K_S^2.
\)
\end{proof}

\begin{corollary}
\label{cor:log-energy-additive-quasi-convexity}
Under the hypotheses of Theorem~\ref{thm:energy-quasi-convexity},
$\log\mathcal E_{u_0}$ is additively quasi-convex along Teichm\"uller
geodesics:
\begin{equation}
  \log\mathcal E_{u_0}(\gamma(t))
  \le
  \max\left\{
    \log\mathcal E_{u_0}(\gamma(s)),
    \log\mathcal E_{u_0}(\gamma(r))
  \right\}
  +
  \log K_E.
  \label{eq:log-energy-additive-quasi-convexity}
\end{equation}
Moreover, for every $p>0$,
\begin{equation}
  \mathcal E_{u_0}(\gamma(t))^p
  \le
  K_E^p
  \max\left\{
    \mathcal E_{u_0}(\gamma(s))^p,
    \mathcal E_{u_0}(\gamma(r))^p
  \right\}.
  \label{eq:positive-powers-energy-quasi-convex}
\end{equation}
\end{corollary}

\begin{remark}
The filling hypothesis is a sufficient condition used to obtain the uniform
comparison \eqref{eq:energy-filling-length-comparison}; it is not asserted to
be necessary. For instance, the circle energy
\eqref{eq:circle-energy-length-square} is quasi-convex by the Lenzhen--Rafi
theorem even though its fundamental-group image is cyclic.

On the other hand, quasi-convexity cannot in general be replaced by convexity.
Lenzhen and Rafi construct Teichm\"uller geodesics for which a hyperbolic
length function has a positive average slope on one interval and an arbitrarily
small average slope on a later interval \cite[Example~24]{LenzhenRafi2011}.
The relevant length function stays in a fixed compact subinterval of
\((0,\infty)\) on the later intervals, whose lengths tend to infinity. Hence
the same secant-slope argument applies to every positive power
\(\ell_X(\alpha)^p\), and in particular to \(\ell_X(\alpha)^2\). By
\eqref{eq:circle-energy-length-square}, the corresponding circle energy is
therefore not convex along every Teichm\"uller geodesic. Thus multiplicative
quasi-convexity, or equivalently additive quasi-convexity after taking the
logarithm, is the natural general global statement.
\end{remark}

\section{Quasi-convexity of the energy function with varying domain}
\label{subsec:quasi-convexity-energy-varying-domain}

In the preceding section, the domain was fixed while the hyperbolic structure
on the target varied. {We now consider the complementary case
introduced in
Theorem
\ref{thm:intro-varying-domain-quasiconvexity} and Subsection
\ref{subsec:def-energy-varying-domain}: $u_t: X_t\to S$
is a harmonic covering map with the target
 fixed and the conformal structure on the domain varying.}
 Unless explicitly
stated otherwise, all coarse estimates after Assumption~\ref{assump:-covering}
use the covering assumption.

We recall that a positive function $G$ on $\mathcal T(\Sigma)$ is
\emph{multiplicatively $K$-quasi-convex along Teichm\"uller geodesics} if,
whenever $X_a,X_b,X_c$ occur in this order on a Teichm\"uller geodesic,
\[
  G(X_b)
  \le
  K\max\{G(X_a),G(X_c)\}.
\]
We now prove that the covering-map energy has this property.

The first ingredient is an exponential upper bound.

\begin{proposition}
\label{prop:covering-energy-upper-bound}
Let  $Y\in T(\Sigma)$
  be as in (\ref{eq:def-covering-minimizer-Y}).
For any $X\in\mathcal T(\Sigma)$,
\begin{equation}
  E_{u_0}(X)
  \le
  A_Ye^{2d_{\mathrm T}(X,Y)}.
  \label{eq:covering-energy-upper-exponential}
\end{equation}
\end{proposition}

\begin{proof}
Set \(T:=d_{\mathrm T}(X,Y)\), and let
\(X_t=\gamma_q(t)\), \(0\le t\le T\), be the unit-speed Teichm\"uller
geodesic with \(X_0=Y\) and \(X_T=X\). Let \(v_t=u_{X_t}\circ F_t\) as
above. Since \(v_T\) minimizes \(\mathcal E(T,\cdot)\) in the prescribed
homotopy class, \(v_0\) may be used as a competitor at time \(T\). Thus
\begin{align*}
  E_{u_0}(X)
  &=
  \mathcal E(T,v_T)
  \le
  \mathcal E(T,v_0)
  \\
  &=
  e^{-2T}H(v_0)+e^{2T}V(v_0)
  \\
  &\le
  e^{2T}\bigl(H(v_0)+V(v_0)\bigr)
  =
  A_Ye^{2T},
\end{align*}
where $\mathcal{E}(t,v)$ is defined by \eqref{eq:fixed-map-energy-teich-geodesic}.
This is \eqref{eq:covering-energy-upper-exponential}.
\end{proof}

For the lower bound, let \(\mathcal{ML}(\Sigma)\) denote the space of
measured laminations on \(\Sigma\). A measured lamination is a geodesic
lamination equipped with a locally finite transverse measure, considered up
to measure-preserving isotopy. Weighted essential simple closed curves
\(a\alpha\), with \(a>0\), form a dense subset of
\(\mathcal{ML}(\Sigma)\). For every
\(Y\in\mathcal T(\Sigma)\), the rule
\(
  \ell_Y(a\alpha)
  :=
  a\,\ell_Y(\alpha)
\)
extends uniquely to a continuous homogeneous function
\(
  \ell_Y:
  \mathcal{ML}(\Sigma)
  \to
  \mathbb R_{\ge0};
\)
see \cite{Bonahon1988}.

Using the standard correspondence between measured laminations and measured
foliations, extremal length likewise extends continuously from weighted
essential simple closed curves to
\(\mathcal{ML}(\Sigma)\); see \cite{Kerckhoff1980} and
\cite[Sections~2.1--2.2]{LiuSu2017}. For an essential simple closed curve
\(\alpha\), its extremal length on \(X\in\mathcal T(\Sigma)\) is
\[
  \operatorname{Ext}_X(\alpha)
  :=
  \sup_{\varrho}
  \frac{L_{\varrho}(\alpha)^2}
       {\operatorname{Area}_{\varrho}(X)},
\]
where \(\varrho\) ranges over all conformal metrics  satisfying
\(0<\operatorname{Area}_{\varrho}(X)<\infty\). Locally,
\(
  ds_{\varrho}
  =
  \varrho(z)|dz|,
\) \(
  ds_{\varrho}^2
  =
  \varrho(z)^2|dz|^2
\) and \(
  \operatorname{Area}_{\varrho}(X)
  =
  \int_X\varrho(z)^2\,dx\,dy.
\)
Here \(L_{\varrho}(\alpha)\) denotes the infimum of the
\(\varrho\)-lengths of representatives of the free homotopy class of
\(\alpha\). For \(a>0\), set
\(
  \operatorname{Ext}_X(a\alpha)
  :=
  a^2\operatorname{Ext}_X(\alpha).
\)
For \(\lambda\in\mathcal{ML}(\Sigma)\), we write
\(\ell_Y(\lambda)\) and \(\operatorname{Ext}_X(\lambda)\) for the resulting
continuous extensions.

\begin{lemma}
\label{lem:energy-extremal-length-inequality}
For every \(X\in\mathcal T(\Sigma)\) and every
\(\lambda\in\mathcal{ML}(\Sigma)\),
\begin{equation}
  \ell_Y(\lambda)^2
  \le
  2E_{u_0}(X)\operatorname{Ext}_X(\lambda).
  \label{eq:energy-extremal-length-inequality}
\end{equation}
\end{lemma}

\begin{proof} The assertion is immediate for \(\lambda=0\). By Lemma~\ref{lem:covering-factorization}, \(E_{u_0}(X)=E(f_X)\), where \( f_X:(X,\sigma_X)\to(Y,\sigma_Y) \) is the harmonic diffeomorphism in the identity homotopy class. There is
  the natural conformal
   metric
  \( \sigma'_X:=|df_X|^2\sigma_X=\operatorname{tr}_{\sigma_X}\left(f_X^* \sigma_Y\right)\sigma_X. \)
 Its area is \begin{equation} 
 \operatorname{Area}_{\sigma'_X}(X) = \int_X|df_X|^2\,dA_{\sigma_X} = 2E(f_X) = 2E_{u_0}(X). \label{eq:pullback-pseudometric-area} 
 \end{equation}
 Let \(\alpha\) be an essential simple closed curve and \(c\) any smooth representative of its free homotopy class. Since \(f_X\) is homotopic to the identity, \( \ell_Y(\alpha) \le \operatorname{Length}_Y(f_X\circ c). \) Moreover, \[ |df_X(c')|_{\sigma_Y} \le |df_X|_{\mathrm{op}}|c'|_{\sigma_X} \le
   |df_X|_{\mathrm{HS}}|c'|_{\sigma_X}
   =|df_X||c'|_{\sigma_X}
      |c'|_{\sigma_X}
   =|c'|_{\sigma'_X}. \] Here \(|df_X|_{\mathrm{op}}\) denotes the operator norm of \(df_X\) with respect to \(\sigma_X\) and \(\sigma_Y\). Taking the infimum over \(c\) gives \( \ell_Y(\alpha)\le L_{\sigma'_X}(\alpha). \) The same inequality holds for every weighted simple closed curve \(a\alpha\), \(a>0\). By the analytic definition of extremal length, \( L_{\sigma'_X}(a\alpha)^2\le \operatorname{Ext}_X(a\alpha)\operatorname{Area}_{\sigma'_X}(X). \) Consequently, \[ \ell_Y(a\alpha)^2 \le 2E_{u_0}(X)\operatorname{Ext}_X(a\alpha). \] Weighted simple closed curves are dense in \(\mathcal{ML}(\Sigma)\), and both \(\lambda\mapsto\ell_Y(\lambda)\) and \(\lambda\mapsto\operatorname{Ext}_X(\lambda)\) are continuous \cite{Bonahon1988,Kerckhoff1980}. Passing to the limit proves \eqref{eq:energy-extremal-length-inequality} for every measured lamination. \end{proof}

Define
\begin{equation}
  \eta_Y
  :=
  \min_{[\lambda]\in\mathbb P\mathcal{ML}(\Sigma)}
  \frac{\ell_Y(\lambda)^2}
       {\operatorname{Ext}_Y(\lambda)}.
  \label{eq:def-eta-Y}
\end{equation}
The quotient is homogeneous of degree zero and is a positive continuous
function on the compact space $\mathbb P\mathcal{ML}(\Sigma)$
\cite{Bonahon1988,Kerckhoff1980}. Therefore
\begin{equation}
  \eta_Y>0.
  \label{eq:eta-Y-positive}
\end{equation}

\begin{proposition}
\label{prop:covering-energy-lower-bound}
For every $X\in\mathcal T(\Sigma)$,
\begin{equation}
  E_{u_0}(X)
  \ge
  \frac{\eta_Y}{2}e^{2d_{\mathrm T}(X,Y)}.
  \label{eq:covering-energy-lower-exponential}
\end{equation}
\end{proposition}

\begin{proof}
For every nonzero measured lamination $\lambda$, Lemma
\ref{lem:energy-extremal-length-inequality} and the definition of $\eta_Y$
give
\[
  E_{u_0}(X)
  \ge
  \frac{\ell_Y(\lambda)^2}
       {2\operatorname{Ext}_X(\lambda)}
  \ge
  \frac{\eta_Y}{2}
  \frac{\operatorname{Ext}_Y(\lambda)}
       {\operatorname{Ext}_X(\lambda)}.
\]
Taking the supremum over $\lambda\ne0$ and using Kerckhoff's extremal-length
formula \cite[Theorem 4]{Kerckhoff1980}, one obtains
\begin{equation}
  e^{2d_{\mathrm T}(X,Y)}
  =
  \sup_{\lambda\ne0}
  \frac{\operatorname{Ext}_Y(\lambda)}
       {\operatorname{Ext}_X(\lambda)},
  \label{eq:kerckhoff-formula-used}
\end{equation}
which proves the assertion.
\end{proof}

Combining the two estimates yields the central coarse comparison of this
section.

\begin{theorem}
\label{thm:covering-energy-distance-comparison}
Under Assumption~\ref{assump:-covering},
\begin{equation}
  \frac{\eta_Y}{2}e^{2d_{\mathrm T}(X,Y)}
  \le
  E_{u_0}(X)
  \le
  A_Ye^{2d_{\mathrm T}(X,Y)}
  \label{eq:covering-energy-two-sided-exponential}
\end{equation}
for every $X\in\mathcal T(\Sigma)$. Equivalently,
\begin{equation}
  \frac12\log E_{u_0}(X)
  =
  d_{\mathrm T}(X,Y)+O_Y(1).
  \label{eq:half-log-energy-coarse-distance}
\end{equation}
More explicitly,
\begin{equation}
  d_{\mathrm T}(X,Y)+\frac12\log\frac{\eta_Y}{2}
  \le
  \frac12\log E_{u_0}(X)
  \le
  d_{\mathrm T}(X,Y)+\frac12\log A_Y.
  \label{eq:half-log-energy-explicit-bounds}
\end{equation}
\end{theorem}

Lenzhen and Rafi proved that Teichm\"uller balls are uniformly quasi-convex:
there exists a constant $c_\Sigma\ge0$, depending only on the topological type
of $\Sigma$, such that, for every $Z\in\mathcal T(\Sigma)$ and every
Teichm\"uller geodesic segment with endpoints in $B_{\mathrm T}(Z,R)$, the
entire segment is contained in $B_{\mathrm T}(Z,R+c_\Sigma)$
\cite[Theorem C]{LenzhenRafi2011}. Applying this result with center $Y$ gives the desired
quasi-convexity of the energy.

\begin{theorem}
\label{thm:covering-energy-quasi-convexity}
Let $X_a,X_b,X_c$ occur in this order on a Teichm\"uller geodesic. Then
\begin{equation}
  E_{u_0}(X_b)
  \le
  K_E
  \max\bigl\{E_{u_0}(X_a),E_{u_0}(X_c)\bigr\},
  \label{eq:covering-energy-multiplicative-quasi-convexity}
\end{equation}
where one may take
\(
  K_E
  :=
  \frac{2A_Y}{\eta_Y}e^{2c_\Sigma}.
\)
Consequently, $E_{u_0}$ is multiplicatively quasi-convex along every
Teichm\"uller geodesic.
\end{theorem}

\begin{proof}
Set
\(
  r:=\max\{d_{\mathrm T}(X_a,Y),d_{\mathrm T}(X_c,Y)\}.
\)
The quasi-convexity of Teichm\"uller balls gives
\(
  d_{\mathrm T}(X_b,Y)\le r+c_\Sigma.
\)
Using first the upper bound and then the lower bound in
Theorem~\ref{thm:covering-energy-distance-comparison}, we obtain
\begin{align*}
  E_{u_0}(X_b)
  &\le
  A_Ye^{2d_{\mathrm T}(X_b,Y)}
  \\
  &\le
  A_Ye^{2c_\Sigma}
  \max\left\{
    e^{2d_{\mathrm T}(X_a,Y)},
    e^{2d_{\mathrm T}(X_c,Y)}
  \right\}
  \\
  &\le
  \frac{2A_Y}{\eta_Y}e^{2c_\Sigma}
  \max\bigl\{E_{u_0}(X_a),E_{u_0}(X_c)\bigr\}.
\end{align*}
The proof is complete.
\end{proof}

The natural additive formulation is obtained by taking one half of the
logarithm:
\begin{equation}
  \frac12\log E_{u_0}(X_b)
  \le
  \max\left\{
    \frac12\log E_{u_0}(X_a),
    \frac12\log E_{u_0}(X_c)
  \right\}
  +\frac12\log K_E.
  \label{eq:half-log-energy-additive-quasi-convexity}
\end{equation}
Thus $\frac12\log E_{u_0}$ is additively quasi-convex and, by
\eqref{eq:half-log-energy-coarse-distance}, is at uniformly bounded distance
from the function $d_{\mathrm T}(\,\cdot\,,Y)$. Moreover, for every $k>0$,
\begin{equation}
  E_{u_0}(X_b)^k
  \le
  K_E^k
  \max\bigl\{E_{u_0}(X_a)^k,E_{u_0}(X_c)^k\bigr\}.
  \label{eq:domain-positive-powers-energy-quasi-convex}
\end{equation}
Hence every positive power of the energy is quasi-convex, although this does
not imply genuine convexity.

We conclude by recording why no universal scalar transformation can yield a
global convexity theorem in the broader class of harmonic-map energies. The
next obstruction is already present on a closed surface of genus two.

\begin{lemma}
\label{lem:lr-secants-extremal-length}
There exists a closed surface \(\Sigma_2\) of genus two, unit-speed
Teichm\"uller geodesics
\(
  X_t^{(n)}\in\mathcal T(\Sigma_2),
\)
simple closed curves \(\alpha_n\), and intervals
\(I_n=[r_n,s_n]\) with
\(
  0<r_n<s_n,
  s_n-r_n\to\infty,
\)
such that
\[
  \operatorname{Ext}_{X_0^{(n)}}(\alpha_n)\ge \frac13,
  \qquad
  \operatorname{Ext}_{X_{-2}^{(n)}}(\alpha_n)\le e^{-4}.
\]
Moreover, there are constants \(0<m<M<\infty\), independent of \(n\), such
that
\[
  m
  \le
  \operatorname{Ext}_{X_t^{(n)}}(\alpha_n)
  \le
  M
\quad (t\in I_n).
\]
\end{lemma}

\begin{proof}
  For every sufficiently small \(a>0\), Lenzhen and Rafi construct in
  \cite[Example~24]{LenzhenRafi2011} a quadratic differential on a genus-two
  surface, its Teichm\"uller geodesic \(X_t(a)\), and a simple closed curve
  \(\alpha(a)\). We identify all underlying marked genus-two surfaces with a
  fixed topological surface \(\Sigma_2\). Their estimates (14) and (15) give
\[
  \operatorname{Ext}_{X_0(a)}(\alpha(a))\ge\frac13,
  \qquad
  \operatorname{Ext}_{X_{-2}(a)}(\alpha(a))\le e^{-4}.
\]
Put \(L(a):=\frac12\log(a^{-2})=-\log a\). For \(0<t<L(a)\), their
estimates (17) and (18) give a constant \(m>0\), independent of \(a\) and
\(t\), such that
\[
  m
  \le
  \operatorname{Ext}_{X_t(a)}(\alpha(a))
  \le
  \frac{\pi}{\log2}.
\]
Choose \(a_n\downarrow0\) with \(L(a_n)>3\), and set
\[
  X_t^{(n)}:=X_t(a_n),
  \qquad
  \alpha_n:=\alpha(a_n),
  \qquad
  I_n:=[1,L(a_n)-1].
\]
Then \(|I_n|=L(a_n)-2\to\infty\), and one may take
\(M=\pi/\log2\).
\end{proof}

\begin{proposition}
\label{prop:no-universal-convexifying-transform}
For the genus-two surface \(\Sigma_2\) considered in
Lemma~\ref{lem:lr-secants-extremal-length}, there is no strictly increasing
function \(F:(0,\infty)\to\mathbb R\) such that
\(
  t\mapsto F\bigl(\operatorname{Ext}_{X_t}(\lambda)\bigr)
\)
is convex for every \(\lambda\in\mathcal{ML}(\Sigma_2)\) and every
Teichm\"uller geodesic \(X_t\subset\mathcal T(\Sigma_2)\). Consequently,
after allowing equivariant harmonic maps to
$\mathbb R$-trees as targets, there is no universal strictly increasing
transformation $F$ that makes all harmonic-map energy functions convex along
all Teichm\"uller geodesics. In particular, no fixed power $F(s)=s^k$,
$k>0$, has this property.
\end{proposition}

\begin{proof}
Use Lemma~\ref{lem:lr-secants-extremal-length}. If
\(F\circ\operatorname{Ext}_{X_t^{(n)}}(\alpha_n)\) were convex for every
\(n\), then the secant slopes over ordered disjoint intervals would be
nondecreasing. On the fixed interval \([-2,0]\), the secant slope is bounded
below by
\(
  \frac{1}{2}(F(1/3)-F(e^{-4}))>0.
\)
On the later interval \(I_n\), the values of
\(\operatorname{Ext}_{X_t^{(n)}}(\alpha_n)\) stay in a compact interval
\([m,M]\subset(0,\infty)\). Hence the absolute value of the secant slope on
\(I_n\) is at most
\(
  (F(M)-F(m))/(s_n-r_n),
\)
which tends to \(0\). This contradicts the monotonicity of ordered secant slopes
for a convex function.

Finally, identify a measured lamination with its associated measured
foliation, and let \(T_\lambda\) be the dual \(\mathbb R\)-tree with its
standard \(\pi_1(\Sigma_2)\)-action. With the normalization in
\cite[Proposition~2.5]{LiuSu2017}, the energy of the equivariant harmonic map
\(\widetilde X\to T_\lambda\) equals
\(\operatorname{Ext}_X(\lambda)\). Thus the preceding examples directly rule
out a universal convexifying transformation in the class allowing
equivariant harmonic maps to \(\mathbb R\)-trees. If another convention gives
\(E(X,T_\lambda)=c\operatorname{Ext}_X(\lambda)\) for a fixed \(c>0\), then
  \(s\mapsto F(cs)\) is again strictly increasing, so the same
contradiction applies.
\end{proof}

\begin{remark}
Proposition~\ref{prop:no-universal-convexifying-transform} uses singular
$\mathbb R$-tree targets. It therefore does not rule out a special convexity
statement restricted to the much narrower class consisting of a fixed smooth
hyperbolic target and homotopy classes of coverings. What it does
show is that neither the variational identity
\(
  E''=4E-\mathcal Q_t
\)
nor the general principle that energy extends length or extremal length can,
by itself, produce a universal power-convexity theorem. For the covering-map
class considered here, the global conclusion established above is the
multiplicative quasi-convexity of $E_{u_0}$, or equivalently the additive
quasi-convexity of $\frac12\log E_{u_0}$.
\end{remark}

\section{Energy asymptotics along a Jenkins--Strebel ray}
\label{Energy asymptotics}

For positive functions \(A(t)\) and \(B(t)\), we write
\(A(t)\asymp B(t)\) as \(t\to\infty\) if there are constants
\(c,C>0\) and \(t_0\geq0\) such that
\(cB(t)\leq A(t)\leq CB(t)\) for every \(t\geq t_0\). 

Recall that a nonzero holomorphic quadratic differential \(q\) on a closed
Riemann surface \(X_0\) is \emph{Jenkins--Strebel} if every nonsingular
trajectory of its vertical foliation is closed. Equivalently, the complement
of its critical graph is a finite disjoint union of maximal flat cylinders
\(C_1,\ldots,C_r\), each foliated by closed vertical trajectories;
see \cite{Strebel1984}. 
Let \(\delta_i\) be  the core curves of \(C_i\).
Then they  are pairwise
nonisotopic; see
\cite[Chapter~I, \S3, p.~225]{HubbardMasur1979}. 
 We use the vertical foliation here;  with the opposite
horizontal convention,  \(q\) is replaced by \(-q\).

The two energy functions considered in this paper have markedly different
growth along Teichm\"uller rays.

\begin{proposition}
\label{cor:asymptotic-growth-two-energy-settings}
The following statements hold.

\begin{enumerate}
\item
Assume that \(u_0:\Sigma\to(S,h)\) is an orientation-preserving 
covering map, and let \(Y=(\Sigma,u_0^*h)\) be the unique minimizing point
of the varying-domain energy. Let \(X_t\in\mathcal T(\Sigma)\), \(t\geq0\),
be any unit-speed Teichm\"uller ray and put
\(d_0=d_{\mathrm T}(X_0,Y)\). Then
there exists constant $A_{Y}$ and
\begin{equation}
  \frac{\eta_Y}{2}e^{-2d_0}e^{2t}
  \leq E_{u_0}(X_t)
  \leq A_Ye^{2d_0}e^{2t}
  \qquad (t\geq0).
  \label{eq:varying-domain-energy-ray-growth}
\end{equation}
Consequently, \(E_{u_0}(X_t)\asymp e^{2t}\) and
\(\log E_{u_0}(X_t)=2t+O(1)\). If \(X_0=Y\), one may omit the factors
\(e^{\pm2d_0}\) in \eqref{eq:varying-domain-energy-ray-growth}.

\item
Assume that the varying-target energy \(\mathcal E_{u_0}\) satisfies the
filling hypothesis with filling system
\(\Gamma=\{\alpha_1,\ldots,\alpha_N\}\). Let
\(q\in\mathcal Q^1(X_0)\) be Jenkins--Strebel and let
\(X_t=\gamma_q(t)\), \(t\geq0\), be the associated unit-speed ray. Define
\(\Theta(q,\Gamma)=\sum_{j=1}^N(\sum_{r=1}^m
i(\delta_r,\alpha_j))^2\). Then \(\Theta(q,\Gamma)>0\), and
there exist constants
$A_{\mathfrak F}$ and
$B_{\mathfrak F}$ such that
\begin{equation}
  16
  A_{\mathfrak F}\,
  \Theta(q,\Gamma)
  \leq \frac{\mathcal E_{u_0}(X_t)}{t^2}
  \leq 16
  B_{\mathfrak F}
  \,
  \Theta(q,\Gamma)
    \qquad (t\geq0).
\label{eq:varying-target-strebel-energy-ray-growth}
\end{equation}
Consequently, \(\mathcal E_{u_0}(X_t)\asymp t^2\) and
\(\log\mathcal E_{u_0}(X_t)=2\log t+O(1)\).
\end{enumerate}
\end{proposition}

\begin{proof} The unit-speed Teichm\"u{}ller geodesic
  gives
\(d_{\mathrm T}(X_t,X_0)=t\), and hence
\(t-d_0\leq d_{\mathrm T}(X_t,Y)\leq t+d_0\). Applying Theorem \ref{thm:covering-energy-distance-comparison},
namely \((\eta_Y/2)e^{2d_{\mathrm T}(X,Y)}\leq E_{u_0}(X)\leq
A_Ye^{2d_{\mathrm T}(X,Y)}\), proves
\eqref{eq:varying-domain-energy-ray-growth}.

For the varying-target energy, Masur's asymptotic formula gives
\begin{equation}
  \lim_{t\to\infty}\frac{\ell_{X_t}(\alpha)}{4t}
  =\sum_{r=1}^m i(\delta_r,\alpha)
  \label{eq:masur-strebel-simple-curve-asymptotic}
\end{equation}
for every essential simple closed curve \(\alpha\); see
\cite[Theorem~1.1 and the final estimates in its proof]{Masur1982}. Applying this to the curves in \(\Gamma\)
gives \(\Lambda_\Gamma(X_t)/t^2\to16\Theta(q,\Gamma)\), where
\(\Lambda_\Gamma(X)=\sum_j\ell_X(\alpha_j)^2\). Since \(\Gamma\) fills
\(S\), every \(\delta_r\) meets some member of \(\Gamma\), so
\(\Theta(q,\Gamma)>0\). Proposition \ref{prop:energy-filling-length-comparison} now gives
\(A_{\mathfrak F}\Lambda_\Gamma\leq\mathcal E_{u_0}\leq
B_{\mathfrak F}\Lambda_\Gamma\), and
\eqref{eq:varying-target-strebel-energy-ray-growth} follows.
\end{proof}

\begin{remark}
The exponential estimate in the varying-domain covering case holds along
every unit-speed Teichm\"uller ray and does not require the
Jenkins--Strebel assumption. In the varying-target case, by contrast, the
Jenkins--Strebel geometry makes the relevant hyperbolic lengths grow
linearly, and comparison with their squares gives quadratic energy growth.
The constants are uniform in \(t\), but no uniformity over all
Jenkins--Strebel directions is asserted.
\end{remark}

Let
\(
  X_t=\gamma_q(t)
\)
be the unit-speed Teichm\"uller ray determined by
\(q\in\mathcal Q^1(X_0)\), and let
\(
  F_t:X_0\to X_t
\)
be the corresponding Teichm\"uller maps. For each \(t\), let
\(
  f_t:(X_t,\sigma_{X_t})\to(Y,u_0^*h)
\)
be the harmonic diffeomorphism in the marking-preserving homotopy class, so
that, by Lemma~\ref{lem:covering-factorization},
\[
  E(t):=E_{u_0}(X_t)=E(f_t).
\]

\begin{proposition}\label{prop:normalized-covering-energy-limit}

Let \(\mathscr H_{\mathrm{id}}\) denote the marking-preserving homotopy class
of Lipschitz maps \(v:X_0\to Y\). In a \(q\)-natural coordinate
\(z=x+iy\), define
\[
  H_q(v)
  :=
  \frac12
  \int_{X_0\setminus Z(q)}
  |v_x|_{u_0^*h}^2\,dx\,dy,
  \quad
  V_q(v)
  :=
  \frac12
  \int_{X_0\setminus Z(q)}
  |v_y|_{u_0^*h}^2\,dx\,dy.
\]
These quantities are globally well-defined because the transition maps
between \(q\)-natural coordinates are of the form
\(z\mapsto\pm z+c\), and \(Z(q)\) has measure zero. Set
\begin{equation}
  \mathscr V_q(Y)
  :=
  \inf_{v\in\mathscr H_{\mathrm{id}}}V_q(v).
  \label{eq:def-minimal-vertical-energy}
\end{equation}
Then
\begin{equation}
  e^{-2t}E(t)
  \searrow
  \mathscr V_q(Y)
  \qquad
  \text{as }t\to\infty.
  \label{eq:normalized-energy-monotone-limit}
\end{equation}
In particular,
\begin{equation}
  \lim_{t\to\infty}
  \frac{E(t)}{e^{2t}}
  =
  \mathscr V_q(Y),
  \qquad
  E(t)
  =
  \mathscr V_q(Y)e^{2t}+o(e^{2t}).
  \label{eq:covering-energy-leading-coefficient}
\end{equation}

Suppose, in addition, that \(q\) is Jenkins--Strebel and that its vertical
foliation is the contracting foliation of the ray. Write the complementary
components of its vertical critical graph as maximal flat cylinders
\(
  C_i
  \cong
  (0,h_i)\times
  \bigl(\mathbb R/c_i\mathbb Z\bigr),
\) \(
  i=1,\ldots,r,
\)
where \(q=dz^2\), \(z=x+iy\), the vertical trajectories
\(x=\mathrm{constant}\) are closed, \(c_i\) is their \(q\)-length, and
\(h_i\) is the transverse height of \(C_i\). Let \(\delta_i\) be the core
curve of \(C_i\), and put
\(
  M_i:=\operatorname{Mod}(C_i)=\frac{h_i}{c_i}.
\)
Then
\begin{equation}
  \mathscr V_q(Y)
  =
  \frac12
  \sum_{i=1}^r
  M_i\,\ell_Y(\delta_i)^2
  =
  \frac12
  \sum_{i=1}^r
  \frac{h_i}{c_i}\,
  \ell_Y(\delta_i)^2.
  \label{eq:strebel-minimal-vertical-energy}
\end{equation}
Consequently,
\begin{equation}
  \lim_{t\to\infty}
  \frac{E(t)}{e^{2t}}
  =
  \frac12
  \sum_{i=1}^r
  \operatorname{Mod}(C_i)\,
  \ell_Y(\delta_i)^2.
  \label{eq:strebel-covering-energy-leading-coefficient}
\end{equation}
\end{proposition}

\begin{proof}
For a Lipschitz map \(v\in\mathscr H_{\mathrm{id}}\), define
\(
  \mathcal F(t,v)
  :=
  E_{X_t,Y}(v\circ F_t^{-1}).
\)
By conformal invariance of the two-dimensional Dirichlet energy and the
affine expression
\(
  F_t(x+iy)=e^t x+i e^{-t}y
\)
in \(q\)-natural coordinates, one has
\begin{equation}
  \mathcal F(t,v)
  =
  e^{-2t}H_q(v)+e^{2t}V_q(v).
  \label{eq:covering-pulled-back-HV-energy}
\end{equation}
Since \(f_t\) minimizes energy in the marking-preserving homotopy class,
precomposition by \(F_t^{-1}\) gives
\(
  E(t)
  =
  \inf_{v\in\mathscr H_{\mathrm{id}}}
  \mathcal F(t,v).
\)
Dividing by \(e^{2t}\), we obtain
\begin{equation}
  e^{-2t}E(t)
  =
  \inf_{v\in\mathscr H_{\mathrm{id}}}
  \left\{
    V_q(v)+e^{-4t}H_q(v)
  \right\}.
  \label{eq:normalized-energy-as-penalized-infimum}
\end{equation}

The right-hand side is nonincreasing in \(t\), since \(H_q(v)\ge0\).
Moreover, if
\(
  \mathscr V_q(Y)
  =
  \inf_{\mathscr H_{\mathrm{id}}}V_q,
\)
then
\(
  e^{-2t}E(t)\ge\mathscr V_q(Y)
\)
for every \(t\). Conversely, given \(\varepsilon>0\), choose
\(v_\varepsilon\in\mathscr H_{\mathrm{id}}\) such that
\(
  V_q(v_\varepsilon)
  \le
  \mathscr V_q(Y)+\varepsilon.
\)
Since \(v_\varepsilon\) is Lipschitz, \(H_q(v_\varepsilon)<\infty\), and
\eqref{eq:normalized-energy-as-penalized-infimum} gives
\(
  e^{-2t}E(t)
  \le
  V_q(v_\varepsilon)
  +
  e^{-4t}H_q(v_\varepsilon).
\)
Taking the upper limit as \(t\to\infty\) yields
\(
  \limsup_{t\to\infty}e^{-2t}E(t)
  \le
  \mathscr V_q(Y)+\varepsilon.
\)
Letting \(\varepsilon\to0\) proves
\eqref{eq:normalized-energy-monotone-limit} and
\eqref{eq:covering-energy-leading-coefficient}.

We now assume that \(q\) is Jenkins--Strebel. Let
\(v\in\mathscr H_{\mathrm{id}}\). For each \(i\) and each
\(x\in(0,h_i)\), the loop
\(
  y\mapsto v(x,y),
\) \(
  y\in\mathbb R/c_i\mathbb Z,
\)
represents the free homotopy class \(\delta_i\) on \(Y\). Hence
\(
  \ell_Y(\delta_i)
  \le
  \int_0^{c_i}|v_y(x,y)|_{u_0^*h}\,dy.
\)
The Cauchy--Schwarz inequality gives
\(
  \frac12
  \int_0^{c_i}
  |v_y(x,y)|_{u_0^*h}^2\,dy
  \ge
  \frac{\ell_Y(\delta_i)^2}{2c_i}.
\)
Integrating with respect to \(x\in(0,h_i)\), we obtain
\(
  V_q(v;C_i)
  \ge
  \frac{h_i}{2c_i}\,
  \ell_Y(\delta_i)^2.
\)
The vertical critical graph has \(q\)-area zero, so summing over the
cylinders yields
\begin{equation}
  V_q(v)
  \ge
  \frac12
  \sum_{i=1}^r
  \frac{h_i}{c_i}\,
  \ell_Y(\delta_i)^2.
  \label{eq:strebel-vertical-energy-lower-bound}
\end{equation}

It remains to prove the reverse inequality for the infimum. Choose a smooth
map \(v_0\in\mathscr H_{\mathrm{id}}\). For each \(i\), let
\(
  \gamma_i:
  \mathbb R/c_i\mathbb Z
  \to
  Y
\)
be the constant-speed parametrization of the geodesic representative of
\(\delta_i\). Thus
\(
  |\gamma_i'|
  =
  \frac{\ell_Y(\delta_i)}{c_i}.
\)
On the central portion of \(C_i\), consider the map
\(
  g_i(x,y):=\gamma_i(y).
\)
Its vertical energy is
\(
  V_q(g_i;C_i)
  =
  \frac{h_i}{2c_i}\,
  \ell_Y(\delta_i)^2.
\)

To obtain a global map in \(\mathscr H_{\mathrm{id}}\), leave \(v_0\)
unchanged in shrinking neighborhoods of the cylinder boundaries and of the
critical graph, use \(g_i\) on the complementary central subcylinder, and
interpolate between the two maps in thin transition strips. The boundary
loops of \(v_0|_{C_i}\) and the loop \(\gamma_i\) represent the same free
homotopy class, so the interpolations may be chosen through smooth
homotopies. They may also be chosen so that the resulting global map is
obtained from \(v_0\) by homotopies supported in the cylinder interiors and
therefore remains in \(\mathscr H_{\mathrm{id}}\).

More precisely, one may choose a sequence of such Lipschitz maps
\(v_\nu\in\mathscr H_{\mathrm{id}}\) for which the boundary collars shrink
to the critical graph and the transition widths are chosen after the
interpolating homotopies so that their total vertical energy tends to zero.
The vertical energy of \(v_0\) on the shrinking boundary collars also tends
to zero by absolute continuity of the integral. Consequently,
\(
  \lim_{\nu\to\infty}V_q(v_\nu)
  =
  \frac12
  \sum_{i=1}^r
  \frac{h_i}{c_i}\,
  \ell_Y(\delta_i)^2.
\)
Together with
\eqref{eq:strebel-vertical-energy-lower-bound}, this proves
\eqref{eq:strebel-minimal-vertical-energy}, and hence
\eqref{eq:strebel-covering-energy-leading-coefficient}.
\end{proof}

\begin{remark}
\label{rem:normalized-energy-extremal-length}
The preceding coefficient also has an extremal-length interpretation.
For \(X\in\mathcal T(\Sigma)\), define
\[
  \mathcal L_Y(X)
  :=
  \frac12
  \sup_{0\ne\lambda\in\mathcal{MF}(\Sigma)}
  \frac{\ell_Y(\lambda)^2}
       {\operatorname{Ext}_X(\lambda)}.
\]
Minsky's energy lower bound and length--energy comparison imply that there
is a constant \(C_\Sigma>0\), depending only on the topological type of
\(\Sigma\), such that
\begin{equation}
  \mathcal L_Y(X)
  \le
  E_{u_0}(X)
  \le
  \mathcal L_Y(X)+C_\Sigma;
  \label{eq:minsky-energy-extremal-length-comparison}
\end{equation}
see
\cite[Proposition~3.1 and Theorem~7.2]{Minsky1992}.
The covering factorization is used here to reduce \(E_{u_0}(X)\) to the
energy of the harmonic diffeomorphism \(X\to Y\).
Applying
\eqref{eq:minsky-energy-extremal-length-comparison} to \(X_t\) and dividing
by \(e^{2t}\), the bounded additive error disappears. Therefore
\begin{equation}
  \mathscr V_q(Y)
  =
  \frac12
  \lim_{t\to\infty}
  \sup_{0\ne\lambda\in\mathcal{MF}(\Sigma)}
  \frac{\ell_Y(\lambda)^2}
       {e^{2t}\operatorname{Ext}_{X_t}(\lambda)}.
  \label{eq:minimal-vertical-energy-extremal-length-limit}
\end{equation}

In \cite{LyuQi2026}
Lyu and Qi
give a precise pointwise formula for the denominator in
\eqref{eq:minimal-vertical-energy-extremal-length-limit}. Suppose that the
vertical foliation decomposes into indecomposable components as
\(
  \mathcal F_{\mathrm v}(q)
  =
  \sum_{j=1}^n a_jG_j.
\)
Then, for every \(F\in\mathcal{MF}(\Sigma)\),
\begin{equation}
  \lim_{t\to\infty}
  e^{2t}\operatorname{Ext}_{X_t}(F)
  =
  \sup_{F'\in\mathcal{MF}(\Sigma)}
  \frac{i(F,F')^2}
       {\displaystyle
        \sum_{j=1}^n
        \frac{a_j\,i(G_j,F')^2}
             {i(G_j,\mathcal F_{\mathrm h}(q))}};
  \label{eq:lyu-qi-contracting-extremal-length}
\end{equation}
see \cite[Theorem~1.1]{LyuQi2026}.
In the Jenkins--Strebel case,
\(
  \mathcal F_{\mathrm v}(q)
  =
  \sum_{i=1}^r h_i\delta_i,
 \) \(
  i\bigl(\delta_j,\mathcal F_{\mathrm h}(q)\bigr)
  =
  c_j.
\)
Thus, for
\(
  F=\sum_{i=1}^r b_i\delta_i
\)
with \(b_i\ge0\), Corollary~1.2 of
\cite{LyuQi2026} gives
\begin{equation}
  \lim_{t\to\infty}
  e^{2t}\operatorname{Ext}_{X_t}(F)
  =
  \sum_{i=1}^r
  \frac{b_i^2c_i}{h_i}
  =
  \sum_{i=1}^r
  \frac{b_i^2}{M_i}.
  \label{eq:lyu-qi-strebel-contracting-extremal-length}
\end{equation}
Since
\(
  \ell_Y(F)
  =
  \sum_{i=1}^r b_i\ell_Y(\delta_i),
\)
the weighted Cauchy--Schwarz inequality gives
\[
  \sup_{\substack{b_i\ge0\\(b_1,\ldots,b_r)\ne0}}
  \frac{
    \left(
      \sum_{i=1}^r b_i\ell_Y(\delta_i)
    \right)^2
  }{
    \sum_{i=1}^r b_i^2/M_i
  }
  =
  \sum_{i=1}^r
  M_i\ell_Y(\delta_i)^2.
\]
The maximum is realized, up to a common positive factor, by
\(
  b_i=M_i\ell_Y(\delta_i).
\)
This recovers precisely the coefficient in
\eqref{eq:strebel-minimal-vertical-energy} and shows that the vertical-energy
formula is compatible with the asymptotic extremal-length description of
Minsky and Lyu--Qi.
\end{remark}

\medskip
\noindent\textbf{Dual trees and equivariant energy.}
Now we consider the case of harmonic maps $u_t: M\to  X_t$ by fixing the domain $M$ and varying the targets $X_t$.
Let \(\Delta=\sum_{i=1}^r\delta_i\) be a multicurve on
\(S\), with pairwise nonisotopic essential simple closed curves
\(\delta_i\) \cite[Section~1.2]{FarbMargalit2012}. Its dual tree \(T_\Delta\) is
obtained by lifting \(\Delta\) to \(\widetilde S\), collapsing each
complementary component to a vertex, and assigning length \(1\) to an
edge dual to a lift of \(\delta_i\). Thus the natural
\(\pi_1(S)\)-action is normalized by
\(\ell_{T_\Delta}(\gamma)=i(\Delta,\gamma)\).

A map \(u_0:M\to S\) induces an action of \(G:=\pi_1(M)\) on
\(T_\Delta\) through \((u_0)_*:G\to\pi_1(S)\). We define
$$
\mathscr E_\Delta(u_0)=\inf_U E(U),
$$where the infimum is over all
\(G\)-equivariant maps \(U\in W^{1,2}_{\mathrm{loc}}(\widetilde M,T_\Delta)\)
and
\(E(U)=\frac12\int_D|dU|^2\,d\mu_g\) for any relatively compact
fundamental domain \(D\subset\widetilde M\) with piecewise smooth boundary.
Equivariance makes the energy density \(G\)-invariant, so this is equivalent
to integration over \(M\). Metric-space-valued energy is understood in the
sense of Korevaar--Schoen \cite{KorevaarSchoen1997}, normalized to agree with the usual Dirichlet
energy for smooth Riemannian targets.

The set of $G$-equivariant maps in $W_{\text {loc }}^{1,2}\left(\widetilde{M}, T_{\Delta}\right)$ is not empty.
In fact, one can construct an equivariant locally Lipschitz map, see
\cite[Proposition~2.6.1]{KorevaarSchoen1993}.
In the present situation, such a map can also be described explicitly.
Choose pairwise disjoint collars of the components of $\Delta$ and lift
them to $\widetilde S\simeq\mathbb H^2$. Collapsing each complementary
component to the corresponding vertex of $T_\Delta$ and projecting each
lifted collar linearly onto the corresponding edge gives a
$\pi_1(S)$-equivariant Lipschitz map
$c_\Delta:\widetilde S\to T_\Delta$. If
$\widetilde u_0:\widetilde M\to\widetilde S$ is a lift of a smooth
representative of $u_0$, then
$U_0:=c_\Delta\circ\widetilde u_0$ satisfies
$U_0(\gamma x)=\rho(\gamma)U_0(x)$ for every $\gamma\in G$. Moreover,
$U_0$ is locally Lipschitz, and hence belongs to
$W_{\mathrm{loc}}^{1,2}(\widetilde M,T_\Delta)$. Since $M$ is compact,
$U_0$ has finite energy on a fundamental domain. If
\(T_\Delta^{\min}\) is the minimal invariant subtree, nearest-point
projection and inclusion give
\(\mathscr E_{T_\Delta^{\min}}(u_0)=\mathscr E_\Delta(u_0)\); and if a
tree metric is multiplied by \(a>0\), then
\begin{equation}
  \mathscr E_{aT}(u_0)=a^2\mathscr E_T(u_0).
  \label{eq:tree-energy-scaling-preliminary}
\end{equation}

\begin{lemma}\label{lemma-pullback}
Let \(t_n\to\infty\), put
\(E_n=\mathcal E_{u_0}(X_{t_n})\), and suppose that
\(E_n/t_n^2\to c\in(0,\infty)\). Under the filling hypothesis, after
passing to a subsequence there are a minimal \(G\)-tree \(T\) and a
\(G\)-equivariant energy-minimizing map
\(U:\widetilde M\to T\) such that \(E(U)=1\) and
\begin{equation}
  \ell_T(\gamma)
  =\frac{4}{\sqrt c}\,
    i\bigl(\Delta_q,(u_0)_*\gamma\bigr)
  \qquad (\gamma\in G).
  \label{eq:normalized-pullback-length-function}
\end{equation}
The limiting action factors through \(H:=(u_0)_*(G)
  \subset \pi_1(S)
  \) on \(T\). Here $\Delta_q=\sum_{i=1}^r \delta_i$ denotes the multicurve associated with Jenkins--Strebel differential $q$.
\end{lemma}

\begin{proof}
  Write \(\Pi=\pi_1(S)\), and then
  \(H=(u_0)_*(G)
  \subset \Pi\). For each \(t\), choose
the corresponding representation
\(\overline\rho_t:\Pi\to\operatorname{PSL}_2(\mathbb R)\)
for the uniformization of the marked
surface \(X_t\). A closed-surface Fuchsian representation lifts
further to the double cover
\(\operatorname{SL}_2(\mathbb R)\)
of \(\operatorname{PSL}_2(\mathbb R)\), since its Euler number
\(\pm(2g(S)-2)\) has trivial mod-two reduction. Fix a lift
\(\widetilde{\overline\rho}_t:\Pi\to\operatorname{SL}_2(\mathbb R)\) and
set \(\rho_t=\overline\rho_t\circ(u_0)_*\) and
\(\widetilde\rho_t=\widetilde{\overline\rho}_t\circ(u_0)_*\).

The filling hypothesis implies that \(H\) is noncyclic: it contains
representatives of all curves in the filling system, whereas the nontrivial
elements of a cyclic subgroup of a closed surface group are powers of a
single primitive class. Hence
\(\overline\rho_t(H)\) is a non-elementary discrete Fuchsian group and
\(\widetilde\rho_t\) is irreducible as an
\(\operatorname{SL}_2(\mathbb C)\)-representation. Let
\(\widetilde u_t:\widetilde M\to\mathbb H^2\) be the
\(\rho_t\)-equivariant lift of the harmonic representative. Its energy on a
fundamental domain is \(\mathcal E_{u_0}(X_t)\). Viewing
\(\mathbb H^2\) as the invariant totally geodesic plane in \(\mathbb H^3\),
the equivariant nearest-point projection
\(\mathbb H^3\to\mathbb H^2\) shows that \(\widetilde u_t\) is also
energy minimizing among \(\mathbb H^3\)-valued equivariant maps.

Choose \(\alpha_j\in\Gamma\) with \(i(\Delta_q,\alpha_j)>0\), and choose
\(\gamma_j\in G\) whose image represents \(\alpha_j\). Masur's formula
implies \(\ell_{X_t}(\alpha_j)\to\infty\), while
\(\lvert\operatorname{tr}\widetilde\rho_t(\gamma_j)\rvert
=2\cosh(\ell_{X_t}(\alpha_j)/2)\). Thus the representations leave every
compact subset of the character variety.

Rescale the target distance by
\(\widehat d_n=E_n^{-1/2}d_{\mathbb H^2}\). Then
\(\widetilde u_{t_n}:\widetilde M\to(\mathbb H^2,\widehat d_n)\) has
energy one. To verify the uniform local modulus of continuity required by
Korevaar--Schoen, fix \(K\Subset K'\Subset\widetilde M\). The compact set
\(K'\) is contained in finitely many translates of a fixed fundamental
domain, and equivariance therefore gives
\(E_{\widehat d_n}(\widetilde u_{t_n};K')\leq N_{K'}\), with
\(N_{K'}\) independent of \(n\). The local Bochner and mean-value estimates
for harmonic maps into nonpositively curved targets then yield
\begin{equation}
  \sup_K|d\widetilde u_{t_n}|_{\widehat d_n}^2\leq C_K,
  \label{eq:uniform-local-lipschitz-normalized}
\end{equation}
where \(C_K\) is independent of \(n\). See the proof of Proposition \ref{prop:energy-filling-length-comparison}.

Proposition~3.7 of \cite{KorevaarSchoen1997} now gives, after passage to a
subsequence, locally uniform convergence in the pullback sense. The
completed quotient construction in
\cite[Lemma~3.1, Definition~3.3, and Proposition~3.4]{KorevaarSchoen1997}
produces a complete NPC space \(T\) and a map
\(U:\widetilde M\to T\) whose pullback pseudometric is the limit of those
of \(\widetilde u_{t_n}\). Lemma~3.5 of the same paper supplies an
isometric \(G\)-action on \(T\) for which \(U\) is equivariant. Since the
normalized maps are exact equivariant minimizers of uniformly bounded
energy, Theorem~3.9 applies with \(\varepsilon_n=0\). It follows that \(U\)
is energy minimizing and that the energy-density measures converge weakly.
Because the quotient \(M\) is compact and every normalized map has energy
one, their total masses converge and
\begin{equation}
  E(U)=1.
  \label{eq:normalized-limit-energy-one}
\end{equation}
The stronger \(L^1\)-convergence of the pullback tensors and energy
densities is recorded in \cite[Corollary~3.10]{KorevaarSchoen1997}.

The convex hulls of the images lie in the invariant plane \(\mathbb H^2\),
so the pullback limit agrees with the one used in
\cite{DaskalopoulosDostoglouWentworth1998}. Since the hyperbolicity
constants of \((\mathbb H^2,\widehat d_n)\) tend to zero,
\cite[Theorem~3.1]{DaskalopoulosDostoglouWentworth1998} shows that \(T\)
is an \(\mathbb R\)-tree and that the limiting action has no global fixed
point. By \cite[Theorem~4.4]{DaskalopoulosDostoglouWentworth1998}, the
subtree generated by \(U(\widetilde M)\) is minimal. Replacing the ambient
tree by this subtree changes neither \(U\), its energy, nor any translation
length.

Let \(C_n\) be the convex hull of \(\widetilde u_{t_n}(\widetilde M)\) in
the rescaled target, and let \(\sigma_n\) be the induced \(G\)-action. The
convex-hull comparison in the proof of
\cite[Theorem~3.2]{DaskalopoulosDostoglouWentworth1998}, followed by the
scaling by \(E_n^{-1/2}\), gives a constant \(\delta>0\), independent of
\(n\) and \(\gamma\), such that
\begin{equation}
  0\leq \ell_{\sigma_n}(\gamma)
  -\frac{\ell_{X_{t_n}}((u_0)_*\gamma)}{\sqrt{E_n}}
  \leq\frac{2\delta}{\sqrt{E_n}}.
  \label{eq:ddw-normalized-length-comparison}
\end{equation}
The common augmented-domain construction in the same proof gives
\(\ell_T(\gamma)=\lim_n\ell_{\sigma_n}(\gamma)\). If
\((u_0)_*\gamma\neq1\), Masur's formula for arbitrary free homotopy classes
\cite{Masur1982} gives
\(\ell_{X_{t_n}}((u_0)_*\gamma)/t_n\to
4i(\Delta_q,(u_0)_*\gamma)\). Together with the ratio
\(E_n/t_n^2\to c\) and \eqref{eq:ddw-normalized-length-comparison}, this
proves \eqref{eq:normalized-pullback-length-function}. If
\((u_0)_*\gamma=1\), both sides vanish. Finally, every element of
\(\ker(u_0)_*\) acts trivially in each pullback space, and hence in the
limit, so the limiting action factors through \(H\).
\end{proof}

\begin{lemma}\label{lemma-rigidity}
Assume the filling hypothesis and let \(T_{\Delta_q}^{\min}\) be the minimal
\(G\)-invariant subtree of \(T_{\Delta_q}\). The induced action on
\(T_{\Delta_q}^{\min}\) is irreducible in the sense of Culler--Morgan \cite{CullerMorgan1987}. If
\(T\) is a minimal \(G\)-tree and, for some \(a>0\),
\(\ell_T=a\ell_{T_{\Delta_q}}\), then there is a \(G\)-equivariant isometry
\begin{equation}
  T\cong aT_{\Delta_q}^{\min}.
  \label{eq:dual-tree-rigidity-isometry}
\end{equation}
\end{lemma}

\begin{proof}
Note that
\(
  \Delta_q=\delta_1+\cdots+\delta_r
\)
is already a reduced multicurve, with unit weight on each component. The
underlying simplicial tree of \(T_{\Delta_q}\) is the Bass--Serre tree of the
splitting of
\(\Pi=\pi_1(S)\) along the multicurve
\(\{\delta_1,\ldots,\delta_r\}\), with every edge assigned length \(1\).

The stabilizer of an edge corresponding to a lift of \(\delta_i\) is a
conjugate of the maximal cyclic subgroup generated by \(\delta_i\). We first
record a consequence concerning adjacent edges. Let \(e\) and \(e'\) be
distinct edges incident to the same vertex. If their stabilizers had a
nontrivial intersection, the malnormality of maximal cyclic subgroups in the
closed surface group \(\Pi\) would imply that these stabilizers coincide.
At the common vertex, incident edges are parametrized by cosets of the
boundary subgroups of the corresponding complementary subsurface. Since
these boundary subgroups are self-normalizing and pairwise nonconjugate in
the vertex group, equality of the two edge stabilizers would force the two
edge cosets, and hence the two edges, to coincide. This is a contradiction.
Therefore distinct adjacent edges have stabilizers with trivial intersection.
It follows that every segment \(I\subset T_{\Delta_q}\) containing two
distinct edges satisfies
\begin{equation}
\label{eq:two-edge-segment-stabilizer}
  \operatorname{Stab}^{\mathrm{pt}}_{\Pi}(I)=\{1\}.
\end{equation}

The filling hypothesis
implies that \(H\) is noncyclic. It also provides an element of positive
translation length. Indeed, some curve \(\alpha_j\) in the filling system
satisfies \(i(\Delta_q,\alpha_j)>0\), and its conjugacy class is represented
by an element \(h_j\in H\). Hence
\(
  \ell_{T_{\Delta_q}}(h_j)
  =
  i(\Delta_q,\alpha_j)
  >
  0.
\)
In particular, the \(H\)-action on \(T_{\Delta_q}\) is not elliptic.

We claim that the action of \(H\) on its minimal invariant subtree is
irreducible. First, it cannot preserve a line \(L\). If it did, the
hyperbolic element \(h_j\) would have its axis contained in \(L\), so \(L\)
would be a simplicial line, namely a union of edges of
\(T_{\Delta_q}\). The kernel of the homomorphism
\(H\to\operatorname{Isom}(L)\) would fix \(L\) pointwise and hence would
fix a segment containing two edges. By
\eqref{eq:two-edge-segment-stabilizer}, this kernel would be trivial.
Consequently, \(H\) would embed in the simplicial isometry group of a line,
which is virtually cyclic. Since \(H\) is a torsion-free subgroup of the
closed surface group \(\Pi\), it would then be cyclic, contrary to the
filling hypothesis.

The group \(H\) cannot fix an end \(\xi\) either. Otherwise, the signed
Busemann displacement would define a homomorphism
\(b_\xi:H\to\mathbb R\). If \(h\in\ker b_\xi\), then \(h\) preserves a ray
representing \(\xi\) without any translation along its eventual common
tail. Hence \(h\) fixes a terminal subray pointwise and, in particular, a
segment containing two edges. Equation
\eqref{eq:two-edge-segment-stabilizer} therefore gives \(h=1\). Thus
\(b_\xi\) is injective, so \(H\) embeds in the abelian group
\((\mathbb R,+)\). It follows that \(H\) is abelian and hence cyclic, since
every abelian subgroup of a closed surface group is cyclic. This again
contradicts the filling hypothesis.

It follows that the action of \(H\) on its minimal invariant subtree
\(T_{\Delta_q}^{\min}\) is irreducible. Since the \(G\)-action factors
surjectively through \(H\), the same subtree is the minimal
\(G\)-invariant subtree and the minimal \(G\)-action is irreducible as
well. In particular, its translation length function is non-abelian, that
is, it is not of the form
\(\lvert\chi(\,\cdot\,)\rvert\) for a homomorphism
\(\chi:G\to\mathbb R\); see
\cite[Corollary~2.3]{CullerMorgan1987}.

Now suppose that \(T\) is a minimal \(G\)-tree whose translation length
function satisfies
\(
  \ell_T
  =
  a\,\ell_{T_{\Delta_q}}
\)
for some \(a>0\). The scaled tree \(aT_{\Delta_q}^{\min}\), obtained by
multiplying the metric of \(T_{\Delta_q}^{\min}\) by \(a\), has the same
translation length function as \(T\). Its action is irreducible and hence
semisimple. If the action on \(T\) were non-semisimple, its translation
length function would agree with that of a shift action and would therefore
be abelian; see \cite[Corollary~2.4]{CullerMorgan1987}. This contradicts
the non-abelianity established above. Thus the action on \(T\) is also
semisimple.

The two actions on \(T\) and \(aT_{\Delta_q}^{\min}\) are therefore
minimal and semisimple and have the same translation length function. The
Culler--Morgan rigidity theorem now yields a \(G\)-equivariant isometry
\(
  T
  \cong
  aT_{\Delta_q}^{\min};
\)
see \cite[Theorem~3.7]{CullerMorgan1987}. This is precisely
\eqref{eq:dual-tree-rigidity-isometry}.
\end{proof}\begin{theorem}
\label{thm:energy-masur-asymptotic}
Let \(q\in\mathcal Q^1(X_0)\) be Jenkins--Strebel, let
\(X_t=\gamma_q(t)\), \(t\geq0\), be the associated unit-speed
Teichm\"uller ray, and let \(\delta_1,\ldots,\delta_r\) be the core curves
of the cylinders in the contracting foliation. Set
\(\Delta_q=\delta_1+\cdots+\delta_r\). Under the filling hypothesis,
\begin{equation}
  \lim_{t\to\infty}\frac{\mathcal E_{u_0}(X_t)}{t^2}
  =16\,\mathscr E_{\Delta_q}(u_0)>0.
  \label{eq:energy-masur-asymptotic}
\end{equation}
\end{theorem}

\begin{proof}
Proposition \ref{cor:asymptotic-growth-two-energy-settings} shows that
\(\mathcal E_{u_0}(X_t)/t^2\) is bounded above and bounded away from zero.
Let \(t_n\to\infty\) be arbitrary. After passing to a subsequence, we may
therefore assume that
\(\mathcal E_{u_0}(X_{t_n})/t_n^2\to c\in(0,\infty)\).

Lemma \ref{lemma-pullback} gives a minimal \(G\)-tree \(T\) and an
energy-minimizing equivariant map \(U:\widetilde M\to T\) with \(E(U)=1\)
and
\(\ell_T=(4/\sqrt c)\ell_{T_{\Delta_q}}\). By Lemma \ref{lemma-rigidity},
\(T\cong(4/\sqrt c)T_{\Delta_q}^{\min}\). Using the projection identity for
the minimal subtree and the scaling law
\eqref{eq:tree-energy-scaling-preliminary}, we obtain
\begin{equation}
  1=E(U)=\mathscr E_T(u_0)
  =\frac{16}{c}\,\mathscr E_{\Delta_q}(u_0),
  \qquad
  c=16\,\mathscr E_{\Delta_q}(u_0).
  \label{eq:subsequential-coefficient-determined}
\end{equation}
Thus every convergent subsequence has the same limit. Since the original
sequence \(t_n\to\infty\) was arbitrary, the full limit exists and equals
the right-hand side of \eqref{eq:energy-masur-asymptotic}. Its positivity
follows from the lower bound in
\eqref{eq:varying-target-strebel-energy-ray-growth}.
\end{proof}

\begin{remark}
\label{rem:circle-energy-masur-compatibility}
The filling hypothesis is unnecessary in dimension one. Let \(M=S_L^1\)
be a circle of length \(L\), and suppose that \(u_0\) represents a nontrivial
free homotopy class \(\alpha\). The harmonic representative has constant
speed, so \(\mathcal E_\alpha(X_t)=\ell_{X_t}(\alpha)^2/(2L)\). Masur's
formula therefore gives
\begin{equation}
  \lim_{t\to\infty}\frac{\mathcal E_\alpha(X_t)}{t^2}
  =\frac{8}{L}i(\Delta_q,\alpha)^2.
  \label{eq:circle-energy-masur-limit}
\end{equation}
On the dual tree, the generator of \(\pi_1(S_L^1)\) has translation length
\(i(\Delta_q,\alpha)\). For any equivariant
\(U:\mathbb R\to T_{\Delta_q}\), Cauchy--Schwarz gives
\(i(\Delta_q,\alpha)^2\leq2LE(U)\), and equality is attained by the
constant-speed map onto the axis. Hence
\(\mathscr E_{\Delta_q}(\alpha)=i(\Delta_q,\alpha)^2/(2L)\), so
\(16\mathscr E_{\Delta_q}(\alpha)\) agrees exactly with
\eqref{eq:circle-energy-masur-limit}. If the intersection number is zero,
both sides vanish.
\end{remark}

By a standard abuse of notation, we use the same symbol
\(\Delta_q=\sum_{i=1}^r \delta_i\) for the associated measured foliation class.
We specialize the above result to the case $M=S$.

\begin{corollary}
\label{cor:identity-map-tree-energy}
Let \(M=S\) be a closed oriented surface of genus at least two and
\(u_0=\mathrm{id}: M=S\to S,
X_t=\gamma_q(t)\). Let \(X_g\in\mathcal T(S)\)
with  the metric \(g\) on $M=S$. Let
\(\Phi_{\Delta_q,X_g}\in\mathcal Q(X_g)\) be the Hubbard--Masur
differential whose vertical measured foliation is \(\Delta_q\) \cite{HubbardMasur1979}. For each
\(i\), let \(A_i\) be its characteristic Jenkins--Strebel cylinder with
 conformal modulus
\(M_i=\mathrm{Mod}(A_i)\). Then
\begin{equation}
  \mathscr E_{\Delta_q}(\mathrm{id}_S)
  =\frac12\Ext_{X_g}(\Delta_q)
  =\frac12\int_{X_g}|\Phi_{\Delta_q,X_g}|
  =\sum_{i=1}^r\frac{1}{2M_i}.
  \label{eq:identity-tree-energy-all-expressions}
\end{equation}
Consequently,
\begin{equation}
  \lim_{t\to\infty}
  \frac{\mathcal E_{\mathrm{id}_S}(X_t)}{t^2}
  =8\Ext_{X_g}(\Delta_q)
  =\sum_{i=1}^r\frac{8}{M_i}.
  \label{eq:identity-energy-asymptotic-extremal-length}
\end{equation}
\end{corollary}

\begin{proof}
The two-dimensional Korevaar--Schoen energy is conformally invariant, so
\(\mathscr E_{\Delta_q}(\mathrm{id}_S)\) depends on \(g\) only through
\(X_g\). By the Hubbard--Masur theorem
\cite[Chapter~I, \S2, Main Theorem]{HubbardMasur1979}, there is a unique holomorphic quadratic
differential \(\Phi=\Phi_{\Delta_q,X_g}\) whose vertical measured foliation
class is \(\mathcal{F}_v(\Phi)=\Delta_q\).  Since \(\Delta_q\) is a weighted multicurve, \(\Phi\) is
Jenkins--Strebel, see \cite[Chapter~I, \S3, Theorem~2]{HubbardMasur1979}.

Let \(T_\Phi\) be the dual tree of the lifted vertical measured foliation
of \(\Phi\), and let
\(\pi_\Phi:\widetilde X_g\to T_\Phi\) be the leaf-space projection.
Its translation length function is
\(\ell_{T_\Phi}(\gamma)=i(\Delta_q,\gamma)\), so \(T_\Phi\) is naturally
identified, equivariantly and isometrically, with \(T_{\Delta_q}\).

The projection \(\pi_\Phi\) is harmonic in the sense of Wolf: it pulls back
germs of convex functions on the tree to germs of subharmonic functions;
see \cite[\S2.3.2, especially~(2.47)]
{DaskalopoulosWentworth2007} and \cite[\S3]{Wolf1995}.
The dual action is small and hence semisimple, or equivalently reductive;
see \cite[\S3.1.2]{DaskalopoulosWentworth2007}.
It therefore follows from
\cite[Theorem~3.8]{DaskalopoulosWentworth2007}
that \(\pi_\Phi\) minimizes the Korevaar--Schoen energy among all
equivariant \(W_{\mathrm{loc}}^{1,2}\)-maps into \(T_{\Delta_q}\).
Consequently,
\(
  \mathscr E_{\Delta_q}(\mathrm{id}_S)
  =
  E(\pi_\Phi).
\)
See also \cite[Sections~3.1--3.2]{Wolf1996} for the direct-method
harmonic-map realization of the Hubbard--Masur theorem.

Away from the critical graph, choose a natural coordinate
\(\zeta=x+iy\) with \(\Phi=d\zeta^2\). The vertical leaves are
\(x=\mathrm{constant}\), the transverse measure is \(|dx|\), and locally
\(\pi_\Phi(x,y)=x\). If \(g=e^{2\eta}(dx^2+dy^2)\), then
\(\frac12|d\pi_\Phi|_g^2d\mu_g=\frac12dx\,dy=\frac12|\Phi|\). The critical
graph has area zero, and hence
\begin{equation}
  \mathscr E_{\Delta_q}(\mathrm{id}_S)
  =E(\pi_\Phi)
  =\frac12\int_{X_g}|\Phi|.
  \label{eq:identity-map-half-area}
\end{equation}
This is the unrescaled transverse-measure normalization; compare
\cite[Section~2.4]{Sagman2023}, where rescaling the tree metric by a factor
of two changes the energy by a factor of four.

The Hubbard--Masur--Kerckhoff area formula gives
\(\Ext_{X_g}(\Delta_q)=\int_{X_g}|\Phi|\); see
\cite{HubbardMasur1979,Kerckhoff1980} and the account in
\cite[Section~2.2]{Sagman2023}. It remains to compute the area in cylinder
coordinates. The cylinder \(A_i\) is isometric in the flat metric
\(|\Phi|\) to
\([0,1]\times[0,c_i]/((x,0)\sim(x,c_i))\): its transverse height is the
weight \(1\), its circumference is \(c_s\), and its modulus is
\(M_i=1/c_i\). Therefore
\(\Area_{|\Phi|}(A_i)=c_i=1/M_i\). The critical graph has zero area,
so summing over the cylinders proves the last equality in
\eqref{eq:identity-tree-energy-all-expressions}; see also
\cite{Strebel1984} for the cylinder decomposition. Finally,
Theorem~\ref{thm:energy-masur-asymptotic} gives
\eqref{eq:identity-energy-asymptotic-extremal-length}.
\end{proof}

\begin{remark}
\label{rem:identity-tree-energy-normalization}
The factor \(1/2\) in
\eqref{eq:identity-tree-energy-all-expressions} comes from using the
unrescaled transverse-measure metric on \(T_{\Delta_q}\), for which
\(\ell_{T_{\Delta_q}}(\gamma)=i(\Delta_q,\gamma)\), together with the
Dirichlet convention \(E(U)=\frac12\int|dU|^2\). Some references rescale
the tree so that its energy is written directly as extremal length.

The Hubbard--Masur differential \(\Phi_{\Delta_q,X_0}\) is
not equal to the original Jenkins--Strebel differential \(q\)
generally. The vertical foliation of
\(q\) records the transverse flat widths of its cylinders, whereas
\(\Delta_q=\delta_1+\cdots+\delta_r\) assigns unit weight to each cylinder
core. Thus the passage
from \(q\) to \(\Delta_q\) generally discards the original widths. In the
special case \(\mathcal F_v(q)=c\Delta_q\), uniqueness in the
Hubbard--Masur theorem gives
\(\Phi_{\Delta_q,X_0}=c^{-2}q\), and, since \(q\) has unit area,
\begin{equation}
  \mathscr E_{\Delta_q}(\mathrm{id}_S)=\frac{1}{2c^2},
  \qquad
  \lim_{t\to\infty}\frac{\mathcal E_{\mathrm{id}_S}(X_t)}{t^2}
  =\frac{8}{c^2}.
  \label{eq:special-proportional-foliation-case}
\end{equation}
\end{remark}

\begin{corollary}
\label{cor:covering-map-tree-energy}
Let
\(
  u_0:\Sigma\to S
\)
be an orientation-preserving covering of degree \(d\) between
closed surfaces, and let \(X_g\in\mathcal T(\Sigma)\) be the conformal
structure determined by the domain metric \(g\). Denote by
\(u_0^*\Delta_q\) the pullback measured foliation on \(\Sigma\). Then
\begin{equation}
\label{eq:covering-map-tree-energy-extremal-length}
  \mathscr E_{\Delta_q}(u_0)
  =
  \frac12\,
  \operatorname{Ext}_{X_g}\!\left(u_0^*\Delta_q\right).
\end{equation}
More explicitly, write the pullback  multicurve as
\(u_0^*\Delta_q=\sum_{i=1}^r\sum_{a=1}^{r_i}\widehat\delta_{i,a}\). Let
\(\widehat\Phi=\Phi_{u_0^*\Delta_q,X_g}\) be the Hubbard--Masur
differential on \(X_g\), and let \(\widehat A_{i,a}\) be its characteristic
cylinder with core curve \(\widehat\delta_{i,a}\). If
\(\widehat M_{i,a}=\operatorname{Mod}(\widehat A_{i,a})\) denotes its
actual conformal modulus, then
\begin{equation}
\label{eq:covering-map-tree-energy-cylinder-moduli}
  \mathscr E_{\Delta_q}(u_0)
  =
  \frac12\int_{X_g}|\widehat\Phi|
  =
  \frac12
  \sum_{i=1}^{r}\sum_{a=1}^{r_i}
  \frac{1}{\widehat M_{i,a}}.
\end{equation}

Suppose, in addition, that the  conformal structure on $\Sigma$ is the pullback
of the initial target structure, namely \(X_g=u_0^*X_0\in T(\Sigma)\). Then
\begin{equation}
\label{eq:covering-map-tree-energy-degree}
  \mathscr E_{\Delta_q}(u_0)
  =
  \frac d2\,
  \operatorname{Ext}_{X_0}(\Delta_q)=\frac d2
  \sum_{i=1}^{r}\frac1{M_i}.
\end{equation}
\end{corollary}

\begin{proof}
We first describe the pullback measured foliation. Choose the components $\delta_1,\ldots,\delta_r$ to be their geodesic
representatives with respect to an auxiliary hyperbolic metric $\sigma$ on
$S$, and equip $\Sigma$ with the pullback metric $u_0^*\sigma$. Then
$u_0:(\Sigma,u_0^*\sigma)\to(S,\sigma)$ is a local isometry, and every
component $\widehat\delta_{i,a}$ of $u_0^{-1}(\delta_i)$ is a closed
geodesic on $\Sigma$. These lifted curves are pairwise nonisotopic.
Indeed, if $\widehat\delta_{i,a}$ and $\widehat\delta_{j,b}$ were freely
homotopic, the uniqueness of the geodesic representative in a nontrivial
free homotopy class on a closed hyperbolic surface would imply
$\widehat\delta_{i,a}=\widehat\delta_{j,b}$. Applying $u_0$ then gives
$\delta_i=\delta_j$, and hence $i=j$, since $\Delta_q$ is reduced.
Finally, equality of two connected components of
$u_0^{-1}(\delta_i)$ implies $a=b$. Thus $\widehat\delta_{i,a}\simeq\widehat\delta_{j,b}$ implies $(i,a)=(j,b)$.
Consequently, the pullback multicurve is already reduced and
\[
  u_0^*\Delta_q
  =
  \sum_{i=1}^r\sum_{a=1}^{r_i}\widehat\delta_{i,a},
\]
with unit weight on every lifted component. More intrinsically, \(u_0^*\Delta_q\) is characterized
by
\(
  i\bigl(u_0^*\Delta_q,\gamma\bigr)
  =
  i\bigl(\Delta_q,(u_0)_*\gamma\bigr)
\)
for every free homotopy class \(\gamma\) on \(\Sigma\).

Let \(\widetilde\Sigma\) and \(\widetilde S\) be the universal covers.
The covering \(u_0\) lifts to a map
\(\widetilde{u}_0:\widetilde\Sigma\to\widetilde S\). Since both spaces are
simply connected, \(\widetilde{u}_0\) is a diffeomorphism. Moreover, the
lifted foliation of \(u_0^*\Delta_q\) is precisely the pullback under
\(\widetilde{u}_0\) of the lifted foliation of \(\Delta_q\). Since the
transverse measure is also preserved, \(\widetilde{u}_0\) induces an
isometry of leaf spaces
\(
  \iota:
  T_{u_0^*\Delta_q}
  \to
  T_{\Delta_q}.
\)
This isometry is equivariant in the sense that
\(\iota(\gamma\xi)=(u_0)_*(\gamma)\iota(\xi)\) for every
\(\gamma\in\pi_1(\Sigma)\) and \(\xi\in T_{p^*\Delta_q}\), where the
action on \(T_{\Delta_q}\) is restricted through
\((u_0)_*:\pi_1(\Sigma)\to\pi_1(S)\).

Composition with \(\iota\) therefore gives an energy-preserving
bijection between the admissible equivariant maps
\(\widetilde\Sigma\to T_{u_0^*\Delta_q}\) for the identity-map problem on
\(\Sigma\) and the admissible equivariant maps
\(\widetilde\Sigma\to T_{\Delta_q}\) defining
\(\mathscr E_{\Delta_q}(u_0)\). Consequently,
\[
  \mathscr E_{\Delta_q}(u_0)
  =
  \mathscr E_{u_0^*\Delta_q}(\mathrm{id}_\Sigma).
\]
Applying the identity-map formula \eqref{eq:identity-tree-energy-all-expressions} gives
\(
  \mathscr E_{\Delta_q}(u_0)
  =
  \frac12\,
  \operatorname{Ext}_{X_g}(u_0^*\Delta_q),
\)
which proves
\eqref{eq:covering-map-tree-energy-extremal-length}.

We next derive the cylinder-modulus expression. By the Hubbard--Masur
theorem, there is a unique holomorphic quadratic differential
\(\widehat\Phi=\Phi_{u_0^*\Delta_q,X_g}\) whose vertical measured
foliation is \(u_0^*\Delta_q\). Since this foliation is a weighted
multicurve, \(\widehat\Phi\) is Jenkins--Strebel. More precisely,
\cite[Chapter~I, \S3, Theorem~2, p.~225]
{HubbardMasur1979}
shows that the characteristic cylinder
\(\widehat A_{i,a}\) has transverse unit height.

Let \(\widehat c_{i,a}\) denote the circumference of
\(\widehat A_{i,a}\). By definition,
\(\widehat M_{i,a}=1/\widehat c_{i,a}\), and hence
\(\widehat c_{i,a}=1/\widehat M_{i,a}\). The critical
graph has \(|\widehat\Phi|\)-area zero, so
\[
  \int_{X_g}|\widehat\Phi|
  =
  \sum_{i=1}^{r}\sum_{a=1}^{r_i}
  \frac{1}{\widehat M_{i,a}}.
\]
Together with
\(\operatorname{Ext}_{X_g}(u_0^*\Delta_q)=
\int_{X_g}|\widehat\Phi|\), this proves
\eqref{eq:covering-map-tree-energy-cylinder-moduli}.

Finally, assume that \(X_g=u_0^*X_0\). Then
\(u_0:X_g\to X_0\) is a holomorphic covering. Let
\(\Phi_{\Delta_q,X_0}\) be the Hubbard--Masur differential on \(X_0\)
whose vertical measured foliation is \(\Delta_q\). The pullback
\(u_0^*\Phi_{\Delta_q,X_0}\) is a holomorphic quadratic differential on
\(X_g\), and its vertical measured foliation is \(u_0^*\Delta_q\).
Uniqueness in the Hubbard--Masur theorem
\cite[Chapter~I, \S2, Main Theorem, p.~224]
{HubbardMasur1979}
therefore gives
\[
  \Phi_{u_0^*\Delta_q,X_g}
  =
  u_0^*\Phi_{\Delta_q,X_0}.
\]
Since \(u_0\) has degree \(d\), integration of the quadratic-differential
area form gives
\[
  \operatorname{Ext}_{X_g}(u_0^*\Delta_q)
  =
  \int_{X_g}
  \left|u_0^*\Phi_{\Delta_q,X_0}\right|
  =
  d\int_{X_0}
  \left|\Phi_{\Delta_q,X_0}\right|
  =
  d\,\operatorname{Ext}_{X_0}(\Delta_q).
\]
Combining this identity with
\eqref{eq:covering-map-tree-energy-extremal-length} proves
\eqref{eq:covering-map-tree-energy-degree}.
\end{proof}

\section{Variations of energy functions}
\label{Variations of energy functions}

In this section, we will calculate the variations of energy functions and show the essential difficulties that prevent the energy function from being convex. 

\subsection{Variations of energy functions with fixed domain}
\label{subsec:variations-energy-functions}

We first record a general variational formula along
a general path; compare
\cite{Yamada1999,KimWanZhang2022}. Let
$x_t=[X_t,m_t]$ be a smooth path in $\mathcal T(S)$ and write
$\sigma_t:=\sigma_{X_t}$. Choose a smooth family
of marking-compatible identifications $\iota_t\colon S\to X_t$ and pull the
target metrics back to the fixed smooth surface $S$,
\(
  h_t:=\iota_t^*\sigma_t.
\)
After composing with $\iota_t^{-1}$, the harmonic maps may be viewed as a
smooth family
\(
  u_t\colon (M,g)\to(S,h_t)
\)
in the fixed homotopy class of $u_0$. Set
\[
  h:=h_0,
  \qquad
  u:=u_t|_{t=0},
  \qquad
  k:=\dot h_0,
  \qquad
  l:=\ddot h_0,
\]
and let
\[
  W:=\left.\frac{D u_t}{dt}\right|_{t=0}
  \in\Gamma(u^*TS)
\]
be the variational vector field.

For $V,Z\in\Gamma(u^*TS)$, 
the index form is defined by
\begin{align*}
\begin{split}
  \mathcal I_u(V,Z)
  &:=
  \int_M
  [
    \langle\nabla V,\nabla Z\rangle
    -
    \sum_{i=1}^n
    \bigl\langle
      R^h(V,du(e_i))du(e_i),Z
    \bigr\rangle
  ]d\mu_g. 
  \end{split}
\end{align*}
In particular, 
\begin{equation}
 \mathcal{I}_u(V,V)= \int_M
[
    |\nabla V|^2
    +
    \sum_{i=1}^n
    \left(
      |V|^2|du(e_i)|^2
      -
      \langle V,du(e_i)\rangle^2
    \right)
]d\mu_g
  \ge 0 \label{eq:index-form-negative-curvature}
\end{equation}
where $\{e_i\}_{i=1}^n$ is a local $g$-orthonormal frame, and last equality holds since $h$ has
constant curvature $-1$.

Under the noncyclicity assumption above, equality holds only for $V=0$.
The equality case is analyzed in
Lemma~\ref{lem:nondegeneracy-target-jacobi} below.

The corresponding Jacobi operator is
\begin{equation}
  \mathcal J_uV
  :=
  \nabla^*\nabla V
  -
  \sum_{i=1}^n
  R^h(V,du(e_i))du(e_i),
  \label{eq:def-positive-jacobi-operator}
\end{equation}
so that
\(
  \mathcal I_u(V,Z)
  =
  \langle\mathcal J_uV,Z\rangle_{L^2}.
\)

\begin{lemma}
\label{lem:nondegeneracy-target-jacobi}
Under the standing noncyclicity assumption, the index form
\(\mathcal I_u\) is positive definite. Equivalently, the Jacobi operator
\(\mathcal J_u\) has trivial kernel.
\end{lemma}

\begin{proof}
Suppose that \(\mathcal I_u(V,V)=0\). The pointwise formula
\eqref{eq:index-form-negative-curvature} gives
\[
  \nabla V=0,
  \qquad
  |V|^2|du(e_i)|^2-\langle V,du(e_i)\rangle^2=0
  \qquad\text{for all }i.
\]
Thus \(du(TM)\subset\mathbb RV\). Since \(M\) is connected and \(V\) is
parallel, either \(V\equiv0\), or \(V\) has constant positive length. Assume
the latter. Let
\(
  \widetilde u\colon\widetilde M\to\mathbb H^2
\)
be an equivariant lift of \(u\), and let \(\widetilde V\) be the lift of
\(V\). Choose \(\widetilde x_0\in\widetilde M\), and let
\(L\subset\mathbb H^2\) be the complete geodesic through
\(\widetilde u(\widetilde x_0)\) tangent to
\(\widetilde V(\widetilde x_0)\).

Let \(c\) be any piecewise smooth path in \(\widetilde M\) starting at
\(\widetilde x_0\). Along \(c\), one has
\(
  d\widetilde u(\dot c)=a\,\widetilde V,
  \nabla_{\dot c}\widetilde V=0
\)
for some scalar function \(a\). Consequently,
\(\widetilde u\circ c\) is a possibly nonregular reparametrization of a
geodesic with initial tangent line
\(\mathbb R\widetilde V(\widetilde x_0)\), and hence is contained in \(L\).
Connectedness gives \(\widetilde u(\widetilde M)\subset L\).

If \(\widetilde u\) is constant, the induced subgroup of \(\pi_1(S)\) is
trivial. Otherwise, equivariance implies that \((u_0)_*\pi_1(M)\) preserves
\(L\). A torsion-free discrete subgroup of
\(\operatorname{Isom}^+(\mathbb H^2)\) preserving a geodesic is cyclic.
Both alternatives contradict the standing assumption. Therefore \(V=0\).
\end{proof}
In particular, the Jacobi operator $J_u$ is invertible under our standing
assumption. As an unbounded self-adjoint non-negative elliptic operator on
\(L^2(u^*TS)\), with domain \(W^{2,2}(u^*TS)\), it has discrete spectrum.
Lemma~\ref{lem:nondegeneracy-target-jacobi} excludes zero from the spectrum,
and hence \(\mathcal J_u\) has a bounded inverse
\(
  \mathcal J_u^{-1}\colon
  L^2(u^*TS)\to W^{2,2}(u^*TS).
\)

For the nonlinear implicit-function argument, we work in H\"older spaces.
Fix \(0<\alpha<1\). Let \(\operatorname{Met}_{-1}(S)\) denote the space
of smooth hyperbolic metrics on \(S\), and let
\(
  \pi:
  \operatorname{Met}_{-1}(S)
  \to
  \mathcal T(S)
  =
  \operatorname{Met}_{-1}(S)/\operatorname{Diff}_0(S)
\)
be the natural quotient map. Choose an open neighborhood
\(\mathcal U\subset\mathcal T(S)\) of \([h]\) and a smooth local section
\(\sigma:\mathcal U\to\operatorname{Met}_{-1}(S)\) satisfying
\(\sigma([h])=h\). Set
\(
  \mathscr H:=\sigma(\mathcal U).
\)
Thus \(\mathscr H\) is a smooth finite-dimensional submanifold of
\(\operatorname{Met}_{-1}(S)\) through \(h\), and the restriction
\(\pi|_{\mathscr H}:\mathscr H\to\mathcal U\) is a diffeomorphism.

For \(\varepsilon>0\) sufficiently small, let
\(
  \mathcal B_\varepsilon
  :=
  \left\{
    V\in C^{2,\alpha}(u^*TS):
    \|V\|_{C^{2,\alpha}}<\varepsilon
  \right\},
\)
and use the exponential map of the fixed reference metric \(h\) to define
\[
  \operatorname{Exp}_u(V)(x)
  :=
  \exp^h_{u(x)}\bigl(V(x)\bigr),
  \qquad x\in M.
\]
For \(V\in\mathcal B_\varepsilon\), let
\(
  P_V(x):
  T_{\operatorname{Exp}_u(V)(x)}S
  \to
  T_{u(x)}S
\)
be parallel transport with respect to \(h\) along the geodesic
\(s\mapsto\exp^h_{u(x)}(sV(x))\), \(0\leq s\leq1\). The tension field then
defines a smooth map between fixed Banach spaces,
\[
  \mathscr T:
  \mathscr H\times\mathcal B_\varepsilon
  \longrightarrow
  C^{0,\alpha}(u^*TS),
  \qquad
  \mathscr T(h',V)
  :=
  P_V\!\left(
    \tau_{g,h'}\bigl(\operatorname{Exp}_u(V)\bigr)
  \right).
\]
Here
\(\tau_{g,h'}(\operatorname{Exp}_u(V))\) is naturally a section of
\(\operatorname{Exp}_u(V)^*TS\), and \(P_V\) identifies this varying bundle
with the fixed bundle \(u^*TS\). Since \(u:(M,g)\to(S,h)\) is harmonic,
\(
  \mathscr T(h,0)=0.
\)
Moreover, 
\(D_V\mathscr T(h,0)=-\mathcal J_u\). Elliptic Fredholm theory and
Lemma~\ref{lem:nondegeneracy-target-jacobi} show that
\(
  \mathcal J_u\colon C^{2,\alpha}(u^*TS)
  \to C^{0,\alpha}(u^*TS)
\)
is an isomorphism. The Banach-space implicit function theorem, followed by
parameter-dependent elliptic regularity, therefore gives a smooth family $(u_t)$ of
smooth harmonic maps $u_t$. 

To state the next result  and let
\(
  C_k(X,Y)
  :=
  \left.
  \frac{d}{dt}
  \right|_{t=0}
  \nabla_X^tY.
\)
Then \(C_k\) is a tensor, symmetric in $(X, Y)$ and taking value
in $TS$. By a routine computation we find $C_k$ can also be defined as
\begin{equation}
\begin{aligned}
  2h\bigl(C_k(X,Y),Z\bigr)
  ={}&
  (\nabla_Xk)(Y,Z)
  +
  (\nabla_Yk)(X,Z)
  -
  (\nabla_Zk)(X,Y),
\end{aligned}
\label{eq:variation-levi-civita-connection}
\end{equation}
where \(\nabla=\nabla^0\).
We denote $S_k$ as the tensor field
\begin{equation}
  S_k
  :=\text{tr}_h (u^\ast C_k)=
  \sum_{i=1}^n
  C_k\bigl(du(e_i),
  du(e_i)\bigr)
  \label{eq:def-source-linearized-harmonic-map}
\end{equation}
where $(e_i)$ is a \(g\)-orthonormal frame.

\begin{proposition}
\label{prop:general-energy-variation}
Let
\(
  E(t):=\mathcal E_{u_0}(x_t).
\)
Then
\begin{equation}
  E'(0)
  =
  \frac12\int_M
  \operatorname{tr}_g\bigl(u^*k\bigr)\,d\mu_g,
  \label{eq:first-variation-energy-general}
\end{equation}
and
\begin{equation}
  E''(0)
  =
  \frac12\int_M
  \operatorname{tr}_g\bigl(u^*l\bigr)\,d\mu_g
  -
  \mathcal I_u(W,W).
  \label{eq:second-variation-energy-general}
\end{equation}
Then \(W\) is the unique solution of
\begin{equation}
  \mathcal J_uW=S_k.
  \label{eq:linearized-harmonic-map-equation}
\end{equation}
Consequently,
\begin{equation}
  E''(0)
  =
  \frac12\int_M
  \operatorname{tr}_g\bigl(u^*l\bigr)\,d\mu_g
  -
  \bigl\langle
    S_k,\mathcal J_u^{-1}S_k
  \bigr\rangle_{L^2}.
  \label{eq:second-variation-energy-jacobi-green}
\end{equation}
Here
\(
  \langle A,B\rangle_{L^2}
  :=
  \int_M h(A,B)\,d\mu_g
\)
for sections \(A,B\in\Gamma(u^*TS)\).
\end{proposition}

\begin{proof}
By the smooth-dependence result established above, the harmonic
representatives \(u_t:(M,g)\to(S,h_t)\) form a smooth family. For a map \(v:M\to S\), define
\(
  \mathcal F(t,v)
  :=
  \frac12\int_M
  \operatorname{tr}_g(v^*h_t)\,d\mu_g.
\)
Thus
\(
  E(t)=\mathcal F(t,u_t).
\)
Let \(\tau_t(v)=\tau_{g,h_t}(v)\) denote the tension field of \(v\) with
respect to the domain metric \(g\) and the target metric \(h_t\). If
\(V\in\Gamma(v^*TS)\), the first variation of the map energy is
\begin{equation}
  D_v\mathcal F(t,v)[V]
  =
  -\int_M
  h_t\bigl(\tau_t(v),V\bigr)\,d\mu_g.
  \label{eq:map-direction-first-variation}
\end{equation}
Since \(u_t\) is harmonic,
\(
  \tau_t(u_t)=0,
\)
and hence
\begin{equation}
  D_v\mathcal F(t,u_t)=0
  \label{eq:critical-point-family}
\end{equation}
for every \(t\).

We first compute the derivatives of \(\mathcal F\) in the parameter
direction while keeping the map fixed. If \(e_1,\ldots,e_n\) is a local
\(g\)-orthonormal frame, then
\(
  \operatorname{tr}_g(v^*h_t)
  =
  \sum_{i=1}^n
  h_t\bigl(dv(e_i),dv(e_i)\bigr).
\)
Therefore
\[
  \partial_t\mathcal F(t,v)
  =
  \frac12\int_M
  \sum_{i=1}^n
  \dot h_t\bigl(dv(e_i),dv(e_i)\bigr)\,d\mu_g.
\]
At \(t=0\), this gives
\begin{equation}
  \partial_t\mathcal F(0,u)
  =
  \frac12\int_M
  \operatorname{tr}_g(u^*k)\,d\mu_g.
  \label{eq:pure-t-first-derivative-energy}
\end{equation}
Similarly,
\begin{equation}
  \partial_{tt}^2\mathcal F(0,u)
  =
  \frac12\int_M
  \operatorname{tr}_g(u^*l)\,d\mu_g.
  \label{eq:pure-t-second-derivative-energy}
\end{equation}

The first-order chain rule gives
\(
  E'(0)
  =
  \partial_t\mathcal F(0,u)
  +
  D_v\mathcal F(0,u)[W].
\)
The second term vanishes by
\eqref{eq:critical-point-family}. Combining this with
\eqref{eq:pure-t-first-derivative-energy} proves
\eqref{eq:first-variation-energy-general}.

We next identify the Hessian of the map energy. With the sign convention
used here, the linearization of the tension field at the harmonic map \(u\)
is
\begin{equation}
  D_u\tau_0[V]
  =
  -\mathcal J_uV,
  \label{eq:linearization-tension-jacobi}
\end{equation}
where $J_u$ is the Jacobi operator
(\ref{eq:def-positive-jacobi-operator}).
Differentiating \eqref{eq:map-direction-first-variation} in the map
direction at the harmonic map \(u\), and using \(\tau_0(u)=0\), gives
\[
  D_v^2\mathcal F(0,u)[V,Z]
  =
  -\int_M
  h\bigl(D_u\tau_0[V],Z\bigr)\,d\mu_g
  =
  \int_M
  h\bigl(\mathcal J_uV,Z\bigr)\,d\mu_g
 =
  \mathcal I_u(V,Z).
\]
Thus
\begin{equation}
  D_v^2\mathcal F(0,u)
  =
  \mathcal I_u.
  \label{eq:energy-hessian-index-form}
\end{equation}

We now differentiate \(E(t)=\mathcal F(t,u_t)\) twice. The second-order chain rule gives
\begin{equation}
\begin{aligned}
  E''(0)
  ={}&
  \partial_{tt}^2\mathcal F(0,u)
  +
  2D_{tv}^2\mathcal F(0,u)[W]
  \\
  &+
  D_v^2\mathcal F(0,u)[W,W]
  +
  D_v\mathcal F(0,u)[A],
\end{aligned}
\label{eq:full-second-chain-rule-general}
\end{equation}
where \(A\) denotes the second derivative of the curve \(t\mapsto u_t\) in
the chosen chart. The last term vanishes because
\(D_v\mathcal F(0,u)=0\).

To eliminate the mixed derivative, differentiate the critical-point
identity \eqref{eq:critical-point-family}. More precisely, let
\(Z\in\Gamma(u^*TS)\), and extend \(Z\) to a smooth family of tangent
vectors \(Z_t\in T_{u_t}\operatorname{Map}(M,S)\) using the chosen local
trivialization of the tangent bundle of the mapping space. Differentiating
\(
  D_v\mathcal F(t,u_t)[Z_t]=0
\)
at \(t=0\) gives
\[
  D_{tv}^2\mathcal F(0,u)[Z]
  +
  D_v^2\mathcal F(0,u)[W,Z]
  +
  D_v\mathcal F(0,u)
 [
    \tfrac{D}{dt}|_{t=0}Z_t
 ]
  =
  0.
\]
Again, the last term vanishes because \(u\) is a critical point. Therefore,
using \eqref{eq:energy-hessian-index-form},
\(
  D_{tv}^2\mathcal F(0,u)[Z]
  +
  \mathcal I_u(W,Z)
  =
  0\)
for every $Z\in\Gamma(u^*TS).$
Taking \(Z=W\), we obtain
\[
  D_{tv}^2\mathcal F(0,u)[W]
  =
  -\mathcal I_u(W,W).
\]
Substituting this identity and
\eqref{eq:energy-hessian-index-form} into
\eqref{eq:full-second-chain-rule-general}, we find
\[
\begin{aligned}
  E''(0)
  &=
  \partial_{tt}^2\mathcal F(0,u)
  -
  2\mathcal I_u(W,W)
  +
  \mathcal I_u(W,W)
  \\
  &=
  \partial_{tt}^2\mathcal F(0,u)
  -
  \mathcal I_u(W,W),
\end{aligned}
\]
and
then s
\eqref{eq:second-variation-energy-general}.
by \eqref{eq:pure-t-second-derivative-energy}.

It remains to identify the equation satisfied by \(W\). 

We now differentiate the harmonic-map equation
\(
  \tau_t(u_t)=0.
\)
At a fixed point of \(M\), choose a local \(g\)-orthonormal frame
\(e_1,\ldots,e_n\) satisfying
\(\nabla^g_{e_i}e_i=0\) at that point. For the fixed map \(u\), one has
\(
  \tau_t(u)
  =
  \sum_{i=1}^n
  \nabla_{du(e_i)}^tdu(e_i)
\)
at that point. Hence
\[
  \left.
  \frac{d}{dt}
  \right|_{t=0}
  \tau_t(u)
  =
  \sum_{i=1}^n
  C_k\bigl(du(e_i),du(e_i)\bigr)
  =
  S_k.
\]

The variation of the tension field caused by the variation of the map is,
by \eqref{eq:linearization-tension-jacobi},
\(
  D_u\tau_0[W]
  =
  -\mathcal J_uW.
\)
Consequently,
\[
  0
  =
  \left.
  \frac{d}{dt}
  \right|_{t=0}
  \tau_t(u_t)
  =
  \left.
  \frac{d}{dt}
  \right|_{t=0}
  \tau_t(u)
  +
  D_u\tau_0[W]
  =
  S_k-\mathcal J_uW.
\]
This proves
\(
  \mathcal J_uW=S_k.
\)

Equivalently, the equation holds weakly in the form
\(
  \mathcal I_u(W,Z)
  =
  \langle S_k,Z\rangle_{L^2}
\)
for every $Z\in\Gamma(u^*TS).$
Elliptic regularity then shows that the weak solution is smooth. By
Lemma~\ref{lem:nondegeneracy-target-jacobi}, the Jacobi operator
\(\mathcal J_u\) is invertible, so the solution is unique and
\(
  W=\mathcal J_u^{-1}S_k.
\)
Finally,
\(
  \mathcal I_u(W,W)
  =
  \langle\mathcal J_uW,W\rangle_{L^2}
  =
  \langle S_k,\mathcal J_u^{-1}S_k\rangle_{L^2}.
\)
Substituting this identity into
\eqref{eq:second-variation-energy-general} proves
\eqref{eq:second-variation-energy-jacobi-green}.
\end{proof}

\begin{remark}
The two terms in \eqref{eq:second-variation-energy-general} have a useful
interpretation: the first is the direct second-order change of the target
metric, whereas the nonnegative index term is the amount by which the
harmonic representative lowers the energy by readjusting inside its
homotopy class. Although $k$, $l$, and $W$ depend on the chosen path, their
combination in \eqref{eq:first-variation-energy-general} and
\eqref{eq:second-variation-energy-general} is intrinsic.
\end{remark}
  
We now specialize to the unit-speed Teichm\"uller geodesic
$\gamma_q(t)=[X_t,F_t\circ f]$ introduced in
Section~\ref{sec:teichmuller-space-and-geodesics}, and write
\(
  E_q(t):=\mathcal E_{u_0}(\gamma_q(t)).
\)
Let $u\colon(M,g)\to(X_0,\sigma_0)$ be the harmonic representative at
$t=0$, and put
\begin{equation}
  e_u
  :=
  \frac12\operatorname{tr}_g(u^*\sigma_0).
  \label{eq:energy-density-u-at-zero}
\end{equation}

\noindent\emph{The family $\hat \sigma_t$
  along the harmonic maps $f_t$.}
Use the notation of Subsection~\ref{subsec:harmonic-map-gauge}. Thus,
\[
  \Phi(t)
  =
  t\Phi_1+\frac{t^2}{2}\Phi_2+O(t^3)
\]
is the Hopf differential of the harmonic marking
$f_t\colon(X_0,\sigma_0)\to(X_t,\sigma_t)$. Set
\begin{equation}
  U_1:=|\Phi_1|_{\sigma_0}^2,
  \qquad
  \eta_{\Phi_1}
  :=
  2U_1-4(\Delta_0-2)^{-1}U_1.
  \label{eq:def-eta-Phi1}
\end{equation}
By \eqref{eq:harmonic-gauge-second-order-expansion},
\[
  \dot{\widehat\sigma}_0=2\operatorname{Re}\Phi_1,
  \qquad
  \ddot{\widehat\sigma}_0
  =
  2\operatorname{Re}\Phi_2+\eta_{\Phi_1}\sigma_0.
\]
Hence Proposition~\ref{prop:general-energy-variation} gives the global
smooth formulas
\begin{align}
  E_q'(0)
  &=
  \int_M
  \operatorname{tr}_g
  \bigl(u^*\operatorname{Re}\Phi_1\bigr)\,d\mu_g,
  \label{eq:first-variation-energy-harmonic-gauge}
  \\
  E_q''(0)
  &=
  \int_M
  \operatorname{tr}_g
  \bigl(u^*\operatorname{Re}\Phi_2\bigr)\,d\mu_g
  +
  \int_M
  (\eta_{\Phi_1}\circ u)e_u\,d\mu_g
  -
  \mathcal I_u(W_{\mathrm H},W_{\mathrm H}),
  \label{eq:second-variation-energy-harmonic-gauge}
\end{align}
where $W_{\mathrm H}$ is the variation of the harmonic representatives in
the harmonic-map gauge. In the notation of
\eqref{eq:def-source-linearized-harmonic-map}, it is characterized by
\begin{equation}
  \mathcal J_uW_{\mathrm H}
  =
  S_{2\operatorname{Re}\Phi_1}.
  \label{eq:linearized-map-harmonic-gauge}
\end{equation}

\noindent\emph{The family $\title \sigma_t$
 along the Teichm\"u{}ller maps $f_t$.}

On $X_0\setminus Z(q)$, use the functions $a,b$ and the tensor $B_q$ from
Subsection~\ref{subsec:teichmuller-map-gauge}. Thus,
\begin{align*}
  (\Delta_0-2)a
  &=
  \mathcal D_q\log\lambda,
  \\
  (\Delta_0-2)b
  &=
  4\bigl(\mathcal D_qa+a^2-1\bigr),
\end{align*}
and
\[
  \dot{\widetilde\sigma}_0
  =
  2a\sigma_0+2B_q,
  \qquad
  \ddot{\widetilde\sigma}_0
  =
  (4+2b+4a^2)\sigma_0+8aB_q.
\]
Define the anisotropic energy density
\begin{equation}
  \beta_q(u)
  :=
  \frac12\operatorname{tr}_g(u^*B_q).
  \label{eq:def-anisotropic-energy-density}
\end{equation}
Assume that \(u(M)\cap Z(q)=\varnothing\). Since \(u(M)\) is compact,
choose open sets
\[
  u(M)\subset U_0\Subset U\Subset X_0\setminus Z(q).
\]
The Teichm\"uller maps and the pulled-back metrics are smooth on \(U\). By
isotopy extension, after shrinking the parameter interval, the restriction of
\(F_t\) near \(\overline U\) extends to a smooth global family of
marking-compatible identifications \(\iota_t:X_0\to X_t\) agreeing with
\(F_t\) on \(U\). The harmonic representatives pulled back by \(\iota_t\)
remain in \(U\) for small \(t\). Their first and second metric jets along
\(u(M)\) are therefore exactly
\eqref{eq:first-variation-teichmuller-gauge} and
\eqref{eq:second-variation-teichmuller-gauge}. Proposition~
\ref{prop:general-energy-variation} applies and gives
\begin{equation}
  E_q'(0)
  =
  2\int_M
  \left[
    (a\circ u)e_u+\beta_q(u)
  \right]d\mu_g,
  \label{eq:first-variation-energy-teichmuller-gauge}
\end{equation}
and
\begin{equation}
  \begin{aligned}
  E_q''(0)
  ={}&
  \int_M
  \left[
    \bigl[\bigl(4+2b+4a^2\bigr)\circ u\bigr]e_u
    +
    8(a\circ u)\beta_q(u)
  \right]d\mu_g-
  \mathcal I_u(W_{\mathrm T},W_{\mathrm T}),
  \end{aligned}
  \label{eq:second-variation-energy-teichmuller-gauge}
\end{equation}
where $W_{\mathrm T}$ is the map variation in the Teichm\"uller-map family
and satisfies
\begin{equation}
  \mathcal J_uW_{\mathrm T}
  =
  S_{2(a\sigma_0+B_q)}.
  \label{eq:linearized-map-teichmuller-gauge}
\end{equation}
More generally, the metric-jet identities and the corresponding linearized
equation are valid on \(M\setminus u^{-1}(Z(q))\). The global energy identities
\eqref{eq:first-variation-energy-teichmuller-gauge} and
\eqref{eq:second-variation-energy-teichmuller-gauge} are asserted here only
under the assumption \(u(M)\cap Z(q)=\varnothing\). When the image of \(u\)
meets \(Z(q)\), the harmonic-map formula
\eqref{eq:second-variation-energy-harmonic-gauge} is the appropriate global
formula.

It is useful to compare the preceding second-variation formula for
$E(t)$ along a Teichm\"uller geodesic with the corresponding formula
for the restriction $E_{\mathrm{WP}}(t)$ of the same energy function
to a Weil--Petersson geodesic. In the harmonic-map-family for a
Weil--Petersson geodesic, the transverse-traceless part of the metric
acceleration vanishes at the base point; equivalently, the coefficient
$\Phi_2$ in
\eqref{eq:second-variation-energy-harmonic-gauge} is zero. Hence the
second variation reduces to the positive scalar metric-acceleration
term and the negative Jacobi relaxation term. Under the standing
noncyclicity assumption, the estimates of
\cite{Yamada1999,KimWanZhang2022} show that the Jacobi term is
controlled by part of the scalar term, leaving a strictly positive
remainder. Consequently, $E_{\mathrm{WP}}''(0)>0$ in every nonzero
Weil--Petersson direction. Since the base point is arbitrary, the
energy is strictly convex along Weil--Petersson geodesics.

For a general Teichm\"uller geodesic, by contrast, the coefficient
$\Phi_2$ need not vanish, since such a geodesic need not be a straight
line in Wolf's harmonic-map coordinates; compare \cite{Wolf1989}.
Accordingly, the contribution
$\int_M \operatorname{tr}_g
\bigl(u^*\operatorname{Re}\Phi_2\bigr)\,d\mu_g$
has no a priori sign. The same obstruction is visible in the
Teichm\"uller-map-family: the direct metric-acceleration density
appearing in
\eqref{eq:second-variation-energy-teichmuller-gauge}
has no a priori pointwise sign. Thus the Weil--Petersson convexity
argument does not carry over directly to Teichm\"uller geodesics.
This motivates the global quasi-convexity theorem proved above.

\subsection{Variations of energy functions with fixed target}
\label{subsec:variations-energy-fixed-target}

Let $X\in\mathcal T(\Sigma)$ and let
\(
  \Phi_X:=\operatorname{Hopf}(u_X)\in\mathcal Q(X)
\)
be the Hopf differential of the harmonic map $u_X$. Our convention is
\(
  \Phi_X
  =
  h\bigl((u_X)_z,(u_X)_z\bigr)\,dz^2
\)
in a local conformal coordinate $z$ on $X$. If $X_s$ is a smooth path
through $X=X_0$ whose tangent vector is represented by a Beltrami differential
$\mu$, then the first variation of the energy is
\begin{equation}
  \left.\frac{d}{ds}\right|_{s=0}E_{u_0}(X_s)
  =
  -4\operatorname{Re}\int_X\Phi_X\mu.
  \label{eq:first-variation-domain-hopf}
\end{equation}
Here the integral is the canonical pairing between holomorphic quadratic
differentials and Beltrami differentials. This is the standard first-variation
formula in the normalization used here; see
\cite{Tromba1992}. In particular, a point $X$ is critical if
and only if $\Phi_X=0$, or equivalently, if and only if $u_X$ is weakly
conformal.

We next derive formulas adapted to a Teichm\"uller geodesic. Let
\(
  X_t=\gamma_q(t),
\) \(
  q\in\mathcal Q^1(X_0),
\)
be a unit-speed Teichm\"uller geodesic, and let
$F_t\colon X_0\to X_t$ be its Teichm\"uller maps. Set
\(
  v_t:=u_{X_t}\circ F_t\colon X_0\to S,
 \) \(
  E(t):=E_{u_0}(X_t).
\)
Away from the zero set of $q$, choose a natural coordinate
\(
  \zeta=x+iy,
\) \(
  q=d\zeta^2.
\)
For a map $v\colon X_0\to S$, define
\begin{equation}
  H(v)
  :=
  \frac12\int_{X_0\setminus Z(q)}|v_x|_h^2\,dx\,dy,
  \quad
  V(v)
  :=
  \frac12\int_{X_0\setminus Z(q)}|v_y|_h^2\,dx\,dy.
  \label{eq:horizontal-vertical-energy-domain}
\end{equation}
These quantities are globally well defined: the transition maps between
natural coordinates are of the form $\zeta\mapsto\pm\zeta+c$, and $Z(q)$ has
measure zero.

Let \(\sigma_t\) denote the hyperbolic metric on \(X_t\), and, for
\(w\in W^{1,2}(X_t,S)\), write
\(
  E_{X_t}(w)
  :=
  \frac12\int_{X_t}
  |dw|_{\sigma_t,h}^2\,dA_{\sigma_t}.
\)
For a smooth map \(v:X_0\to S\), define the energy functional pulled back
to the fixed domain \(X_0\) by
\begin{equation}
  \mathcal E(t,v)
  :=
  E_{X_t}(v\circ F_t^{-1})
  =
  e^{-2t}H(v)+e^{2t}V(v).
  \label{eq:fixed-map-energy-teich-geodesic}
\end{equation}
The first expression is well defined because \(F_t^{-1}:X_t\to X_0\) is
locally bi-Lipschitz across \(Z(q_t)\), and hence
\(v\circ F_t^{-1}\in W^{1,2}(X_t,S)\).

To verify the second equality, work first on \(X_0\setminus Z(q)\). In a
\(q\)-natural coordinate \(\zeta=x+iy\), with \(q=d\zeta^2\), and in the
corresponding \(q_t\)-natural coordinate on \(X_t\), one has
\(
  \zeta_t\circ F_t(x+iy)
  =
  e^t x+i e^{-t}y
\) and \(
  F_t^*|q_t|
  =
  e^{2t}dx^2+e^{-2t}dy^2.
\)
Since \(\sigma_t\) and \(|q_t|\) determine the same conformal structure on
\(X_t\setminus Z(q_t)\), conformal invariance of the two-dimensional
Dirichlet energy allows \(E_{X_t}\) to be computed using \(|q_t|\).
Changing variables by \(F_t\) then gives
\[
  E_{X_t}(v\circ F_t^{-1})
  =
  \frac12\int_{X_0\setminus Z(q)}
  \left(
    e^{-2t}|v_x|_h^2
    +
    e^{2t}|v_y|_h^2
  \right)\,dx\,dy,
\]
which is precisely
\(\mathcal E(t,v)=e^{-2t}H(v)+e^{2t}V(v)\).

Let \(m_t:\Sigma\to X_t\) be the marking and define
\(
  \mathscr H_t
  :=
  \{w:X_t\to S: w\circ m_t\simeq u_0\}.
\)
Thus \(\mathscr H_t\) is the homotopy class on \(X_t\) determined by
\(u_0\) and the marking \(m_t\). Let \(\mathscr H_0\) be the corresponding class on
\(X_0\). Since \(F_t\) is marking-compatible, precomposition with
\(F_t^{-1}\) defines a bijection
\(
  \mathscr H_0\to\mathscr H_t
\) and \(
  v\mapsto v\circ F_t^{-1}.
\)
In particular, if \(v_t:=u_{X_t}\circ F_t\), then
\(v_t\in\mathscr H_0\) and
\(v_t\circ F_t^{-1}=u_{X_t}\). Since \(u_{X_t}\) minimizes the Dirichlet
energy in \(\mathscr H_t\), for every \(v\in\mathscr H_0\) one has
\[
  \mathcal E(t,v_t)
  =
  E_{X_t}(u_{X_t})
  \le
  E_{X_t}(v\circ F_t^{-1})
  =
  \mathcal E(t,v).
\]
Thus \(v_t\) minimizes \(\mathcal E(t,\cdot)\) in the fixed homotopy class
\(\mathscr H_0\). Moreover, with
\(E(t):=E_{u_0}(X_t)\),
\(
  E(t)
  =
  E_{X_t}(u_{X_t})
  =
  \mathcal E(t,v_t).
\)

\begin{proposition}
\label{prop:domain-energy-first-second-variation}
Let $X_t=\gamma_q(t)$ be the Teichm\"uller geodesic determined by
$q\in\mathcal Q^1(X_0)$, let $F_t:X_0\to X_t$ be the associated
Teichm\"uller maps, and set $v_t:=u_{X_t}\circ F_t$ and
$E(t):=\mathcal E(t,v_t)$. Then
\begin{equation}
  E'(t)
  =
  -2e^{-2t}H(v_t)+2e^{2t}V(v_t).
  \label{eq:domain-energy-first-variation-teich}
\end{equation}
For every compact interval $J\subset\mathbb R$ and every $p>2$, one has
$t\mapsto v_t\in C^2(J,W^{1,p}(X_0,S))$. Let
$W_t:=D_tv_t\in W^{1,p}(X_0,v_t^*TS)$ and define
$\mathcal Q_t:=D_v^2\mathcal E(t,v_t)[W_t,W_t]$. Then
\begin{equation}
  E''(t)=4E(t)-\mathcal Q_t.
  \label{eq:domain-energy-second-variation-teich}
\end{equation}
Moreover, if
$\widetilde W_t:=W_t\circ F_t^{-1}$, then
$\widetilde W_t\in W^{1,p}(X_t,u_{X_t}^*TS)
\subset W^{1,2}(X_t,u_{X_t}^*TS)$ and
\begin{equation}
  \mathcal Q_t
  =
  I_{u_{X_t}}
  \bigl(
    \widetilde W_t,\widetilde W_t
  \bigr).
  \label{eq:brief-hessian-jacobi-identification}
\end{equation}
In particular, $\mathcal Q_t\ge0$.
\end{proposition}
\begin{proof}
We first justify the regularity across the finite zero set $Z(q)$. Near a
zero of $q$, the canonical orienting double cover \cite[Construction~1]{Lanneau2004} reduces the local
Teichm\"uller map to a homogeneous model of the form
$re^{i\theta}\mapsto r\Psi_t(\theta)$, where $\Psi_t$ and its relevant
$t$- and $\theta$-derivatives are uniformly bounded on compact parameter
intervals, and the angular derivative is uniformly nondegenerate. It
follows that $F_t$ and $F_t^{-1}$ are locally bi-Lipschitz across $Z(q)$,
and that their first two parameter derivatives have uniformly bounded weak
first derivatives.

Choose a smooth marking-compatible trivialization
$\psi_t:X_0\to X_t$ adapted to these local models and put
$G_t:=\psi_t^{-1}\circ F_t$. Combining the preceding estimates near
$Z(q)$ with the affine description away from $Z(q)$ gives
$t\mapsto G_t\in C^2(J,W^{1,p}(X_0,X_0))$ for every finite $p$.
On the other hand,
$\widehat u_t:=u_{X_t}\circ\psi_t$ solves a smooth family of harmonic-map
equations on the fixed surface $X_0$. Since the Jacobi operator has trivial
kernel in the present nonzero-degree homotopy class, the implicit function
theorem and parameter-dependent elliptic regularity give smooth dependence
of $\widehat u_t$ on $t$; see
\cite{Tromba1992,DaskalopoulosWentworth2007}. Since
$v_t=\widehat u_t\circ G_t$, the Sobolev chain rule gives
$t\mapsto v_t\in C^2(J,W^{1,p}(X_0,S))$ for every $p>2$.

We next explain the choice of Sobolev exponent. Fix a smooth background
metric on $X_0$ and let $P_{\mathrm H}$ and $P_{\mathrm V}$ be the
orthogonal projections onto the horizontal and vertical line fields of
$q$, defined almost everywhere. They are fixed $L^\infty$ coefficients,
and their values on $Z(q)$ are irrelevant. After an isometric embedding
$S\hookrightarrow\mathbb R^N$, the functionals $H$ and $V$ are
restrictions of continuous quadratic integral functionals in
$d(\iota\circ v)$ with coefficients $P_{\mathrm H}$ and
$P_{\mathrm V}$.
Because $p>2=\dim X_0$, the embedding
$W^{1,p}(X_0)\hookrightarrow C^0(X_0)$ allows the target exponential map
to define the standard Banach-manifold structure on
$W^{1,p}(X_0,S)$. In these charts, $H$ and $V$, and hence
$\mathcal E(t,v)=e^{-2t}H(v)+e^{2t}V(v)$, are of class $C^2$.
Thus the first- and second-order Banach-manifold chain rules apply.

Since $v_t$ is a critical point of $\mathcal E(t,\cdot)$,
$D_v\mathcal E(t,v_t)=0$. Therefore
\[
  E'(t)
  =
  \partial_t\mathcal E(t,v_t)
  =
  -2e^{-2t}H(v_t)+2e^{2t}V(v_t),
\]
which proves \eqref{eq:domain-energy-first-variation-teich}.

Differentiating the critical-point equation and evaluating it on $W_t$
gives
$D_{tv}^2\mathcal E(t,v_t)[W_t]+\mathcal Q_t=0$. Differentiating
$E'(t)=\partial_t\mathcal E(t,v_t)$ consequently yields
$E''(t)=\partial_{tt}^2\mathcal E(t,v_t)-\mathcal Q_t$. For every fixed
map $v$, one has
$\partial_{tt}^2\mathcal E(t,v)=4\mathcal E(t,v)$. Since
$\mathcal E(t,v_t)=E(t)$, this proves
\eqref{eq:domain-energy-second-variation-teich}.

It remains to identify the Hessian. Since $F_t$ and $F_t^{-1}$ are
bi-Lipschitz, the map $Z\mapsto Z\circ F_t^{-1}$ is a bounded
isomorphism on both the $W^{1,p}$ and the $W^{1,2}$ spaces of sections.
Here $W^{1,p}$ is the space on which the Hessian
$D_v^2\mathcal E(t,v_t)$ is defined, whereas $W^{1,2}$ is the natural
space on which the index form is continuous.

For $Z\in W^{1,p}(X_0,v_t^*TS)$, set
$\widetilde Z:=Z\circ F_t^{-1}$. Approximate $\widetilde Z$ strongly in
$W^{1,p}$ by smooth sections $\widetilde Z_\nu$ on $X_t$, and put
$Z_\nu:=\widetilde Z_\nu\circ F_t$. The classical second-variation
formula on the closed surface $X_t$ gives
\[
  D_v^2\mathcal E(t,v_t)[Z_\nu,Z_\nu]
  =
  I_{u_{X_t}}(\widetilde Z_\nu,\widetilde Z_\nu).
\]
The left-hand side converges by the $W^{1,p}$-continuity of the Hessian.
Since $p>2$ and $X_t$ is compact, $W^{1,p}\hookrightarrow W^{1,2}$, so
the right-hand side converges by the $W^{1,2}$-continuity of the index
form. Taking $Z=W_t$ proves
\eqref{eq:brief-hessian-jacobi-identification}.

Finally, since the target metric $h$ has curvature $-1$, one has
\[
\begin{aligned}
  \mathcal Q_t
  =
  \int_{X_t}
  \biggl[
    |\nabla\widetilde W_t|_h^2
    +
    \sum_{i=1}^2
    \left(
      |\widetilde W_t|_h^2|du_{X_t}(e_i)|_h^2
      -
      \bigl\langle
        \widetilde W_t,du_{X_t}(e_i)
      \bigr\rangle_h^2
    \right)
  \biggr]
  dA_{\sigma_{X_t}}
  \ge0.
\end{aligned}
\]
This identification is obtained globally on the closed surface $X_t$ and
does not use integration by parts on $X_0\setminus Z(q)$. Hence no term
supported on $Z(q)$ and no puncture-boundary contribution occurs.
\end{proof}

The minimizing characterization also gives the following global Lipschitz
estimate for the logarithmic energy.

\begin{corollary}
\label{cor:log-energy-two-lipschitz}
Along every unit-speed Teichm\"uller geodesic,
\begin{equation}
  |E'(t)|\le2E(t),
  \qquad
  \left|\frac{d}{dt}\log E(t)\right|\le2.
  \label{eq:log-energy-two-lipschitz}
\end{equation}
\end{corollary}

\begin{proof}
For every map \(v\) in the fixed homotopy class and all \(s,t\in\mathbb R\),
equation~\eqref{eq:fixed-map-energy-teich-geodesic} gives
\[
  e^{-2|t-s|}\mathcal E(s,v)
  \le
  \mathcal E(t,v)
  \le
  e^{2|t-s|}\mathcal E(s,v).
\]
Taking the infimum over \(v\) yields
\[
  e^{-2|t-s|}E(s)
  \le E(t)\le
  e^{2|t-s|}E(s).
\]
Therefore \(|\log E(t)-\log E(s)|\le2|t-s|\). Since the energy is smooth on
Teichm\"uller space, differentiation gives
\eqref{eq:log-energy-two-lipschitz}.
\end{proof}

\begin{remark}
\label{rem:general-target-obstruction}
The logarithmic Lipschitz estimate in
Corollary~\ref{cor:log-energy-two-lipschitz} uses only the conformal invariance
of two-dimensional energy and the minimizing property of the harmonic
representative. It therefore has an analog for maps from a varying Riemann
surface to a fixed nonpositively curved Riemannian manifold. Moreover, if the target has
non-positive Hermitian sectional
curvature, then the energy function enjoys strong complex-analytic positivity
properties on Teichm\"uller space: Toledo proved that the energy is
plurisubharmonic, and Kim--Wan--Zhang proved that the reciprocal energy is
plurisuperharmonic and that both \(\log E\) and \(E\) are plurisubharmonic
\cite{Toledo2012,KimWanZhang2020}.  These results are very different in nature
from the coarse quasi-convexity proved above.  The proof of
Theorem~\ref{thm:covering-energy-quasi-convexity} uses two features special to
the hyperbolic-surface covering case: the factorization through a fixed
hyperbolic surface \(Y\), and the extremal-length estimate
\(
  \ell_Y(\lambda)^2
  \le
  2E_{u_0}(X)\operatorname{Ext}_X(\lambda),
\)
which can then be combined with Kerckhoff's formula.  For a general
Riemannian target there is no comparable fixed hyperbolic length function
\(\ell_Y(\lambda)\), no direct Kerckhoff-type formula, and no automatic
coarse comparison between \(E(X)\) and \(e^{2d_{\mathrm T}(X,Y)}\).  Thus the
present quasi-convexity theorem is
proved only for hyperbolic
surface targets; extending it to more general nonpositively curved targets
would require additional geometric or representation-theoretic input replacing
this two-sided comparison.
\end{remark}

In the covering case, the critical-point behavior is particularly rigid.

\begin{proposition}
\label{prop:covering-energy-critical-hessian}
Under Assumption~\ref{assump:-covering}, the point $Y$ is the unique
critical point of $E_{u_0}$. If
$\mu\in T_Y\mathcal T(\Sigma)$ is represented by a harmonic Beltrami
differential, then
\begin{equation}
  \operatorname{Hess}_Y E_{u_0}(\mu,\mu)
  =
  4\|\mu\|_{\mathrm{WP}}^2,
  \qquad
  \|\mu\|_{\mathrm{WP}}^2
  :=
  \int_Y|\mu|^2\,dA_Y.
  \label{eq:covering-energy-hessian-WP}
\end{equation}
In particular, the Hessian is positive definite at $Y$.
\end{proposition}

\begin{proof}
This is the covering-map case of the critical-point and Hessian formulas in
\cite{KimWanZhang2024}. Since $Y$ is a critical point, the second derivative
of $E_{u_0}$ along any $C^2$ path through $Y$ depends only on its initial
tangent vector. Thus the formula, computed there along a Weil--Petersson
geodesic, is the intrinsic Hessian formula
\eqref{eq:covering-energy-hessian-WP} and applies in particular to a
Teichm\"uller geodesic through $Y$.
\end{proof}

Proposition~\ref{prop:covering-energy-critical-hessian} is a local statement. It
does not imply convexity along an entire Teichm\"uller geodesic. Indeed, for
$k>0$, equation~\eqref{eq:domain-energy-second-variation-teich} gives
\begin{equation}
  (E^k)''
  =
  kE^{k-2}
  \left[
    4E^2-E\mathcal Q_t+(k-1)(E')^2
  \right].
  \label{eq:second-variation-energy-power}
\end{equation}
Thus a global convexity theorem for $E^k$ would require an additional estimate
controlling the Jacobi relaxation term $\mathcal Q_t$ in terms of $E$ and
$E'$. The variational identity alone supplies no such estimate. In
particular, at a stationary point of the one-variable function $E(t)$, the
term involving $k-1$ in \eqref{eq:second-variation-energy-power} vanishes, so
increasing $k$ cannot by itself correct a negative value of $E''(t)$.

At the unique critical point $Y$, however, every positive power is locally
strictly convex. More precisely, if $X_t$ is any path through $Y$ with nonzero
initial tangent $\mu$, then
\begin{equation}
  \left.\frac{d^2}{dt^2}\right|_{t=0}E_{u_0}(X_t)^k
  =
  4kA_Y^{k-1}\|\mu\|_{\mathrm{WP}}^2>0.
  \label{eq:power-energy-local-convexity-covering}
\end{equation}
These are Hessian statements at the critical point $Y$; they should not be
interpreted as global convexity statements along complete Teichm\"uller
geodesics.

\medskip

The two quasi-convexity theorems proved above show that, in the natural energy settings considered here, harmonic-map energy retains the coarse convexity behavior of hyperbolic length along Teichm\"uller geodesics. The variation formulas explain why this coarse statement is the appropriate global
substitute for genuine convexity.

\providecommand{\bysame}{\leavevmode\hbox to3em{\hrulefill}\thinspace}
\providecommand{\MR}{\relax\ifhmode\unskip\space\fi MR }
\providecommand{\MRhref}[2]{%
  \href{http://www.ams.org/mathscinet-getitem?mr=#1}{#2}
}
\providecommand{\href}[2]{#2}


\end{document}